\documentclass{article}

\usepackage[utf8]{inputenc}
\usepackage[T1]{fontenc}
\usepackage[english]{babel}
\usepackage{fullpage}
\usepackage{hyperref}
\usepackage{tikz}
\usetikzlibrary{patterns,decorations.pathmorphing,positioning,arrows,fit,decorations.pathreplacing,calc,shapes.arrows,decorations.markings,shapes.multipart}

\makeatletter
\renewcommand\section{\@startsection
	{section}
{1}
{0pt}
{-3.5ex plus -1ex minus -.2ex}
{2.3ex plus.2ex}
{\centering\normalfont\Large\scshape}}

\renewcommand\subsection{\@startsection
	{subsection}
{2}
{0pt}
{-3ex plus -1ex minus -.2ex}
{1ex plus.2ex}
{\normalfont\large\bfseries}}

\renewcommand\subsubsection{\@startsection
	{subsubsection}
{3}
{0pt}
{-1.5ex plus -1ex minus -.2ex}
{0.8ex plus .2ex}
{\normalfont\bfseries}}

\renewcommand\paragraph{\@startsection
	{paragraph}
{4}
{0em}
{-1.2ex plus -0.4ex minus -.2ex}
{0ex}
{\normalfont\bfseries }}

\makeatother

\usepackage{latexsym,amsfonts,amsmath,amssymb,amsthm,amscd}
\usepackage{thmtools}
\usepackage{thm-restate}

\usepackage{bbm}
\usepackage{yhmath}

\usepackage[capitalise]{cleveref}

\usepackage{subcaption}

\usepackage{enumerate}

\newcommand{\Z}{\mathbbm{Z}}
\newcommand{\N}{\mathbbm{N}}

\newcommand{\A}{\mathcal{A}}
\newcommand{\B}{\mathcal{B}}

\newcommand{\F}{\mathcal{F}}

\newcommand{\s}{\sigma}

\providecommand{\ceil}[1]{\lceil #1 \rceil}

\theoremstyle{plain}
\newtheorem{theorem}{Theorem}[section]
\newtheorem{proposition}[theorem]{Proposition}

\newtheorem{corollary}[theorem]{Corollary}
\newtheorem{lemma}[theorem]{Lemma}

\theoremstyle{definition}
\newtheorem*{definition}{Definition}

\theoremstyle{remark}
\newtheorem*{remark}{Remark}
\newtheorem{example}{Example}[section]

\newcounter{claimcount}[theorem]

\newcommand{\bprf}[1][Proof:]{\begin{list}{}    {\setlength{\leftmargin}{0.5em}
\setlength{\rightmargin}{0em}  \setlength{\listparindent}{1em}}   \item {\em
\hspace{-1em}  #1  }}
\newcommand{\eprf}{\end{list}}

\usepackage{marvosym,fancybox}

\title{Parametrization by Horizontal Constraints in the Study of Algorithmic Properties of $\Z^2$-Subshift of Finite Type}
\date{}

\author{Solène J. Esnay, Mathieu Sablik}

\newcommand{\define}{\emph}

\begin{document}
\maketitle

\begin{abstract}
The non-emptiness, called the Domino Problem, and the characterization of the possible entropies of $\Z^2$-subshifts of finite type are standard problems of symbolic dynamics. In this article we study these questions with horizontal constraints fixed beforehand as a parameter. We determine for which horizontal constraints the Domino Problem is undecidable and when all right-recursively enumerable numbers can be obtained as entropy, with two approaches: either the additional local rules added to the horizontal constraints can be of any shape, or they can only be vertical rules.
\end{abstract}

\section*{Introduction}

The Domino Problem is a classical decision problem introduced by H. Wang~\cite{wang}: given a finite set of tiles that are squares with colored edges, called Wang tiles, we ask if it is possible to tile the plane with translated copies of these tiles so that contiguous edges have the same color. This question is also central in symbolic dynamics. A $\Z^d$-subshift of finite type (SFT for short) is a set of colorings of $\Z^d$ by a finite alphabet, called configurations, and a finite set of patterns that are forbidden to appear in those configurations. The set of tilings obtained when we tile the plane with a Wang tile set is an example of two-dimensional SFT. In the setting of SFTs, the Domino Problem becomes: given a finite set of forbidden patterns $\mathcal{F}$, is the associated subshift of finite type, denoted $X_{\mathcal{F}}$, empty?

On SFTs over $\Z$, the Domino Problem is known to be decidable. On those over $\Z^2$, Wang conjectured that the Domino Problem was decidable too, and produced an algorithm of decision relying on the hypothetical fact that all subshifts of finite type contained some periodic configuration. However, his claim was disproved by Berger~\cite{dpberger} who proved that the Domino Problem over any $\Z^d, d\geq 2$ is algorithmically undecidable. The key of the proof is the existence of a $\Z^d$-subshift of finite type containing only aperiodic configurations, on which computations are encoded. In the decades that followed, many alternative proofs of this fact were provided~\cite{dprobinson,dpmozes,dpkari}.

The exact conditions to cross this frontier between decidability and undecidability have been intensively studied under different points of view during the last decade. To explore the difference of behavior between $\Z$ and $\Z^2$, the Domino Problem has been extended on discrete groups~\cite{CohenGoodmanS2015,Jeandel-2015,twoends,Aubrun-Barbieri-Jeandel-2019,Aubrun-Barbieri-Moutot-2019} and fractal structures~\cite{Barbieri-Sablik-2016-fractal} in view to determine which types of structures can implement computation. The frontier has also been studied by restraining complexity (number of patterns of a given size)~\cite{karimoutot} or bounding the difference between numbers of colors and tiles~\cite{undecidabilityconstraints}. Additional dynamical constraints were also considered, such as the block gluing property~\cite[Lemma 3.1]{pavlovschraudner} or minimality~\cite{Gangloff-Sablik-2018-SimMinimal}.

Another way to highlight the computational power of multidimensional SFTs is to consider the algorithmic complexity of the entropy. A famous result by M. Hochman and T. Meyerovitch~\cite{HM} states that the possible values of the entropy for SFTs of dimension $d\geq 2$ are exactly the non-negative right-recursively enumerable ($\Pi_1$-computable) numbers; whereas in dimension one they can only be the logarithm of Perron numbers. It is natural to try to determine, using different parameters, where this change in computational power happens: this can be achieved considering, for instance, SFTs indexed on discrete groups~\cite{Barbieri-2021}. Once again, dynamical constraints are another relevant parameter: under strong irreducibility constraints, such as being block gluing, the entropy becomes computable~\cite{pavlovschraudner}. It is possible to extend the notion of block gluingness by adding a gap function $f$ that yields the distance $f(n)$ which allows for the concatenation of two rectangular blocks of size $n$ of the language: depending on the asymptotic behavior of $f(n)$, the set of entropies can be either any $\Pi_1$-computable number or only some computable numbers~\cite{Gangloff-Hellouin-2019}. The exact frontier for this parametrization is only known for subshifts with decidable language~\cite{Gangloff-Sablik-2017-BlockGluing}. 

This decrease in computational complexity, on the entropy or the Domino Problem, can be interpreted as a reduction of the computational power of the model as a whole under the added restriction.

In this article we study the algorithmic complexity of these two properties, the Domino Problem and the entropy, parametrized by local constraints. Formally, given a subshift of finite type defined by a set of forbidden patterns $\mathcal{H}$, we  want to characterize for which of these sets we have the following properties:
\begin{itemize}
\item the set $DP_h(\mathcal{H})=\{<\mathcal{F}>:X_{\mathcal{H}\cup\mathcal{F}}\ne\emptyset\}$ is undecidable;
\item $\left\{h(X_{\mathcal{H}\cup\mathcal{F}}):\mathcal{F} \textrm{ set of forbidden pattens}\right\}$ is the set of $\Pi_1$-computable numbers in $[0,h(X_{\mathcal{H}})]$.
\end{itemize}

Of course, any set of forbidden patterns which defines a SFT conjugated to $X_\mathcal{H}$ satisfies the same properties as $\mathcal{H}$; in other words these properties are invariant by conjugacy. Determining the possible values of $h(X_{\mathcal{H}\cup\mathcal{F}})$ for any set of forbidden patterns $\mathcal{F}$ comes down to knowing the possible entropies of all subsystems of $X_\mathcal{H}$ that are of finite type. A result by A. Desai~\cite{desai} proves that this set is dense in $[0,h(X_{\mathcal{H}})]$; we want to know when all the $\Pi_1$-computable numbers of that interval are obtained. 

This article focuses on when the parametrization is given by a set of constraints $\mathcal{H}$ that are specifically horizontal constraints, associated to a one-dimensional SFT $H$. This means that we impose the horizontal projective subaction of two-dimensional SFTs to be included in $H$. For this case we have a full characterization:
\begin{itemize}
\item the set $DP_h(\mathcal{H})$ (also denoted  $DP(H)$) is decidable if and only if $H$ contains only periodic points (Proposition~\ref{prop:DecPeriodicPoint});
\item $\left\{h(X_{\mathcal{H}\cup\mathcal{F}}):\mathcal{F} \textrm{ set of forbidden patterns}\right\}$ is the set of $\Pi_1$-computable numbers in $[0,h(H)]$ (Theorem~\ref{th:RealizationEntropy}).
\end{itemize}

A consequence of the second point is that given an alphabet $\A$, the set of possible entropies of two-dimensional SFTs on this alphabet is exactly the $\Pi_1$-computable numbers of $[0,\log(\A)]$. The result of M. Hochman and T. Meyerovitch does not answer this question because their construction can use an arbitrarily large number of letters. In particular, given a $\Pi_1$-computable number $h$, their construction gives a SFT of entropy $h$ on an alphabet where the cardinal is proportional to the number of states of a Turing machine which approaches $h$ from above, but it is known that it is possible to find numbers with arbitrarily high Kolmogorov complexity in any interval, so their construction needs a huge alphabet to obtain an algorithmically complex entropy. 
 
In the last Sections (Sections~\ref{sec:simulation},~\ref{sec:consequences} and~\ref{EntropyCombined}) we consider again the two aforementioned properties parametrized by horizontal constraints, but we force the additional local rules to be vertical (in the previous sections they could have arbitrary shape). This point of view has various motivations. First, to perform an efficient computer search of the smallest aperiodic Wang tile set (reached with 11 Wang tiles in~\cite{Jeandel-Rao-2015}), it is natural to eliminate horizontal constraints that necessarily yield periodic configurations. Second, a classical result is that every effective $\Z$-subshift can be realized as the projective subaction on $\Z$ of a $\Z^2$-subshift of finite type~\cite{hochman,Aubrun-Sablik-2010,Durand-Romashchenko-Shen-2012}; however, in these constructions, the dynamic on the other direction is trivial. We can ask  which other vertical dynamics can be compatible with a given horizontal dynamic. In this article we consider a easier question: given two one dimensional subshifts of finite type $\mathcal{H}$ and $\mathcal{V}$ we ask if there exist a two dimensional configuration where every horizontal lines are in  $\mathcal{H}$ and vertical lines in $\mathcal{V}$. This interplay helps us to understand a two-dimensional transfer of information using one-dimensional constraints, and ultimately how to transfer information in view to implement computation.  This point of view is an extension of a joint conference article with N. Aubrun \cite{DPH}.

Considering only vertical patterns is more complicated that considering arbitrary patterns, because we need to understand how to transfer information with few interplay between the two directions. In that case, in the present article, we restrict ourselves to horizontal constraints given by a one-dimensional nearest-neighbor SFT $H$. If this SFT satisfies a certain set of conditions, it is possible to simulate any two-dimensional SFT, as root of the one obtained adding carefully chosen vertical constraints to the horizontal constraints of $H$ (Theorem~\ref{th:root}). With this, we obtain a characterization of the nearest-neighbor one-dimensional SFTs which have undecidable Domino Problem with this parametrization (Theorem~\ref{th:DP}). For the entropy, we obtain a partial characterization; yet surprisingly we find horizontal subshifts $H$ that can only have a decidable Domino Problem when vertical constraints are added, but that realize two-dimensional SFTs with $\Pi_1$-computable numbers as entropies.

\section{Definitions}

Any dimension of patterns in $\Z^2$ that follows will be written under the format ``width $\times$ height''.

\subsection{Symbolic Dynamics}

For a given finite set $\A$ called the alphabet, $\A^{\Z^d}$ endowed with the product topology is called the \emph{$d$-dimensional full shift} over $\A$, and is a compact space. Any $x \in \A^{\Z^d}$, called a \emph{configuration}, can be seen as a function from $\Z^d$ to $\A$ and we write $x_{\vec{k}} := x(\vec{k})$.
For any $\vec{v} \in \Z^d$ define the \emph{shift map} $\sigma^{\vec{v}}: \A^{\Z^d} \rightarrow \A^{\Z^d}$ such that $\sigma^{\vec{v}}(x)_{\vec{k}} = x_{\vec{k}-\vec{v}}$. A \emph{pattern} $p$ is a finite configuration $p \in \mathcal{A}^{P_p}$ where $P_p \subset \Z^d$ is finite. We say that a pattern $p \in \mathcal{A}^{P_p}$ \emph{appears} in a configuration $x\in\A^{\Z^d}$ -- or that $x$ \emph{contains} $p$ -- if there exists $\vec{v} \in \Z^d$ such that for every $\vec{\ell} \in P_p$, $\sigma^{\vec{v}}(x)_{\vec{\ell}} = p_{\vec{\ell}}$. We denote it $p \sqsubset x$.

A \emph{$d$-dimensional subshift} associated to a set of patterns $\mathcal{F}$, called set of \emph{forbidden patterns}, is defined by
\[
X_{\mathcal{F}} = \{ x \in \mathcal{A}^{\Z^d} \mid \forall p \in \mathcal{F}, p \not\sqsubset x \}
\]
that is, $X_{\mathcal{F}}$ is the set of all configurations that do not contain any pattern from $\mathcal{F}$.
Note that there can be several sets of forbidden patterns defining the same subshift $X$. A subshift can equivalently be defined as a subset of $\A^{\Z^d}$ that is closed under both the topology and the shift map.
If $X=X_\F$ with $\mathcal{F}$ finite, then $X$ is called a \emph{Subshift of Finite Type}, SFT for short.
For a $d$-dimensional subshift $X$, the set of all patterns of size $n_1 \times n_2 \times \dots \times n_d$ (for a usually-ordered $\Z^d$-basis) that appear in configurations of $X$ is denoted by $\mathcal{L}_X(n_1,n_2,\dots,n_d)$, and its cardinality by $N_X(n_1,n_2,\dots,n_d)$. For a one-dimensional SFT $H$, we write $\mathcal{L}_H = \cup_{n} \mathcal{L}_H(n)$, which is a language.

For one-dimensional subshifts, we talk about \emph{nearest-neighbor} SFTs if  $\F\subset\A^{\{0,1\}}$. For two-dimensional subshifts, the most well-known are the \emph{Wang shifts}, defined by a finite number of squared tiles with colored edges that must be placed matching colors called Wang tiles. Formally, these tiles are quadruplets of symbols $(t_e,t_w,t_n,t_s)$. A Wang shift is described by a finite Wang tile set, and local rules $x(i,j)_e = x(i+1,j)_w$ and $x(i,j)_n = x(i,j+1)_s$ for all integers $i,j \in \Z$.

Two subshift $X$ and $Y$ are \emph{conjugate} if there exists a continuous function from $X$ to $Y$ which commutes with the shift. Two conjugate subshift have the same dynamical properties. For example, every two-dimensional SFT is conjugate to a Wang shift, though this changes the underlying local rules and alphabet. Notably, a two-dimensional SFT is empty if and only if the corresponding Wang shift is.

\subsection{One-dimensional SFTs as graphs}

As explained in \cite{LindMarcus}, the Rauzy graph of order $M$ of a one-dimensional SFT $H \subseteq {\A_H}^\Z$ denotes the following graph $\mathcal{G}_M(H) = (\mathcal{V}, \vec{E})$:
\begin{itemize}
	\item $\mathcal{V} = \mathcal{L}_H(M)$;
	\item $(u_1\dots u_M, u_2 \dots u_{M+1}) \in \vec{E}$ for $u_1,\dots,u_{M+1} \in \A_H$ if and only if $u_1 \dots u_{M+1} \in \mathcal{L}_{H}(M+1)$;
\end{itemize}
where all the stranded vertices (with in-degree or out-degree $1$) have been iteratively removed. The edge $(u_1\dots u_M, u_2 \dots u_{M+1})$ is additionally labeled $u_{M+1}$. Up to a renaming of the symbols, this graph is unique, no matter the forbidden patterns used to describe $H$.

Note that a Rauzy graph can be made of one or several \emph{strongly connected components (SCC for short)}. We recall that  it is constituted of one unique strongly connected components if and only if the SFT associated is \emph{transitive}; that is, for any $u, w \in \mathcal{L}_H$ there exists $v \in \mathcal{L}_H$ so that $uvw \in \mathcal{L}_H$. If the Rauzy graph has several SCCs it can also contain \emph{transient vertices}, that are vertices with no path from themselves to themselves. We refer to~\cite{LindMarcus} for more details.

\begin{example}
	The subshifts $X = X_{\{10,21,11,30,31,32,33\}} \subset \{0,1,2,3\}^\Z$ and $Y = Y_{\{10,21,11\}} \subset \{0,1,2\}^\Z$ are the same SFT. They have the same Rauzy graph made of two SCCs $\{0\}$ and $\{2\}$, and one transient vertex $1$ (vertex $3$ has been deleted from $\mathcal{G}(X)$ else it would be of out-degree 0).
\end{example}

This technique that algorithmically associates a graph to an SFT will be of great use, because it means that most proofs can focus on combinatorics over graphs to describe all one-dimensional SFTs.

In all that follows, if there is no further precision, \emph{the} Rauzy graph of a SFT is of order adapted to the size of its forbidden patterns. That is, for a one-dimensional SFT with forbidden patterns of size at most~$M+1$, $\mathcal{G}(H) := \mathcal{G}_M(H)$.

\subsection{Horizontal constraints}
\label{subsec:horizontalconstraints}

Adding a dimension to a one-dimensional SFT $H$, the local rules allow to define a two-dimensional SFT where each line is a configuration of $H$ chosen independently. We want to study the consequence of adding extra rules to that subshift: the most natural way of doing so is to add forbidden patterns -- which can be two-dimensional. This is formalized in the next definition.

\begin{definition}
	Let $H\subset\A^\Z$ be a one-dimensional SFT and $\F$ be a finite set of forbidden patterns. The two-dimensional SFT
	\[
	X_{H,\F} := \{ x \in \A^{\Z^2} \mid \forall j \in \Z, (x_{k,j})_{k \in \Z} \in H \text{ and } x \in X_\F \}
	\]
	is called the subshift $X_{\F}$ \emph{with added horizontal constraints from H}. 
\end{definition}

Another point of view is to search if two one-dimensional subshifts can be combined into a two-dimensional subshift where the first one appears horizontally and the second one vertically.
Given a one-dimensional SFT, the main focus of this article is to understand when it can  be compatible with another one to build a two-dimensional SFT.

\begin{definition}
	Let $H,V \subset \A^\Z$ be two one-dimensional SFTs. The two-dimensional subshift
	\[
	X_{H,V} := \{ x \in \A^{\Z^2} \mid \forall i, j \in \Z, (x_{k,j})_{k \in \Z} \in H \text{ and } (x_{i,\ell})_{\ell \in \Z} \in V \}
	\]
	
	is called the \emph{combined subshift} of $H$ and $V$, and uses $H$ as horizontal rules and $V$ as vertical rules.
\end{definition}

\begin{remark}
    The projection of the horizontal configurations that appear in $X_{H,V}$ does not necessarily recover all of $H$; indeed, some of the configurations in $H$ will not necessarily appear because they may not be legally extended vertically.
\end{remark}

\begin{example}
 	Choose $\A = \{0,1\}$, $H$ nearest-neighbor and forbidding $00$ and $11$, and $V$ forcing to alternate a $0$ and two $1$s: the resulting $X_{H,V}$ is empty, although neither $H$ nor $V$ are. In some sense, said $H$ and $V$ are incompatible.
\end{example}

\subsection{Root of a subshift}
\label{subsec:root}

Given an SFT $X$, we want to know if it can simulate in some sense any two-dimensional subshift by adding some local rules. Of course this simulation cannot be a conjugacy since we are limited by the topological entropy of $X$; therefore we introduce the notion of root of a subshift:

\begin{definition}
The subshift $X\subset \A^{\Z^2}$ is a \emph{$(m,n)$th root} of the subshift $Y\subset\B^{\Z^2}$ if there exist a clopen $Z\subset X$ with $\sigma^{(m,0)}(Z)=Z$ and $\sigma^{(0,n)}(Z)=Z$, and an homeomorphism $\varphi\colon Z\to Y$ such that:
\begin{itemize}
\item $\varphi(\s^{(k_1m,k_2n)}(x))=\s^{(k_1,k_2)}(\varphi(x))$ for all $x\in Z$;
\item $X = \bigsqcup_{0 \leq i < m, 0 \leq j < n} \sigma^{(i,j)}(Z)$.
\end{itemize}
\end{definition}

\subsection{Entropy}

Consider a $d$-dimensional subshift $X$. The \emph{topological entropy} of $X$ is
\[h(X)=\lim_{n\to\infty}\frac{\log_2(N_X(n,\dots,n))}{n^d}=\inf_{n}\frac{\log_2(N_X(n,\dots,n))}{n^d}.\]
It is a conjugacy invariant; we refer to~\cite{LindMarcus} for details.

%
%

\section{Domino problem under horizontal constraints}
\label{sec:subsystems}

\subsection{Theorem of simulation under horizontal constraints}

\begin{proposition}
	\label{rootsubsystem}
	Let $H\subset\A^\Z$ be a one-dimensional SFT which is not made solely of periodic points. For any two-dimensional SFT $Y\subset\B^{\Z^2}$, there exists a set of forbidden patterns $\F$ such that $X_{H,\F}$ is a $(m,n)$th root of $Y$ for some $m, n \in \Z$.
\end{proposition}

\begin{proof}
	Let $Y\subset\B^{\Z^2}$ be a two-dimensional SFT. Up to renaming the symbols, suppose the alphabet $\B$ is made of letters $T_1, \dots, T_n$.
	
	Suppose $H\subset\A^\Z$ is not solely made of periodic points. Consider its Rauzy graph, $\mathcal{G}(H)$. Fix some vertex $s$ of $\mathcal{G}(H)$ so that there are at least two paths from $s$ to itself -- one such vertex does exist, else $H$ would have only periodic points. Name those paths $\gamma_1$ and $\gamma_2$.
	
	Define, for any $k$ in $\{1,\dots,n\}$,
	\begin{align*}
		U_k & := \ell(\gamma_2 \gamma_1 \phi_k \gamma_1 \gamma_2)
	\end{align*}
	where $\phi_k$ is a succession of $n$ $\gamma_1$'s, except the $k$th one, replaced by $\gamma_2$; and where $\ell(\gamma)$ designates the succession of labels of edges in a path $\gamma$, consequently being an element of $\A^*$ (the set of finite words made of elements of $\A$).
	
	All of these words $U_k$ have the same length, call it $N$, and we can juxtapose them as desired by construction of the Rauzy graph. Moreover, juxtaposing two of these words creates two consecutive $\gamma_2$'s, allowing a clear segmentation of a word written with the $U_k$'s into these basic units.
	
	Now, consider the two-dimensional extension of $H$ into $X_{H,\F}$ with $\F$ forbidding the following patterns:
	\begin{itemize}
		\item anything horizontally that is not a succession of $U_k$'s (possibly with different $k$'s);
		\item any $\ell(\gamma_2 \gamma_2)$ being above something that is not another $\ell(\gamma_2 \gamma_2)$ (this forces the $U_k$'s to be vertically aligned);
		\item any pattern $p$ of size $aN \times b$ of $U_k$'s so that, when building a pattern $q$ with $T_l$ at position $(i,j), i \in \{0,\dots,a-1\}, j \in \{0,\dots,b-1\}$ if $p_{iN,j}$ belongs to $U_l$, yields a pattern $q$ that does not appear in $Y$.
	\end{itemize}
	This is a finite number of additional forbidden patterns, since $Y$ is itself a SFT.

	It is clear, considering the clopen $Z$ of the configurations in $X_{H,\F}$ that have a $U_k$ starting at $(0,0)$ exactly, and the homeomorphism $\varphi\colon Z \rightarrow Y$ that sends a $U_k$ on a $T_k$, that $X_{H,\F}$ is a $(N,1)$th root of $Y$.
\end{proof}

\subsection{The Domino Problem under horizontal constraints}\label{section.HDominoProblem}

Define $DP(\Z^d) = \{\langle \mathcal{F} \rangle \mid X_\mathcal{F} \text{ is a nonempty SFT}\}$ where $\langle \mathcal{F} \rangle$ is an encoding of the set of forbidden patterns $\mathcal{F}$ suitable for a Turing Machine. $DP(\Z^d)$ is a language called the \emph{Domino Problem} on $\Z^d$. As for any language, we can ask if it is algorithmically decidable, i.e. recognizable by a Turing Machine. Said otherwise, is it possible to find a Turing Machine that takes as input any \emph{finite} set of patterns $\mathcal{F} \subset \mathcal{A}^{\Z^d}$ of rules and answers \texttt{YES} if $X_{\mathcal{F}}$ contains at least one configuration, and \texttt{NO} if it is empty?

It is known that $DP(\Z)$ is decidable, because the problem can be reduced to the emptiness of nearest-neighbor one-dimensional SFTs, and finding a valid configuration in a nearest-neighbor SFT is equivalent to finding a biinfinite path -- hence a cycle -- in a finite oriented (Rauzy) graph. On the contrary, $DP(\Z^2)$ is undecidable~\cite{dpberger,dprobinson,dpmozes,dpkari}, and so is any $DP(\Z^d)$ for $d \geq 2$ by reduction.

Given $H$ a one-dimensional SFT, we consider the following \emph{Domino Problem under horizontal constraints}:

\[DP_h(H) = \{\langle \mathcal{F} \rangle \mid X_{H,\mathcal{F}} \text{ is a nonempty SFT}\}.\]

The purpose of its existence is to determine the frontier between decidability ($DP(\Z)$) and undecidability ($DP(\Z^2)$). 

\begin{remark}
	This Domino Problem is defined for a given $H$, and its decidability depends on such a $H$ chosen beforehand.
\end{remark}

\begin{proposition}\label{prop:DecPeriodicPoint}
$DP_h(H)$ is decidable if and only if $H$ has only periodic points.
\end{proposition}

\begin{proof}
If $H \subset \A^\Z$ has a nonperiodic point, then for any two-dimensional SFT $Y$ we can apply \cref{rootsubsystem} and build a two-dimensional SFT $X_{H,\F}$ so that it is a root of $Y$. Using the definition of a root of a SFT, it is clear that the emptiness of $X_{H,\F}$ is equivalent to the emptiness of $Y$. Consequently, we can reduce $DP_h(H)$ to $DP(\Z^2)$, and $DP_h(H)$ is undecidable.

If $H \subset \A^\Z$ only has periodic points, suppose these points' smallest common period is some integer $p > 0$. Now, take as input a finite set of forbidden patterns $\F$, all of them as a rectangle of size $pL \times M, L,M\in \N$ up to extending them.

If there is no rectangle of size $pL \times M(|\A|^{pLM}+1)$ respecting the local rules of $X_{H,\F}$, then that SFT is empty.
If there is at least one such rectangle, list all of these possible candidates. Consider $R$ a candidate: either it can be horizontally juxtaposed with itself without containing a pattern of $\F$, and we keep it; or it cannot, and we delete it. If all candidates are deleted, then $X_{H,\F}$ is empty, because $H$ forces a $p$-periodic repetition horizontally in any configuration, which happens to be incompatible with all candidate rectangles.

If at least one candidate $R$ remains, then by the pigeonhole principle it contains at least twice the same rectangle $R^\prime$ of size $pL \times M$. To simplify the writing, we assume that the rectangle that repeats is the one with coordinates $[1,pL] \times [1,M]$ inside $R$ where $[1,pL]$ and $[1,M]$ are intervals of integers, and that it can be found again with coordinates $[1,pL] \times [k,k+M-1]$. Else, we simply truncate a part of $R$ so that this becomes true.

Define $P := R|_{[1,pL] \times [1,k+M-1]}$.
Since $\F$ has forbidden patterns of size $pL \times M$, and since $R$ respects our local rules and begins and ends with $R^{\prime}$, $P$ can be vertically juxtaposed with itself (overlapping on $R^\prime$). Moreover, $P$ can be horizontally glued with itself. The result tiles $\Z^2$ periodically while respecting the constraints of $H$ and of $\F$, since any $pL \times M$ rectangle found in it is already present in the horizontal juxtaposition of $P$, which is valid by $p$-periodicity.
Therefore, $X_{H,\F}$ is nonempty.

With this, we have an algorithm to decide the emptiness of $X_{H,\F}$ for any input $\F$.
\end{proof}

%
%

\section{Characterization of the possible entropies under horizontal constraints}\label{section:EntropyHorizontalConstraint}

The purpose of this section is to characterize the entropies accessible to two-dimensional SFTs, as in~\cite{HM}, but under projective constraints. Formally, given $H$ a one-dimensional SFT, we want to characterize the set
 \[\left\{h(X_{H,\F}):\mathcal{F} \textrm{ set of forbidden pattens}\right\}.\]
 
Clearly it is a subset of $[0,h(H)]$ and by~\cite{desai} it is dense in that interval. The computability obstruction of~\cite{HM} says that it is a subset of the $\Pi_1$-computable numbers (also named \emph{right-recursively enumerable} or r.r.e. numbers). We recall that a real is $\Pi_1$-computable if there exists a decreasing computable sequence of rationals which converges toward that number. Therefore, it is natural to ask if all $\Pi_1$-computable numbers of  $[0,h(H)]$ can be obtained. We have the following result  proved in \cref{subsec:mainresult}.

\begin{restatable*}{theorem}{RealizationEntropy}
	\label{th:RealizationEntropy}
	Let $H \subset \A^\Z$ be a SFT. We have
	\[\left\{h(X_{H,\mathcal{F}}): \mathcal{F}\textrm{ set of forbiden patterns}\right\}=[0,h(H)]\cap\Pi_1\]
	where $\Pi_1$ is the set of right-recursively enumerable reals.
\end{restatable*}

\subsection{Kolmogorov complexity and number of tiles}

Given $r$ a $\Pi_1$-computable real, denote $K(r)$ its \define{Kolmogorov complexity}. It is the minimal number of states needed for a Turing Machine to enumerate a list of rationals which approach $r$ from above.
Consider the algorithm from \cite[Alg. 7.3]{HM}, defined for a given r.r.e. real $h$, that takes as input a sequence $(x_N)_{N \in \Z} \in \{0,1\}^\Z$, and makes its frequency of $1$'s -- on specific indices -- approach $r$ from above. The associated Turing Machine is built from the one that approaches $h$ from above so that its number of states is of the form $cK(h)$, $c>0$ not depending on $h$. The authors in \cite{HM} turn that Turing Machine into a Wang tile set. Since the size of that tile set depends linearly on the number of states of the Turing Machine, we obtain the following:

\begin{theorem}[From \cite{HM}, Alg. 7.3]
	\label{prop:algo}
	There exists $C > 0$ such that for any $h \in \Pi_1$, there exists a two-dimensional SFT $X$ such that $h(X) = h$, describable by a set of at most $CK(h)$ Wang tiles.
\end{theorem}

\subsection{Technical lemmas on entropy}

\begin{lemma}
	\label{lemma:entropy}
	For all $\alpha, \beta \in \N$, one has:
	\[\lim\limits_{n \to +\infty} \dfrac{\log_2(N_X(\alpha n, \beta n))}{\alpha \beta n^2} = h(X).\]
\end{lemma}

\begin{proof}
	We have the inequality
	\[N_X(\alpha n, \beta n) \leq (N_X(n, n))^{\alpha \beta}\]
	due to the fact that globally admissible patterns -- patterns that belong to a valid configuration -- of size $\alpha n \times \beta n$ are themselves made of a number $\alpha \beta$ of globally admissible patterns of size $n \times n$.
	
	We also have:
	\[
	(N_X(\alpha \beta n, \alpha \beta n)) \leq N_X(\alpha n, \beta n)^{\alpha \beta}
	\]
	with the same reasoning.
	
	Finally, we have
	\[\lim_n \dfrac{\log_2(N_X(\alpha \beta n, \alpha \beta n))}{\alpha^2 \beta^2 n^2} = h(X)\]
	because $(\frac{\log_2(N_X(\alpha \beta n, \alpha \beta n))}{\alpha^2 \beta^2n^2})$ is a subsequence of the converging sequence $(\frac{\log_2(N_X(n, n))}{n^2})$.
	
	Then applying $\lim_n \dfrac{\log_2(.)}{\alpha \beta n^2}$ to
	\begin{align*}
	(N_X(\alpha \beta n, \alpha \beta n))^{\frac{1}{\alpha \beta}} \leq N_X(\alpha n, \beta n) \leq (N_X(n, n))^{\alpha \beta}
	\end{align*}
	we obtain what we want.
\end{proof}

\begin{proposition}
	\label{rootentropy}
Let $X$ be a two-dimensional subshift which is a $(m,n)$th root of  the subshift $Y$. Then
\[h(Y)=mn\,h(X)\]
\end{proposition}
\begin{proof}
	Call $Z \subset X$ the clopen homeomorphic to $Y$. In what follows, we write $\mathcal{L}^{(k_1,k_2)}_{C}(a,b)$ the set of $[k_1,k_1+a-1] \times [k_2,k_2+b-1]$ patterns that appear in configurations of a clopen $C$; and $N^{(k_1,k_2)}_{C}(a,b)$ its cardinal. Note that $\mathcal{L}^{(0,0)}_{X}(a,b) = \mathcal{L}_{X}(a,b)$ since $X$ is shift-invariant; and same for $Y$.
	
Consider $\phi:Z\to Y$ the homeomorphism such that $\phi(\sigma^{(k_1m,k_2n)}(x))=\sigma^{(k_1,k_2)}\phi(x)$ for all $x\in Z$ and $(k_1,k_2)\in\Z^2$. By the same proof as Hedlund's theorem, there exist $r \in \N_0$ and a local map $\overline{\phi}$ applied on patterns of support $[-r,r]^2$ such that for any $x \in Z$, $\phi(x)_{(k_1,k_2)}=\overline{\phi}(x_{(k_1 m,k_2 n)+[-r,r]^2})$ -- that is, $\phi$ can be considered as the application of a local map uniformly on patterns of configurations of $Z$. Thus $|\mathcal{L}^{(0,0)}_Y(k, k)|\leq|\mathcal{L}^{(-r,-r)}_Z(mk+2r, nk+2r)|$. Furthermore, since a pattern of support $[-r,mk+r-1]\times[-r,nk+r-1]$ can be decomposed into a pattern of support $[0,mk-1]\times[0,nk-1]$ and its border, we deduce that $|\mathcal{L}^{(0,0)}_Y(k, k)|\leq|\A_Z|^{2rk(m+n)} |\mathcal{L}^{(0,0)}_Z(mk, nk)|$.

In the same way, there exist $r' \in \N_0$ and a local map $\overline{\phi^{-1}}$ applied on patterns of support $[-r',r']^2$ such that $\phi^{-1}(y)_{(k_1m,k_2n)+[0,m-1]\times[0,n-1]}=\overline{\phi^{-1}}(y_{(k_1,k_2 )+[-r',r']^2})$. Thus $|\mathcal{L}^{(0,0)}_Z(mk,n k)|\leq|\mathcal{L}^{(-r',-r')}_Y(k+2r', k+2r')|$

We have the same bounds for any $\sigma^{(i,j)}(Z)$ with $i,j\in\Z$. Moreover, since $X = \bigsqcup_{0 \leq i < m, 0 \leq j < n} \sigma^{(i,j)}(Z)$, one has $N_X(mk, nk) = \sum_{0 \leq i < m, 0 \leq j < n} N^{(0,0)}_{\sigma^{(i,j)}(Z)}(mk, nk)$. We deduce the following inequalities
\[
\frac{mn}{|\A_Z|^{2rk(m+n)}}N_Y(k,k)\leq N_X(mk, nk)\leq mn N_Y(k+2r', k+2r')
\]

	We apply $\frac{\log_2(.)}{k^2}$ to these inequalities and by using \cref{lemma:entropy}, we get:
	\[
	h(Y) = mn h(X).
	\]
	
\end{proof}

\begin{lemma}
	\label{lemlem}
	Let $H$ be a transitive one-dimensional SFT of positive entropy and $h < h(H)$.
	
	There exist words $u, w_1, w_2 \in \mathcal{L}_H$ and $\alpha \in \N$ such that:
	\begin{itemize}
		\item $\alpha > M$ the biggest size of forbidden patterns of $H$;
		\item $|w_1|=|w_2|=\alpha$;
		\item $u, w_1$ and $w_2$ are cycles from a vertex of the Rauzy graph of $H$ to the same vertex;
		\item $uww_1 \in \mathcal{L}_H$ for all $w \in \{w_1,w_2\}^*$;
		\item $u$ appears as a subword of any word in $u\{w_1,w_2\}^*w_1$ at the very beginning only;
		\item $h(H_u) > h$ where $H_u$ is the SFT included in $H$ where the word $u$ does not appear.
	\end{itemize}

	Moreover, the Rauzy graph of $H_u$ is still transitive.
\end{lemma}

\begin{proof}According to \cite[Th.3]{lind}, for any integer $n$ that is big enough, any $n$-long word $u$ is so that, if we denote $H_u$ the SFT where $u$ is added to the forbidden patterns of $H$, we get
	\[
	h(H) \geq h(H_u)>h
	\]
	because $h(H_u)$ is below yet sufficiently close to $h(H)$ for $n$ big enough.
	As a consequence, the last property we have to find is verified as soon as we choose $\alpha$ big enough for the three words we want.
	
	Now, consider the Rauzy graph of $H$, call it $\mathcal{G}(H)$. Fix some vertex $s$ of $\mathcal{G}(H)$. Since its entropy is positive and it is transitive, for any vertex $s$ of $\mathcal{G}(H)$ there are at least two paths from $s$ to itself that do not pass through $s$ at any other moment. Name those paths $\gamma_1$ and $\gamma_2$.
	
	Define
	\begin{align*}
	u & := \ell(\gamma_2 \gamma_1 \gamma_2 \gamma_1 \gamma_1 \gamma_1 \gamma_2 {\gamma_1}^k \gamma_2)\\
	w_1 & := \ell(\gamma_2 \gamma_1 \gamma_1 \gamma_2 \gamma_1 \gamma_1 \gamma_2 {\gamma_1}^k \gamma_2)\\
	w_2 & := \ell(\gamma_2 \gamma_1 \gamma_1 \gamma_1 \gamma_2 \gamma_1 \gamma_2 {\gamma_1}^k \gamma_2)\\
	\end{align*}
	where $\ell(\gamma)$ designates the succession of labels of edges in $\gamma$, consequently being an element of $\A^*$. Choose $k > 0$ big enough so that $u$ satisfies what we desire on $h(H_u)$. These three words have the same length and we can juxtapose them as desired by construction of the Rauzy graph of $H$. Moreover, juxtaposing any of those words creates two consecutive $\gamma_2$, allowing a clear segmentation of a word written with $u$, $w_1$ and $w_2$ into these basic units. Notably, $u$ appears as a subword of any word in $u\{w_1,w_2\}^*w_1$ at the very beginning only.
	
	Moreover, $H_u$ is still transitive. Indeed, let $v_1$ and $v_2$ be two elements of $\mathcal{L}_{H_u}$; we can suppose they are bigger than $u$ up to an extension of $v_1$ to its left and $v_2$ to its right.
	
	$v_1$ can be seen as a path in $\mathcal{G}(H)$ and extended to $v_1w$ so that the corresponding path in $\mathcal{G}(H)$ reaches $s$, by transitivity of that graph, with the shortest possible $w$.
	It is possible that $v_1w \notin \mathcal{L}_{H_u}$, meaning it contains $u$ as a subword. But then $u$ cannot be a subword of $w$, since the path in $\mathcal{G}(H)$ corresponding to $u$ is a succession of loops from $s$ to $s$; and $u$ cannot be a subword of $v_1$, since $v_1 \in \mathcal{L}(H_u)$. Therefore $u$ begins in $v_1$ and ends in $w$. Since $w$ corresponds by definition to the shortest path to $s$ possible in $\mathcal{G}(H)$, it means most of $u$ is in $v_1$ and only a fraction of the last $\ell(\gamma_2)$ is in $w$. Then change $w$ for a new $w'$ as short as possible that breaks the completion of $u$: this is doable because $v_1 \in \mathcal{L}_{H_u}$, so it is extendable to the right. Then add the shortest possible $w''$ that goes to $s$ in $\mathcal{G}(H)$: $u$ is too large to appear elsewhere in $v_1w'w''$. Rename $w'w''$ as $w$.
	
	In all cases, we found $v_1w \in \mathcal{L}_{H_u}$ that reaches $s$ when considered in $\mathcal{G}(H)$. Do the same so that some $xv_2 \in \mathcal{L}_{H_u}$ starts from $s$ when considered in $\mathcal{G}(H)$.
	
	Now, $v_1wxv_2 \in \mathcal{L}_{H}$ since $v_1$ and $v_2$ are large enough so that the $wx$ part of this word has no effect on its extension to a biinfinite word in $H_u$.
	Furthermore, $v_1wxv_2$ does not contain $u$, since the latter must start from $s$ and end at $s$, and this word has been built so that no $u$ follows when $s$ appears. Since $v_1$ and $v_2$ are also large enough to be extended  respectively to the left and to the right without a $u$, we can conclude that $v_1wxv_2 \in \mathcal{L}_{H_u}$. Hence $\mathcal{G}(H_u)$ is transitive.
\end{proof}

\begin{lemma}
	\label{lembezout}
	Let $c_1, \dots, c_l$ be positive integers.
	Let $m = GCD(c_1,\dots,c_l)$.
	
	There is a rank $N \in \N$ such that for all $n \geq N$, $nm$ can be expressed as some $\sum_{i=1}^l k_ic_i$ with $k_i \in \N$.
\end{lemma}

\begin{proof}
	Using Bézout's identity, we can write $m=\sum_{i=1}^l z_ic_i$ for integers $z_i \in \Z$.
	
	Let $c = \sum_{i=1}^l c_i$. $m$ divides $c$.
	Any multiple $nm$ of $m$ can be written as $nm = qc + rm$ where $q, r$ are integers with $0 \leq r < \frac{c}{m}$. Notably, $nm = \sum_{i=1}^l (q + rz_i)c_i$ using the previous equalities with $m$ and $c$.
	
	If $n \geq 2 \frac{c^2}{m^2} \max_i |z_i|$, then $nm \geq \frac{c^2}{m} \max_i |z_i| + \frac{c^2}{m} \max_i |z_i|$. Since $qc + \frac{c^2}{m} \max_i |z_i| > qc + rc \max_i |z_i| \geq qc + rm = nm$, we deduce that $q > \frac{c}{m} \max_i |z_i| > r \max_i |z_i|$.
	
	Therefore, if $n \geq 2 \frac{c^2}{m^2} \max_i |z_i|$, then for all $i$, $q + rz_i > 0$.
	As a consequence, $nm = \sum_{i=1}^l k_ic_i$ with $k_i = q + rz_i > 0$.
\end{proof}

\begin{lemma}
	\label{magiclemma}
	Let $H$ be a transitive one-dimensional SFT of positive entropy.
	Let $u$ and $w_1$ be the words defined as in \cref{lemlem}, with $\alpha = |u| = |w_1|$.
	
	Let 
	\[
	\widetilde{N_H}(n) = Card(\{ v : |v| = n, u \not\sqsubset v \text{ and } w_1vu \in \mathcal{L}_H \}).
	\]
	Then
	\[
	\frac{\log_2(\widetilde{N_H}(n\alpha))}{n\alpha} \xrightarrow[n \to \infty]{} h(H_u).
	\]
\end{lemma}

\begin{proof}
	Let $s$ be the vertex of the Rauzy graph of $H$ that begins and ends the paths corresponding to $u$ and $w_1$. We will similarly name $s$ its label, a word in ${\mathcal{A}_H}^*$.
	
	Let $m$ be the GCD of the lengths of all cycles $c_i$ from $s$ to $s$ that do not pass through $s$ at any other moment, in the Rauzy graph of $H$. From Lemma~\ref{lembezout}, we deduce that there is some $N \in \N$ such that for any $n \geq N$, there is a path of length $nm$ from $s$ to $s$ in the Rauzy graph of $H$, as a concatenation of all cycles $c_i$ with a positive number of times $k_i$ for each of them. Since any order of concatenation works, one can notably build a cycle from $s$ to $s$ of length $nm$ that does not contain $u$ (by concatenating all occurrences of path $\gamma_2$, as defined in the proof of \cref{lemlem} in a row).
	
	Let $d$ be the diameter of the Rauzy graph of $H_u$.
	Let $n\geq N+\frac{2d+|u|}{\alpha}$. Let $v\in \mathcal{L}_{H_u}$ so that $|v|=n\alpha-2d-N\alpha$ (notably $|v| \geq |u|$).
	Since the graph $\mathcal{G}(H_u)$ is transitive, there are two words $v'$ and $v''$ of size smaller than $d$ so that $v'vv'' \in \mathcal{L}_{H_u}$, with the path corresponding to $v'$ beginning with some vertex $s*$ when seen as a path in $\mathcal{G}(H_u)$, and the path corresponding to $v''$ ending with some vertex $*s$ -- where $s*$ is the label of a vertex that has $s$ as a prefix, and $*s$ the label of a vertex that has $s$ as a suffix. This word can also be seen as a path in the Rauzy graph of $H$; there, that path is a cycle from $s$ to $s$; hence $|v'vv''|=km$ for some $k \in \N$ since $m$ divides the length of all cycles from $s$ to $s$ in $\mathcal{G}(H)$.
	
	A notable consequence of this, using the length of $v$ and the fact that $m$ divides $\alpha$ (since $w_1$ corresponds to a cycle from $s$ to $s$ in $H$) is that $2d-|v'|-|v''| > 0$ is a multiple of $m$.
	
	Let $v'''$ be a word of length $2d+N\alpha-|v'|-|v''|$ that is a cycle from $s$ to $s$ in $H$, so that additionally $v'vv''v''' \in \mathcal{L}_{H_u}$. This is doable by doing a correct concatenation of cycles from $s$ to $s$ to build $v'''$, because $2d+N\alpha-|v'|-|v''|$ is a multiple of $m$ which is greater than $N\alpha$, hence it is greater than $Nm$, so we can apply \cref{lembezout}.
	
	We have $|v'vv''v'''|=n\alpha$, $v'vv''v''' \in \mathcal{L}_{H_u}$ and $w_1v'vv''v'''u \in \mathcal{L}_H$.
	
	We deduce that	
	\[
	\widetilde{N_H}(n\alpha)\geq N_{H_u}(n\alpha-2d-N\alpha).
	\]
	Since we also have $N_{H_u}(n\alpha) \geq \widetilde{N_H}(n\alpha)$, taking the logarithm and dividing by $n\alpha$, we obtain
	\[
	\frac{\log_2(\widetilde{N_H}(n\alpha))}{n\alpha}\underset{n\to\infty}{\longrightarrow}h(H_{u}).
	\]
\end{proof}

\subsection{Main result}
\label{subsec:mainresult}

\RealizationEntropy

\begin{proof}
	Let $h \leq h(H), h \in \Pi_1$. We assume that $h(H) > h > 0$; if not, a trivial two-dimensional SFT satisfies the Theorem.
	It is possible to assume that $H$ is transitive: indeed, in one-dimensional SFTs the entropy is due to only one connected component of the Rauzy graph (see \cite{LindMarcus}). Furthermore, since we assumed $h(H) > 0$, that connected component is not a plain cycle, which are of entropy $0$. Hence we can additionally assume that the Rauzy graph of $H$ is not a plain cycle.
	
	Since $H$ is a transitive SFT with positive entropy, thanks to \cref{lemlem} there exists a word $u \in \mathcal{L}_H$ of size $\alpha = |u|$ larger than the order of $H$ and two different words $w_1, w_2 \in L(H)$ such that for $|w_1|=|w_2|=\alpha$ such that $uww_1 \in \mathcal{L}_H$ for all $w \in \{w_1,w_2\}^*$. Moreover $h(H_u) > h$ where $H_u$ is the subshift included in $H$ where the word $u$ does not appear.
	
	Denote
	\[
	\widetilde{N_H}(n) = Card(\{ v : |v| = n, u \not\sqsubset v \text{ and } w_1vu \in \mathcal{L}_H \}).
	\]
	According to \cref{magiclemma}, one has 
	\[
	\frac{\log_2(\widetilde{N_H}(n\alpha))}{n\alpha} \xrightarrow[n \to \infty]{} h(H_u).
	\]
	
	Consider an integer $t \geq 2$ such that $h(H_u) > (1 + \frac{1}{t})h$. Moreover, we ask for $t$ to be big enough so that for any $n \geq t$,
	\[
	\frac{\log_2(\widetilde{N_H}(n\alpha))}{n\alpha} > (1+\frac{1}{t})h
	\]
	which is possible since the left term converges to $h(H_u)$.
	
	Now consider $K$ the sum of the Kolmogorov
	complexities of the following different elements: $h$, $t$, $\alpha$, the algorithms that compute the addition, the
	multiplication, the logarithm, the algorithm that on input $n$ returns $\widetilde{N_H}(n)$ and the algorithm that on
	input $(a,b,c)$ trio of integers returns the first integer $r$ such that $\frac{a}{b+r-1} > c \geq \frac{a}{b+r}$.
	
	Let $R = \ceil{\log_2(3CK)}$ and consider $q = Rt$. One has
	\[
	\frac{\log_2(\widetilde{N_H}(q\alpha))}{(R+q)\alpha} = \frac{q}{q+R} \frac{\log_2(\widetilde{N_H}(q\alpha))}{q\alpha} > \frac{t}{t+1} (1+\frac{1}{t})h = h.
	\]
	
	Therefore, there exists $r > R$ such that
	\[
	\frac{\log_2(\widetilde{N_H}(q\alpha))}{(q+r-1)\alpha} > h \geq \frac{\log_2(\widetilde{N_H}(q\alpha))}{(q+r)\alpha}
	\]
	
	due to the fact that it is a sequence decreasing to $0$ as $r$ increases.

	Consider $h^\prime = (q+r)\alpha\left( h - \frac{\log_2(\widetilde{N_H}(q\alpha))}{(q+r)\alpha} \right) > 0$. The Kolmogorov complexity of $h^\prime$ is less than $3K$, because $K$ contains the complexity of doing each of the operations used in $h^\prime$ except computing $q$, and computing $q$ requires to compute $R$ which requires to compute $K$ which has Kolmogorov complexity at most $K$. Assembling all of this, we obtain a complexity of at most $3K$.
	
	Consequently, using \cref{prop:algo}, there exists a constant $C>0$ and a Wang tile set $T_W$ with at most $3CK$ tiles such that the associated SFT $W$ has for entropy $h(W) = h^\prime$.
	
	Now, consider the two-dimensional subshift $X \subset \A^{\Z^2}$ with the following local rules:
	\begin{itemize}
		\item every line satisfies the conditions of $H$, thus $\pi_{\vec{e_1}}(X) \subset H$;
		\item if the word $u$ appears horizontally starting at position $(i,j)$, it also appears at positions $(i,j+1)$ and $(i,j-1)$;
		\item the word $u$ appears with the horizontal period $(q+r)\alpha$ and it cannot appear elsewhere on a given line;
		\item the word $u$ is followed by a word of $\{w_1,w_2\}^R.w_1^{r-R-1}$;
		\item the tiles coded in binary by the words of $\{w_1,w_2\}^R$ after $u$ satisfy horizontal and vertical constraints imposed by the Wang tile set $T_W$.
	\end{itemize}

	The window of size $q\alpha$ that remains between two words $u$ and is not filled by the previous constraints has the only restriction of respecting the horizontal conditions of $H$ and of not containing $u$.
	There are horizontal lines that respect all of the horizontal conditions in $X$, because:
	\begin{itemize}
	\item $u\{w_1,w_2\}^*w_1 \subset \mathcal{L}_H$;
	\item $\widetilde{N_H}(q\alpha) \neq 0$ by the use of \cref{magiclemma};
	\item forcing $u$ to appear nowhere else than with a $(q+r)\alpha$ period is possible because of \cref{lemlem}: $u$ appears as a subword of any word in $u\{w_1,w_2\}^*w_1$ at the very beginning only;
	\end{itemize}
	The only vertical restriction that is added between these horizontal lines is only the alignment of the $u$'s between two lines where this word appears; so there are configurations in $X$ overall.
	
	Let $n = (q+r) \alpha$. One has
	\[
	\widetilde{N_H}(q\alpha)^{k(k-1)} N_W(k-1, k) \leq N_X(kn, k)
	\]
	because in any $kn \times k$ window in $X$ there are at least $(k-1) \times k$ complete horizontal segments starting with a word $u$, that encode a $k-1 \times k$ pattern of $W$ binary with $\{w_1,w_2\}^R$, and each of them also contains $q\alpha$ additional characters that link $w_1$ and $u$ and that do not contain $u$.
	
	One also has
	\[
	N_X(kn, k) \leq \widetilde{N_H}(q\alpha)^{k(k+1)} N_W(k+1, k)
	\]
	because similarly, in any $kn \times k$ window in $X$ there are less than $(k+1) \times k$ complete horizontal segments starting with a word $u$. One obtains the resulting inequality:
	
	\[
	\frac{(k-1)\log_2(\widetilde{N_H}(q\alpha))}{kn} + \frac{\log_2(N_W(k-1, k))}{nk^2} \leq \frac{\log_2(N_X(kn, k))}{nk^2} \leq \frac{(k+1)\log_2(\widetilde{N_H}(q\alpha))}{kn} + \frac{\log_2(N_W(k+1, k))}{nk^2}
	\]
	
	Thus, taking the limit when $k \to \infty$, one obtains
	
	\[
	h(X) = \frac{\log_2(\widetilde{N_H}(q\alpha))}{(q+r)\alpha} + \frac{h(W)}{(q+r)\alpha} = h.
	\]
\end{proof}

\subsection{Some consequences}

A direct consequence of \cref{th:RealizationEntropy} -- up to extending the construction to higher dimensions, which is straightforward -- is the characterization of all the entropies in any dimension $d>2$, as is obtained in \cite{HM}:
\begin{corollary}
	\label{th:bonus}
	Let $\A$ be a finite alphabet. For any number $h\in\Pi_1$ such that $0\leq h\leq\log_2(\A)$, there exists a SFT on $\A^{\Z^d}$ of entropy $h$.
\end{corollary}

Given a one-dimensional \emph{effective subshift} $H$, that is a subshift with a list of forbidden patterns enumerable by a Turing Machine, it is known that there exists a two-dimensional sofic subshift which has $H$ as projective subaction \cite{simulation}. A natural question is if it is possible to impose to have a specific entropy $h$ with $h\leq h(H)$. The next corollary is a partial result for effective subshifts containing a SFT.

\begin{corollary}
	\label{th:bonus2}
	Let $H \subset \A^\Z$ be an effective subshift with $H' \subset H$ a SFT.
	Let $0 \leq h \leq h(H'), h \in \Pi_1$.
	
	Then there exists a two-dimensional sofic subshift $X$ such that $h(X) = h$ and $\pi_{\vec{e_1}}(X) = H$.
\end{corollary}

\begin{proof}
	Let $h \leq h(H'), h \in \Pi_1$.
	
	Let $X'$ be the two-dimensional SFT obtained from $H'$ using \cref{th:RealizationEntropy}, which is so that $h(X') = h$ and $\pi_{\vec{e_1}}(X') \subset H'$.
	
	Let $Y_H$ be the SFT with $H$ as horizontal rules and no added vertical rule.
	
	Let $Z$ be a SFT over $\{0, 0', 1\}^{\Z^2}$ so that two horizontally successive elements must be the same, so that only $0$ and $1$ can be above $0$, only $0'$ can be above $1$, and only $0'$ can be above $0'$. The configurations of $Z$ are made of at most a single line of $1$s with $0'$s above and $0$s below.
	
	We define $Y = Z \times Y_H \times Y'$ which is a SFT by product, and finally $X$ which is the projection $\pi(Y)$ with
	\[
	\pi(y)_{i,j} =
	\begin{cases}
	(y_2)_{i,j} &\quad\text{if } (y_1)_{i,j} = 1\\
	(y_3)_{i,j} &\quad\text{if } (y_1)_{i,j} \neq 1\\
	\end{cases}
	\]
	where $y_1, y_2$ and $y_3$ are the projections of a configuration $y \in Y$ on the three SFTs of the product.
	
	In the end, a configuration in $X$ has at most one line that can be any configuration of $H$; and all its other lines are from $X'$, ensuring $h(X) = h(X') = h$. Furthermore, since $\pi_{\vec{e_1}}(X') \subset H' \subset H$, we obtain $\pi_{\vec{e_1}}(X) = H$.
\end{proof}

\begin{remark}
 In the theorem of simulation of~\cite{Aubrun-Sablik-2010,Durand-Romashchenko-Shen-2012}, any effective subshift $H$ can be realized as the projective subaction of a two-dimensional sofic subshift, but that sofic subshift has zero entropy. It is natural to ask if it is possible to realize $H$ with any possible entropy for the two-dimensional subshift, that is to say any $\Pi_1$-computable number of $[0,h(H)]$. This question is related to the conjecture that a one-dimensional subshift $H$ is sofic if and only if the two-dimensional subshift $H^\ast=\{x\in\A^{\Z^2}: \textrm{ for all }i\in\Z, x_{\Z\times\{i\}}\in H\}$ is sofic. We remark that the entropy of $H^\ast$ is $h(H)$, obtained with completely independent rows; thus allowing entropy in the realization is a way of giving some independence between rows.
\end{remark}

%
%

\section{Theorem of simulation under interplay between horizontal and vertical conditions}
\label{sec:simulation}

Another way of looking at one-dimensional constraints in a two-dimensional setting, as mentioned in \cref{subsec:horizontalconstraints}, is to try to understand if a pair of constraints are compatible. Some are chosen as horizontal conditions and the others as vertical conditions: can the resulting combined subshift $X_{H,V}$ be anything we want?

For a few nearest-neighbor horizontal constraints, it is not that hard to realize that, whichever vertical constraints we match them with, the combined subshift will necessarily contain periodic configurations, and therefore not any SFT on $\Z^2$ will be obtainable. These constraints are said to respect condition D, see \cref{subsec:conditionD}, which is the union of three smaller easy-to-understand conditions.

For any other kind of horizontal constraint, we prove in a disjunctive fashion that we can simulate, in some sense, any two-dimensional subshift $Y$. However, since the entropy of $X_{H,V}$ is bounded by $h(H)$, the simulation cannot be a conjugacy; the correct notion is the one of root of a subshift, defined in \cref{subsec:root}.

This section is devoted to proving the following, for a condition D to be defined in \cref{subsec:conditionD}.

\begin{theorem}
\label{th:root}
Let $H\subset\A^\Z$ be a one-dimensional \emph{nearest-neighbor} SFT whose Rauzy graph does not satisfy condition D. For any two-dimensional SFT $Y\subset\B^{\Z^2}$, there exists a one-dimensional SFT $V_Y\subset\A^\Z$ such that $X_{H,V_Y}$ is a $(m,n)$th root of $Y$ for some $m, n \in \N^2$. Furthermore, $m$, $n$ and $V_Y$ can be computed algorithmically.
\end{theorem}

\subsection{The condition D}
\label{subsec:conditionD}

\begin{definition}
	\label{conditionD}
	We say that an oriented graph $\mathcal{G} = (\mathcal{V},\vec{E})$ verifies condition $(D)$ (for "Decidable") if all its SCCs have a type in common among the following list. A SCC $S$ can be of none, one or several of these types:
	\begin{itemize}
		\item for all vertices $v \in S$, we have $(v,v) \in \vec{E}$: we say that $S$ is of \emph{reflexive type};
		
		\item for all vertices $v \neq w \in S$ such that $(v, w) \in \vec{E}$, we have $(w,v) \in \vec{E}$: we say that $S$ is of \emph{symmetric type};
		
		\item there exists $p \in \N$ so that $S = \bigsqcup_{i=0}^{p-1} V_i$ with, for any  $v \in V_i$, we have $[ (v,w) \in \vec{E} \Leftrightarrow w \in V_{i+1} ]$ with $i+1$ meant modulo $p$: we say that $S$ is of \emph{state-split cycle type}.
	\end{itemize}
	
\end{definition}

\begin{remark}
	The term state-split is used in reference to a notion introduced in \cite{LindMarcus}: a state-split cycle is a cycle where some vertices have been split.
	
	Note that $S = \{v\}$ a single vertex with a loop is also of symmetric type. Similarly, a single vertex is of state-split cycle type with a partition with one unique class $V_0$.
\end{remark}

\begin{example}
	Consider the following Rauzy graphs:
	
	\begin{minipage}{0.3\textwidth}
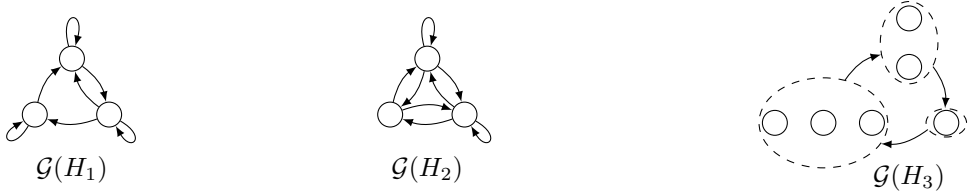
		\begin{tikzpicture}
		\node[draw, circle] (b) {};
		\node[draw, circle] (a) [below left = 0.5cm and 0.25cm of b] {};
		\node[draw, circle] (c) [below right = 0.5cm and 0.25cm of b] {};
		
		\draw[->, >=latex] (a) to[bend left=20] (b);
		\draw[->, >=latex] (b) to[bend left=20] (c);
		\draw[->, >=latex] (c) to[bend left=20] (a);
		
		\draw[->, >=latex] (c) to[bend left=20] (b);
		
		\draw[->, >=latex] (b) to[loop above] (b);
		\draw[->, >=latex] (c) to [out=330,in=300,looseness=12] (c);
		\draw[->, >=latex] (a) to [out=240,in=210,looseness=12] (a);
		
		\node (h) [below = 1cm of b] {$\mathcal{G}(H_1)$};
		\end{tikzpicture}
	\end{minipage}
	\begin{minipage}{0.3\textwidth}
		\begin{tikzpicture}
		\node[draw, circle] (b) {};
		\node[draw, circle] (a) [below left = 0.5cm and 0.25cm of b] {};
		\node[draw, circle] (c) [below right = 0.5cm and 0.25cm of b] {};
		
		\draw[->, >=latex] (a) to[bend left=20] (b);
		\draw[->, >=latex] (b) to[bend left=20] (c);
		\draw[->, >=latex] (c) to[bend left=20] (a);
		
		\draw[->, >=latex] (c) to[bend left=20] (b);
		\draw[->, >=latex] (b) to[bend left=20] (a);
		\draw[->, >=latex] (a) to[bend left=20] (c);
		
		\draw[->, >=latex] (b) to[loop above] (b);
		\draw[->, >=latex] (c) to [out=330,in=300,looseness=12] (c);
		
		\node (h) [below = 1cm of b] {$\mathcal{G}(H_2)$};
		\end{tikzpicture}
	\end{minipage}
	\begin{minipage}{0.3\textwidth}
		\begin{tikzpicture}
		\node[draw, circle] (b) {};
		\node[draw, circle] (a) [below left = 0.5cm and 0.25cm of b] {};
		\node[draw, circle] (c) [below right = 0.5cm and 0.25cm of b] {};
		\node[draw, circle] (d) [left = 0.3cm of a] {};
		\node[draw, circle] (e) [left = 0.3cm of d] {};
		\node[draw, circle] (f) [above = 0.3cm of b] {};

		\node [rotate=0][draw,dashed,inner sep=0.5pt, circle,yscale=.7,fit={(a)(d)(e)}] (meta1) {};
		\node [rotate=90][draw,dashed,inner sep=0.5pt, circle,yscale=.7,fit={(b)(f)}] (meta2) {};
		\node [rotate=0][draw,dashed,inner sep=0.5pt, circle,yscale=.7,fit={(c)}] (meta3) {};
		
		\draw[->, >=latex] (meta1) to[bend left=20] (meta2);
		\draw[->, >=latex] (meta2) to[bend left=20] (meta3);
		\draw[->, >=latex] (meta3) to[bend left=20] (meta1);
		
		\node (h) [below = 1cm of b] {$\mathcal{G}(H_3)$};
		\end{tikzpicture}
	\end{minipage}

	where edges between dotted sets of vertices in the third graph represent that all vertices from the first set have edges leading to all vertices of the second set.
	
	These three graphs respect condition D, being respectively of reflexive, symmetric and state-split cycle type.
\end{example}

\subsection{Generic construction}
\label{subsec:core}

In this section, we describe a set of properties on a directed graph, forming a condition called \emph{condition C} that has stronger requirements than condition D. Condition C allows a generic construction of the proof of \cref{th:root} for the associated one-dimensional, nearest-neighbor SFT.

In all that follows, \emph{we denote elements of cycles with an index that is written modulo the length of the corresponding cycle}.
We need the following defintions before describing condition C.

\begin{definition}
	Let $C^1$ and $C^2$ be two cycles in an oriented graph $\mathcal{G}$, with elements denoted respectively $c^1_i$ and $c^2_j$. Let $M := LCM(|C^1|,|C^2|)$. Let $C$ be any cycle in that graph, with elements denoted $c_i$.
	
	We say that the cycles $C^1$ and $C^2$ contain a \emph{good pair} if there is a pair $(i,j)$ and an integer $1<l<M-1$ such that $c^1_i \neq c^2_j, c^1_{i+1} \neq c^2_{j+1}, \dots, c^1_{i+l} \neq c^2_{j+l}$ and $c^1_{i+(l+1)} = c^2_{j+(l+1)}, \dots, c^1_{i+(M-1)} = c^2_{j+(M-1)}$.
	All pairs $(i+p,j+p), p \in \{0,\dots,M-1\}$ are said to be \emph{in the orbit of a good pair}.
	
	We say that a cycle $C$ \emph{contains a uniform shortcut} if there exists a $k \in \{0, 2, 3 ..., |C|-1\}$ (any value except $1$) such that for any $c_i \in C, (c_i, c_{i+k}) \in \vec{E}$.
	
	We say that there is a \emph{cross-bridge} between $C^1$ and $C^2$ if there are $i \in \{0, ..., |C^1|-1\}$ and $j \in \{0, ..., |C^2|-1\}$ with $c^1_i \neq c^2_j$ and $c^1_{i+1} \neq c^2_{j+1}$ such that $(c^1_i,c^2_{j+1}) \in \vec{E}$ and $(c^2_j,c^1_{i+1}) \in \vec{E}$.
	
	Finally, an \emph{attractive vertex} for $C$ is any vertex $v$ so that for all $c \in C$, $(c,v) \in \vec{E}$. A \emph{repulsive vertex} is defined similarly, with $(v,c) \in \vec{E}$ instead.
\end{definition}

These definitions are illustrated in \cref{fig:condC}.

\begin{definition}
	\label{def:condC}
	Let $H \subset \A^\Z$ be a one-dimensional nearest-neighbor SFT. We say that $H$ \emph{verifies condition $C$} if $\mathcal{G}(H)=(\mathcal{V},\vec{E})$ contains two cycles $C^1$ and $C^2$, of elements denoted respectively $c^1_i$ and $c^2_j$, with the following properties:
	
	\begin{enumerate}[(i)]
		\item $|C^1| \geq 3$;
		
		\item $C^1$ and $C^2$ contain a good pair;
		
		\item There is no uniform shortcut neither in $C^1$ nor in $C^2$;
		
		\item There is no cross-bridge between $C^1$ and $C^2$;
		
		\item $C^1$ has either no attractive vertex in $C^1 \cup C^2$ or no repulsive vertex in $C^1 \cup C^2$.
	\end{enumerate}

	Some vertices can be \emph{repeated}; that is, these cycles can go several times through the same vertex as long as they verify the aforementioned properties.
\end{definition}

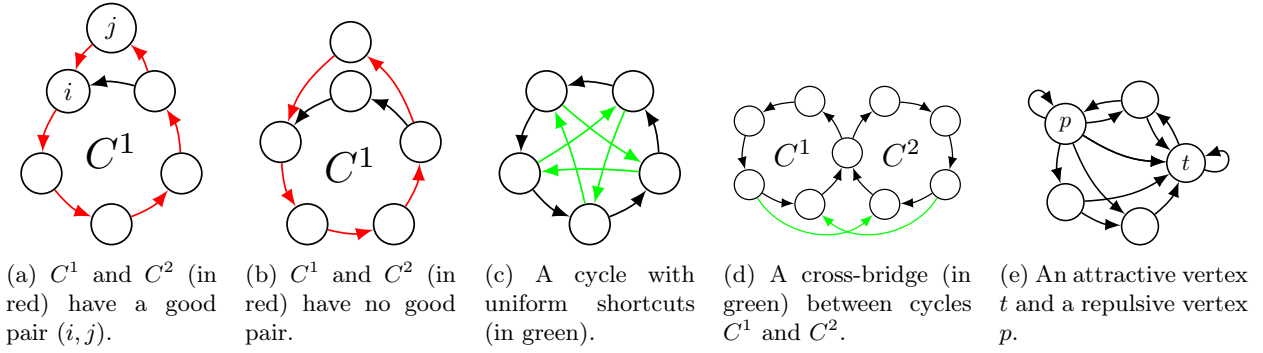
\begin{figure}[t]
	\centering
	\begin{subfigure}[t]{0.17\textwidth}
		\centering
		\resizebox{\columnwidth}{!}{
			\begin{tikzpicture}[scale=0.4,rotate=-90]
			\def \radius {1.5cm};
			
			\node[draw, circle] (a) at (0:\radius) {};
			\node[draw, circle] (b) at (72:\radius) {};
			\node[draw, circle] (c) at (144:\radius) {};
			\node[draw, circle,scale=0.6] (d) at (216:\radius) {$i$};
			\node[draw, circle] (e) at (288:\radius) {};
			
			\node[draw, circle,scale=0.6] (f) at (-2.5,0) {$j$};
			
			\draw (0,0) node {$C^1$};
			
			\draw[->, >=latex,color=red] (a) to[bend right=15] (b);
			\draw[->, >=latex,color=red] (b) to[bend right=15] (c);
			\draw[->, >=latex] (c) to[bend right=15] (d);
			\draw[->, >=latex,color=red] (d) to[bend right=15] (e);
			\draw[->, >=latex,color=red] (e) to[bend right=15] (a);
			
			\draw[->, >=latex,color=red] (c) to[bend right=15] (f);
			\draw[->, >=latex,color=red] (f) to[bend right=15] (d);
			\end{tikzpicture}
		}
		\caption{$C^1$ and $C^2$ (in red) have a good pair $(i,j)$.}
		\label{condC1}
	\end{subfigure}
	\hfill
	\begin{subfigure}[t]{0.17\textwidth}
		\centering
		\resizebox{\columnwidth}{!}{
			\begin{tikzpicture}[scale=0.4,rotate=90]
			\def \radius {1.5cm};
			
			\node[draw, circle] (a) at (0:\radius) {};
			\node[draw, circle] (b) at (72:\radius) {};
			\node[draw, circle] (c) at (144:\radius) {};
			\node[draw, circle] (d) at (216:\radius) {};
			\node[draw, circle] (e) at (288:\radius) {};
			
			\node[draw, circle] (f) at (0:2.5cm) {};
			
			\draw (0,0) node {$C^1$};
			
			\draw[->, >=latex] (a) to[bend right=15] (b);
			\draw[->, >=latex,color=red] (b) to[bend right=15] (c);
			\draw[->, >=latex,color=red] (c) to[bend right=15] (d);
			\draw[->, >=latex,color=red] (d) to[bend right=15] (e);
			\draw[->, >=latex] (e) to[bend right=15] (a);
			
			\draw[->, >=latex,color=red] (e) to[bend right=15] (f);
			\draw[->, >=latex,color=red] (f) to[bend right=15] (b);
			\end{tikzpicture}
		}
		\caption{$C^1$ and $C^2$ (in red) have no good pair.}
		\label{condC2}
	\end{subfigure}
	\hfill
	\begin{subfigure}[t]{0.17\textwidth}
		\centering
		\resizebox{\columnwidth}{!}{
			\begin{tikzpicture}[scale=0.4,rotate=-90]
			\def \radius {1.5cm};
			
			\node[draw, circle] (a) at (0:\radius) {};
			\node[draw, circle] (b) at (72:\radius) {};
			\node[draw, circle] (c) at (144:\radius) {};
			\node[draw, circle] (d) at (216:\radius) {};
			\node[draw, circle] (e) at (288:\radius) {};
			
			\draw[->, >=latex] (a) to[bend right=15] (b);
			\draw[->, >=latex] (b) to[bend right=15] (c);
			\draw[->, >=latex] (c) to[bend right=15] (d);
			\draw[->, >=latex] (d) to[bend right=15] (e);
			\draw[->, >=latex] (e) to[bend right=15] (a);
			
			\draw[->, >=latex,color=green] (a) to[bend right=5] (d);
			\draw[->, >=latex,color=green] (b) to[bend right=5] (e);
			\draw[->, >=latex,color=green] (c) to[bend right=5] (a);
			\draw[->, >=latex,color=green] (d) to[bend right=5] (b);
			\draw[->, >=latex,color=green] (e) to[bend right=5] (c);
			\end{tikzpicture}
		}
		\caption{A cycle with uniform shortcuts (in green).}
		\label{condC3}
	\end{subfigure}
	\hfill
	\begin{subfigure}[t]{0.2\textwidth}
		\centering
		\resizebox{\columnwidth}{!}{
			\begin{tikzpicture}[scale=0.4]
			\def \radius {1.5cm};
			
			\node[draw, circle] (a) at (0:\radius) {};
			\node[draw, circle] (b) at (72:\radius) {};
			\node[draw, circle] (c) at (144:\radius) {};
			\node[draw, circle] (d) at (216:\radius) {};
			\node[draw, circle] (e) at (288:\radius) {};
			
			\begin{scope}[xshift=3cm,xscale=-1]
			\node[draw, circle] (f) at (72:\radius) {};
			\node[draw, circle] (g) at (144:\radius) {};
			\node[draw, circle] (h) at (216:\radius) {};
			\node[draw, circle] (i) at (288:\radius) {};
			\draw (0,0) node {$C^2$};
			\end{scope}
			
			\draw (0,0) node {$C^1$};
			
			\draw[->, >=latex] (a) to[bend right=15] (b);
			\draw[->, >=latex] (b) to[bend right=15] (c);
			\draw[->, >=latex] (c) to[bend right=15] (d);
			\draw[->, >=latex] (d) to[bend right=15] (e);
			\draw[->, >=latex] (e) to[bend right=15] (a);
			
			\draw[->, >=latex] (a) to[bend left=15] (f);
			\draw[->, >=latex] (f) to[bend left=15] (g);
			\draw[->, >=latex] (g) to[bend left=15] (h);
			\draw[->, >=latex] (h) to[bend left=15] (i);
			\draw[->, >=latex] (i) to[bend left=15] (a);
			
			\draw[->, >=latex,color=green] (h) to[bend left=50] (e);
			\draw[->, >=latex,color=green] (d) to[bend right=50] (i);
			\end{tikzpicture}
		}
		\caption{A cross-bridge (in green) between cycles $C^1$ and $C^2$.}
		\label{condC4}
	\end{subfigure}
	\hfill
	\begin{subfigure}[t]{0.2\textwidth}
		\centering
		\resizebox{\columnwidth}{!}{
			\begin{tikzpicture}[scale=0.4]
			\def \radius {1.5cm};
			
			\node[draw, circle,scale=0.6] (a) at (0:\radius) {$t$};
			\node[draw, circle] (b) at (72:\radius) {};
			\node[draw, circle,scale=0.6] (c) at (144:\radius) {$p$};
			\node[draw, circle] (d) at (216:\radius) {};
			\node[draw, circle] (e) at (288:\radius) {};
			
			\draw[->, >=latex] (a) to[bend right=15] (b);
			\draw[->, >=latex] (b) to[bend right=15] (c);
			\draw[->, >=latex] (c) to[bend right=15] (d);
			\draw[->, >=latex] (d) to[bend right=15] (e);
			\draw[->, >=latex] (e) to[bend right=15] (a);
			
			\draw[->, >=latex] (b) to[bend right=15] (a);
			\draw[->, >=latex] (c) to[bend right=15] (a);
			\draw[->, >=latex] (d) to[bend right=15] (a);
			
			\draw[->, >=latex] (c) to[bend right=15] (e);
			\draw[->, >=latex] (c) to[bend right=15] (a);
			\draw[->, >=latex] (c) to[bend right=15] (b);
			
			\draw[->, >=latex] (c) to[out=164,in=124,looseness=6] (c);
			\draw[->, >=latex] (a) to[out=-20,in=20,looseness=6] (a);
			\end{tikzpicture}
		}
		\caption{An attractive vertex $t$ and a repulsive vertex $p$.}
		\label{condC5}
	\end{subfigure}
	\caption{Cases of compliance or not with elements of Condition $C$.}
	\label{fig:condC}
\end{figure}

\begin{proposition}
	\label{propcondC}
	If $\mathcal{G}(H)$ verifies condition $C$, and $W$ is a Wang shift, then there exists an explicit SFT $V_W \subset \A^\Z$ such that $X_{H,V_W}$ is a root of $W$.
\end{proposition}

The rest of the subsection is devoted to proving this result.

Let $H$ with $\mathcal{G}(H)$ verifying condition $C$. We focus on encoding, in the correct $X_{H,V}$, a full shift on an alphabet $\tau$ of cardinality $N$. Then, the possibility to add vertical rules will allow us to encode any Wang shift $W$ using this alphabet, that is, to simulate the configurations of $W$ as a root of $X_{H,V}$.

For the rest of the construction, we name $M := LCM(|C^1|,|C^2|)$ and $K := 2 |C^1| + |C^2| + 3$. We suppose that $N \geq 2$.
Indeed, the case of encoding a monotile Wang shift is easy: consider only the cycle $C^1$ in $\mathcal{G}(H)$, that may have extra edges from one element to another, but no uniform shortcut. Build $V$ the vertical SFT from the graph $\mathcal{G}(H)^\prime$ obtained by removing any of those extra edges, keeping only a plain cycle -- the same as $C^1$. Then $X_{H,V}$ contains only the translate of one configuration, that cycles through the elements of $C^1$ in the correct order, both horizontally and vertically. This is a root of a monotile Wang shift.

We refer to \cref{columns1}
in the description that follows. We use the term \emph{slice} as a truncation of a column: it is a part of width $1$ and of finite height. We use the following more specific denominations for the various scales of our construction:
\begin{itemize}
	\item A \emph{macro-slice} is a slice of height $KMN$. Any column of $X_{H,V}$ will merely be made of a succession of some specific macro-slices called ordered macro-slices (see below).
	\item A \emph{meso-slice} is a slice of height $MN$; meso-slices are assembled into macro-slices.
	\item A \emph{micro-slice} is a slice of height $N$. This subdivision is used inside specific meso-slices called code meso-slices (see below).
\end{itemize}
Although any scale of slice could denote any truncation of column of the right size, we focus on specific slices that are meaningful because of what they contain, so that we can assemble them precisely. They are:
\begin{itemize}
	\item An $(i,j)$ $k$-coding micro-slice is a micro-slice composed of $N-1$ symbols $c^1_i$ and one symbol $c^2_j$ at position $k$. It encodes the $k$th tile of alphabet $\tau$, unless $c^1_i=c^2_j$: in that case, it is called a \emph{buffer} and encodes nothing. We can write ``$(i,j)$ coding micro-slice'' when we do not want to specify which tile is encoded.
	
	\item An $(i_0,j_0)$ $(k,l)$-code meso-slice is a meso-slice made of $M$ vertically successive coding micro-slices. The one on top is an $(i_0,j_0)$ $k$-coding micro-slice.
	We add the following restrictions:
		\begin{itemize}
			\item an $(i,j)$ $k$-coding micro-slice must be vertically followed by a $(i+1,j+1)$ $k$-coding micro-slice, unless $c^1_i \neq c^2_j$ but $c^1_{i+1} = c^2_{j+1}$, that is if the new micro-slice is a buffer but the previous one is not;
			\item if $c^1_i = c^2_j$ but $c^1_{i+1} \neq c^2_{j+1}$, then the buffer must be followed by a $(i+1,j+1)$ $l$-coding micro-slice.
		\end{itemize}
	We can write ``$(i_0,j_0)$ code meso-slice'' when we do not want to specify which tiles are encoded.
	
	\item An $i$-border meso-slice is made of $\frac{M}{|C^1|} N$ times the vertical repetition of all the elements of the cycle $C^1$, starting with $c^1_i$.
	
	\item A $c^1_i$ meso-slice is made of $MN$ times the vertical repetition of element $c^1_i$, denoted $(c^1_i)^MN$ in \cref{columns1}. The same definition holds for a $c^2_j$ meso-slice.
	
	\item The succession of a $c^1_i$ meso-slice, then a $c^1_{i+1}$ meso-slice, ..., then a $c^1_{i-1}$ meso-slice is called a $i$ $C^1$-slice. It is of height $MN|C^1|$. Similarly, we define a $j$ $C^2$-slice (of height $MN|C^2|$).
	
	\item Finally, a $(i,j)$-ordered $(k,l)$-coding macro-slice is the succession of a $i$-border meso-slice, a $i$ $C^1$-slice, a second $i$ $C^1$-slice, a $j$ $C^2$-slice, a $i$-border meso-slice, and finally a $(i,j)$ $(k,l)$-code meso-slice.
	We can write ``$(i,j)$-ordered macro-slice'' when we do not want to specify which tiles are encoded.
\end{itemize}

\begin{remark}
	The $(i_0,j_0)$ $(k,l)$-code meso-slice is well-defined because since $C^1$ and $C^2$ contain a good pair, a code meso-slice is a vertical succession of coding micro-slide and of buffers, with either only one vertical succession of coding micro-slice of one vertical succession of buffers (possibly both). Therefore, there can be at most one change, from $k$ to $l$, in the tiles encoded; $k$ is then called the \emph{main-coded tile}, and $l$ the \emph{side-coded tile}.
	
	Note that if $c^1_{i_0} \neq c^2_{j_0}$ but $c^1_{i_0-1} = c^2_{j_0-1}$, then the meso-slice contains no side-coded tile -- its last coding micro-slice is the last buffer -- so $l$ is actually irrelevant.
	Conversely, if $c^1_{i_0} = c^2_{j_0}$, then the meso-slice contains no main-coded tile, so $k$ is actually irrelevant.
\end{remark}

Now, the patterns we authorize in $V$ are the $(i_0+p,j_0+p)$-ordered macro-slices with a good pair $(i_0,j_0)$ and $p \in \{0, \dots, M-1\}$, and all patterns that allow the vertical juxtaposition of two $(i,j)$-ordered macro-slices, using the same $i$ and $j$, but possibly different code meso-slices.
We prove below that this is enough for our resulting $X_{H,V}$ to simulate a full shift on $\tau$.

\begin{figure}[p]
	\centering
	\begin{subfigure}[t]{0.55\textwidth}
		\centering
		\resizebox{\columnwidth}{!}{
			\begin{tikzpicture}[scale=1
			,rotate=-90,yscale=-1
			]
			
			
			\fill[color=black!20] (-3,-1) rectangle (8,-2);
			\draw (-3,-1) rectangle(8,-2);

			
			\draw(-2.5,-1.5) node[scale=1] {$c_i^1$};
			\draw(-1.5,-1.5) node[scale=1] {$c_{i+1}^1$};
			\draw(-0.5,-1.5) node[scale=0.75] {$\dots$};
			\draw(0.5,-1.5) node[scale=1] {$c_{i-1}^1$};
			\draw(1.5,-1.5) node[scale=1] {$c_{i}^1$};
			\draw(2.5,-1.5) node[scale=1] {$c_{i+1}^1$};
			\draw(3.5,-1.5) node[scale=0.75] {$\dots$};
			\draw(4.5,-1.5) node[scale=1] {$c_{i-1}^1$};
			\draw(5.5,-1.5) node[scale=0.75] {$\dots$};
			\draw(6.5,-1.5) node[scale=0.75] {$\dots$};
			\draw(7.5,-1.5) node[scale=1] {$c_{i-1}^1$};

			
			\draw (-3,-1) -- (0,0);
			\draw (8,-1) -- (1,0);

			\fill[color=black!20] (0,0) rectangle (1,1);
			\fill[color=black!20] (13,0) rectangle (14,1);
			
			\fill[color=black!10](1,0)rectangle(9,1);
			\fill[color=black!5](9,0)rectangle(13,1);
			
			\draw (0,0) rectangle (15,1);
			\foreach \k in {1,...,14}
			{
				\draw (\k,0)--(\k,1);
			};

			\draw(-1,0.7) node[scale=1,align=center] {$(i,j)$-ordered\\macro-slice};
			
			\draw(0.5,0.5) node[scale=0.75] {border};
			\draw(1.5,0.5) node[scale=0.65] {$(c_i^1)^{MN}$};
			\draw(2.5,0.5) node[scale=0.65] {$(c_{i+1}^1)^{MN}$};
			\draw(3.5,0.5) node[scale=0.75] {$\dots$};
			\draw(4.5,0.5) node[scale=0.65] {$(c_{i-1}^1)^{MN}$};
			\draw(5.5,0.5) node[scale=0.65] {$(c_i^1)^{MN}$};
			\draw(6.5,0.5) node[scale=0.65] {$(c_{i+1}^1)^{MN}$};
			\draw(7.5,0.5) node[scale=0.75] {$\dots$};
			\draw(8.5,0.5) node[scale=0.65] {$(c_{i-1}^1)^{MN}$};
			\draw(9.5,0.5) node[scale=0.65] {$(c_{j}^2)^{MN}$};
			\draw(10.5,0.5) node[scale=0.65] {$(c_{j+1}^2)^{MN}$};
			\draw(11.5,0.5) node[scale=0.75] {$\dots$};
			\draw(12.5,0.5) node[scale=0.65] {$(c_{j-1}^2)^{MN}$};
			\draw(13.5,0.5) node[scale=0.75] {border};
			\draw(14.5,0.5) node[scale=0.75] {code};
			
			\draw [decorate,decoration={brace,amplitude=5pt,mirror,raise=3ex}]
			(1,0.75) -- (5,0.75) node[midway,yshift=-3em]{};
			
			\draw(3,2.3) node[scale=1] {$i$ $C^1$-slice};
			
			\draw [decorate,decoration={brace,amplitude=5pt,mirror,raise=3ex}]
			(9,0.75) -- (13,0.75) node[midway,yshift=-3em]{};
			
			\draw(11,2.3) node[scale=1] {$j$ $C^2$-slice};
			
			
			\draw(14,0)--(10,-1);
			\draw(15,0)--(17,-1);
			
			\draw (10,-2) rectangle(17,-1);
			\foreach \k in {11,...,16}
			{
				\draw (\k,-2)--(\k,-1);
			};

			\draw(9,-1.5) node[scale=1] {meso-slices};
			
			\draw(10.5,-1.5) node[scale=0.75,align=center] {$\tau_k$\\$i$\\$j$};
			\draw(11.5,-1.5) node[scale=0.75,align=center] {$\tau_k$\\$i+1$\\$j+1$};
			\draw(12.5,-1.5) node[scale=0.75] {$\dots$};
			\draw(13.5,-1.5) node[scale=0.75,align=center] {$\tau_k$\\$i-4$\\$j-4$};
			\fill[pattern=north west lines, pattern color=black!30] (14,-2) rectangle (15,-1);
			\fill[pattern=north west lines, pattern color=black!30] (15,-2) rectangle (16,-1);
			\draw(16.5,-1.5) node[scale=0.75,align=center] {$\tau_\ell$\\$i-1$\\$j-1$};

			
			\draw(7,-3)--(11,-2);
			\draw(16,-3)--(12,-2);
			
			\draw (7,-4) rectangle(16,-3);
			
			\draw(6,-3.5) node[scale=1,align=center] {$(i+1,j+1)$\\$k$-coding\\micro-slice};
			
			\draw(7.5,-3.5) node[scale=1] {$c_{i+1}^1$};
			\draw(8.5,-3.5) node[scale=1] {$c_{i+1}^1$};
			\draw(9.5,-3.5) node[scale=0.75] {$\dots$};
			\draw(10.5,-3.5) node[scale=1] {$c_{i+1}^1$};
			\draw(11.5,-3.5) node[scale=1] {$c_{j+1}^2$};
			\draw(12.5,-3.5) node[scale=1] {$c_{i+1}^1$};
			\draw(13.5,-3.5) node[scale=1] {$c_{i+1}^1$};
			\draw(14.5,-3.5) node[scale=0.75] {$\dots$};
			\draw(15.5,-3.5) node[scale=1] {$c_{i+1}^1$};
			
			\draw[->,>=latex] (11.5,-4.5) -- (11.5,-4.1);
			\draw(11.5,-4.7) node {$k$};
			
			
			\end{tikzpicture}
		}
		\caption{Columns allowed, for $(i,j)$ in the orbit of a good pair. Here $c^1_{i-3}=c^2_{j-3}$ and $c^1_{i-2}=c^2_{j-2}$, forming buffers in the code meso-tile, represented as hatched squares.}
		\label{columns1}
	\end{subfigure}
	\hfill
	\begin{subfigure}[t]{0.3\textwidth}
		\centering
		\resizebox{\columnwidth}{!}{
			\begin{tikzpicture}[scale=1,,rotate=-90,yscale=-1]
			
			\clip[decorate,decoration={random steps,segment length=2pt,amplitude=3pt}] (-0.8, -0.5) rectangle (17.8,3);
			
			
			\begin{scope}[xshift=-3cm]
			\fill[color=black!20] (0,0) rectangle (1,1);
			\fill[color=black!20] (13,0) rectangle (14,1);
			
			\fill[color=black!10] (1,0) rectangle (9,1);
			
			\fill[color=black!5] (9,0) rectangle (13,1);
			
			\draw (0,0) rectangle (15,1);
			\foreach \k in {1,...,14}
			{
				\draw (\k,1)--(\k,0);
			};
			
			\draw[very thick] (1,1)--(1,0);
			\draw[very thick] (5,1)--(5,0);
			\draw[very thick] (9,1)--(9,0);
			\draw[very thick] (13,1)--(13,0);
			
			\draw (0.5,0.5) node[scale=0.75] {border};
			\draw (1.5,0.5) node[scale=0.6] {$(c_{i}^1)^{MN}$};
			\draw (2.5,0.5) node[scale=0.6] {$(c_{i+1}^1)^{MN}$};
			\draw (3.5,0.5) node {$...$};
			\draw (4.5,0.5) node[scale=0.6] {$(c_{i-1}^1)^{MN}$};
			\draw (5.5,0.5) node[scale=0.6] {$(c_{i}^1)^{MN}$};
			\draw (6.5,0.5) node[scale=0.6] {$(c_{i+1}^1)^{MN}$};
			\draw (7.5,0.5) node {$...$};
			\draw (8.5,0.5) node[scale=0.6] {$(c_{i-1}^1)^{MN}$};
			\draw (9.5,0.5) node[scale=0.6] {$(c_{j}^2)^{MN}$};
			\draw (10.5,0.5) node[scale=0.6] {$(c_{j+1}^2)^{MN}$};
			\draw (11.5,0.5) node {$...$};
			\draw (12.5,0.5) node[scale=0.6] {$(c_{j-1}^2)^{MN}$};
			\draw (13.5,0.5) node[scale=0.75] {border};
			\draw (14.5,0.5) node[scale=0.75] {code};
			\end{scope}
			
			\begin{scope}[xshift=12cm]
			\fill[color=black!20] (0,0) rectangle (1,1);
			\fill[color=black!20] (13,0) rectangle (14,1);
			
			\fill[color=black!10] (1,0) rectangle (9,1);
			
			\fill[color=black!5] (9,0) rectangle (13,1);
			
			\draw (0,0) rectangle (15,1);
			\foreach \k in {1,...,14}
			{
				\draw (\k,1)--(\k,0);
			};
			
			\draw[very thick] (1,1)--(1,0);
			\draw[very thick] (5,1)--(5,0);
			\draw[very thick] (9,1)--(9,0);
			\draw[very thick] (13,1)--(13,0);
			
			\draw (0.5,0.5) node[scale=0.75] {border};
			\draw (1.5,0.5) node[scale=0.6] {$(c_{i}^1)^{MN}$};
			\draw (2.5,0.5) node[scale=0.6] {$(c_{i+1}^1)^{MN}$};
			\draw (3.5,0.5) node {$...$};
			\draw (4.5,0.5) node[scale=0.6] {$(c_{i-1}^1)^{MN}$};
			\draw (5.5,0.5) node[scale=0.6] {$(c_{i}^1)^{MN}$};
			\draw (6.5,0.5) node[scale=0.6] {$(c_{i+1}^1)^{MN}$};
			\draw (7.5,0.5) node {$...$};
			\draw (8.5,0.5) node[scale=0.6] {$(c_{i-1}^1)^{MN}$};
			\draw (9.5,0.5) node[scale=0.6] {$(c_{j}^2)^{MN}$};
			\draw (10.5,0.5) node[scale=0.6] {$(c_{j+1}^2)^{MN}$};
			\draw (11.5,0.5) node {$...$};
			\draw (12.5,0.5) node[scale=0.6] {$(c_{j-1}^2)^{MN}$};
			\draw (13.5,0.5) node[scale=0.75] {border};
			\draw (14.5,0.5) node[scale=0.75] {code};
			\end{scope}

			\begin{scope}[xshift=-7.7cm,yshift=1cm]
			\fill[color=black!20] (0,0) rectangle (1,1);
			\fill[color=black!20] (13,0) rectangle (14,1);
			
			\fill[color=black!10] (1,0) rectangle (9,1);
			
			\fill[color=black!5] (9,0) rectangle (13,1);
			
			\draw (0,0) rectangle (15,1);
			\foreach \k in {1,...,14}
			{
				\draw (\k,1)--(\k,0);
			};
			
			\draw[very thick] (1,1)--(1,0);
			\draw[very thick] (5,1)--(5,0);
			\draw[very thick] (9,1)--(9,0);
			\draw[very thick] (13,1)--(13,0);
			
			\draw (0.5,0.5) node[scale=0.75] {border};
			\draw (1.5,0.5) node[scale=0.65] {$(c_{i}^1)^{MN}$};
			\draw (2.5,0.5) node[scale=0.65] {$(c_{i+1}^1)^{MN}$};
			\draw (3.5,0.5) node {$...$};
			\draw (4.5,0.5) node[scale=0.65] {$(c_{i-1}^1)^{MN}$};
			\draw (5.5,0.5) node[scale=0.65] {$(c_{i}^1)^{MN}$};
			\draw (6.5,0.5) node[scale=0.65] {$(c_{i+1}^1)^{MN}$};
			\draw (7.5,0.5) node {$...$};
			\draw (8.5,0.5) node[scale=0.65] {$(c_{i-1}^1)^{MN}$};
			\draw (9.5,0.5) node[scale=0.65] {$(c_{j}^2)^{MN}$};
			\draw (10.5,0.5) node[scale=0.65] {$(c_{j+1}^2)^{MN}$};
			\draw (11.5,0.5) node {$...$};
			\draw (12.5,0.5) node[scale=0.65] {$(c_{j-1}^2)^{MN}$};
			\draw (13.5,0.5) node[scale=0.75] {border};
			\draw (14.5,0.5) node[scale=0.75] {code};
			\end{scope}
			
			\begin{scope}[xshift=7.3cm,yshift=1cm]
			\fill[color=black!20] (0,0) rectangle (1,1);
			\fill[color=black!20] (13,0) rectangle (14,1);
			
			\fill[color=black!10] (1,0) rectangle (9,1);
			
			\fill[color=black!5] (9,0) rectangle (13,1);
			
			\draw (0,0) rectangle (15,1);
			\foreach \k in {1,...,14}
			{
				\draw (\k,1)--(\k,0);
			};
			
			\draw[very thick] (1,1)--(1,0);
			\draw[very thick] (5,1)--(5,0);
			\draw[very thick] (9,1)--(9,0);
			\draw[very thick] (13,1)--(13,0);
			
			\draw (0.5,0.5) node[scale=0.75] {border};
			\draw (1.5,0.5) node[scale=0.65] {$(c_{i}^1)^{MN}$};
			\draw (2.5,0.5) node[scale=0.65] {$(c_{i+1}^1)^{MN}$};
			\draw (3.5,0.5) node {$...$};
			\draw (4.5,0.5) node[scale=0.65] {$(c_{i-1}^1)^{MN}$};
			\draw (5.5,0.5) node[scale=0.65] {$(c_{i}^1)^{MN}$};
			\draw (6.5,0.5) node[scale=0.65] {$(c_{i+1}^1)^{MN}$};
			\draw (7.5,0.5) node {$...$};
			\draw (8.5,0.5) node[scale=0.65] {$(c_{i-1}^1)^{MN}$};
			\draw (9.5,0.5) node[scale=0.65] {$(c_{j}^2)^{MN}$};
			\draw (10.5,0.5) node[scale=0.65] {$(c_{j+1}^2)^{MN}$};
			\draw (11.5,0.5) node {$...$};
			\draw (12.5,0.5) node[scale=0.65] {$(c_{j-1}^2)^{MN}$};
			\draw (13.5,0.5) node[scale=0.75] {border};
			\draw (14.5,0.5) node[scale=0.75] {code};
			\end{scope}
			
			\end{tikzpicture}
		}
		\caption{Faulty alignment of adjacent columns.}
		\label{fig:aligned}
	\end{subfigure}
	\caption{The generic construction.}
	\label{generic}
\end{figure}

We say that two legally adjacent columns are \emph{aligned} if they are subdivided into ordered macro-slices exactly on the same lines. We say that two adjacent and aligned columns are \emph{synchronized} if any $(i,j)$-ordered macro-slice of the first one is followed by a $(i+1,j+1)$-ordered macro-slice in the second one.

\begin{proposition}
	\label{propapproxaligned}
	In this construction, two legally adjacent columns are aligned up to a vertical translation of size at most $2 |C^1|-1$ of one of the columns.
\end{proposition}

\begin{proof}
	If two columns, call them $K_1$ and $K_2$, can be legally juxtaposed such that they are not aligned even when vertically shifted by $2 |C^1|-1$ elements, it means that one of the border meso-slices of $K_1$ has at least $2 |C^1|$ vertically consecutive elements that are horizontally followed by something that is not a border meso-slice in $K_2$ (see \cref{fig:aligned}). Since $2|C^1| < MN$ which is the length of a meso-slice, at least $|C^1|$ successive elements among the ones of the border meso-slice in $K_1$ are horizontally followed by elements that are part of the same meso-slice in $K_2$. If this is a code meso-slice, simply consider the other border meso-slice of $K_1$ (the first you can find, above or below, before repeating the pattern cyclically): that one must be in contact with a $c^1_i$ or $c^2_j$ meso-slice instead. 
	Either way, we obtain that a border meso-slice has at least $|C^1|$ successive elements that are horizontally followed by some $t$ meso-slice made of a single element $t$. Hence if we suppose that juxtaposing $K_1$ and $K_2$ this way is legal, it means that in $H$ all the elements of $C^1$ lead to $t$, i.e $t$ is an attractive vertex. Either this is forbidden, or the "reverse" reasoning where we focus on the borders of $K_2$ proves that there is also an element $p$ used in a $C^i, i \in \{1,2\}$ slice of $K_1$ that leads to every element of $C^1$; that is, $C^1$ has a repulsive vertex in $C^1 \cup C^2$. Condition C forbids any graph that had both, hence we reach a contradiction. We obtain the proposition we announced.
\end{proof}

\begin{proposition}
	\label{propsync}
	In this construction, two legally adjacent columns are always aligned and synchronized.
\end{proposition}

\begin{proof}
	\cref{propapproxaligned} states that two adjacent columns $K_1$ and $K_2$ are always, in some sense, approximately aligned up to a vertical translation of size at most $2 |C^1|-1$. If the two columns are slightly shifted still, then any meso-slice of the $C^1$ slices of $K_1$ (such a meso-slice consists only of the repetition of some $c^1_i$) is horizontally followed by two different meso-slices in $K_2$. Being different, at least one of them is not $c^1_{i+1}$ but some $c^1_{i+k}, k \in \{2, ..., |C^1|\}$. This is true with the same $k$ for all values of $i$, notably because an ordered macro-slice contains two successive $C^1$ slices, so all meso-slices representing $c^1_i$ are repeated twice -- there is no problem with the meso-slices at the extremities. We obtain something that contradicts our assumption that $C^1$ has no uniform shortcut. Hence there is no vertical shift at all between two consecutive columns. Thus our construction ensures that two adjacent columns are always aligned.
	
	It is easy to see that a meso-slice made only of $c^1_i$ in column $K_1$ is horizontally followed, because the columns are aligned, by a meso-slice made only of $c^1_{i+k}$ in column $K_2$, for some $k \in \{0, \dots, |C^1|-1\}$. This $k$ is once again independent of the $i$ because inside a macro-slice, meso-slices respect the order of cycle $C^1$. But because $C^1$ has no uniform shortcut, we must have $k=1$. The reasoning is the same for the $C^2$ slice, and we use the fact that $C^2$ has no shortcut either. Hence our columns are synchronized.
\end{proof}

With these properties, we have ensured that our structure is rigid: our ordered macro-slices are aligned and synchronized. The last fact to check is the transmission of information, represented by the following proposition:

\begin{proposition}
	\label{proptransmission}
	In this construction, an $(i,j)$-ordered $(k,l)$-coding macro-slice is horizontally followed by an $(i+1,j+1)$-ordered $(k,l)$-coding macro-slice, except in two situations:
	\begin{itemize}
		\item if $c^1_i = c^2_j$, we can have a $(k,l)$-coding macro-slice followed by a $(k^\prime,l)$-coding macro-slice;
		
		\item if $c^1_i \neq c^2_j$ but $c^1_{i-1} = c^2_{j-1}$, we can have a $(k,l)$-coding macro-slice followed by a $(k,l^\prime)$-coding macro-slice.
	\end{itemize}
\end{proposition}

\begin{proof}
	The exceptions are due to an earlier remark: if $c^1_i = c^2_j$, then the code meso-slice contains no main-coded tile, so its value of $k$ is actually irrelevant. Similarly, if $c^1_i \neq c^2_j$ but $c^1_{i-1} = c^2_{j-1}$, the code meso-slice contains no side-coded tile, so its value of $l$ is irrelevant.
	
	For the rest of the proof, it is already clear from \cref{propsync} that an $(i,j)$-ordered macro-slice is horizontally followed by an $(i+1,j+1)$-ordered macro-slice; the only part left to study is the coding part.
	
	Now, consider two horizontally adjacent coding micro-slices. By synchronicity, one is an $(i,j)$ micro-slice, and the other an $(i+1,j+1)$ micro-slice. Since $\mathcal{G}(H)$ verifies condition C, and particularly has no cross-bridge, we cannot have both an edge $(c^1_i,c^2_{j+1}) \in \vec{E}$ and an edge $(c^2_j,c^1_{i+1}) \in \vec{E}$, except if one of the two micro-slices is a buffer. Therefore, either one of them is a buffer, or they both are $p$-coding for the same $p$.
	
	This is enough to prove the proposition.
\end{proof}

\begin{figure}[t]
	\begin{subfigure}[t]{0.4\textwidth}
		\centering
		\scalebox{1.5}{
			\begin{tikzpicture}
			\node[draw, circle] (b) {$b$};
			\node[draw, circle] (a) [below left = 0.5cm and 0.25cm of b] {$a$};
			\node[draw, circle] (c) [below right = 0.5cm and 0.25cm of b] {$c$};
			
			\draw[->, >=latex] (a) to[bend left=20] (b);
			\draw[->, >=latex] (b) to[bend left=20] (c);
			\draw[->, >=latex] (c) to[bend left=20] (a);
			
			\draw[->, >=latex] (c) to[bend left=20] (b);
			
			\draw[->, >=latex] (c) to [out=330,in=300,looseness=12] (c);
			\end{tikzpicture}
		}
	\end{subfigure}
	\hfill
	\begin{subfigure}[t]{0.4\textwidth}
		\centering
		\scalebox{0.6}{
			\begin{tikzpicture}
			\clip[decorate,decoration={random steps,segment length=5pt,amplitude=2pt}] (-0.1, 0.1) rectangle (6.1,-9.1);
			
			\begin{scope}[xshift=-1cm]
			\draw (0,0) rectangle (3,-3);
			\draw (3,0) rectangle (6,-3);
			\draw (6,0) rectangle (9,-3);
			
			\draw[fill=yellow, rounded corners, inner sep=0.3mm] (0,-2) rectangle (2,-3);
			\draw[fill=yellow, rounded corners, inner sep=0.3mm] (3,0) rectangle (5,-1);
			\draw[fill=yellow, rounded corners, inner sep=0.3mm] (6,-1) rectangle (8,-2);
			
			\fill[pattern=north west lines, pattern color=black!30] (2,0) rectangle (3,-3);
			\fill[pattern=north west lines, pattern color=black!30] (5,0) rectangle (6,-3);
			\fill[pattern=north west lines, pattern color=black!30] (8,0) rectangle (9,-3);
			
			\draw (0.5,-0.5) node {$a$};
			\draw (0.5,-1.5) node {$a$};
			\draw (0.5,-2.5) node {$c$};
			\draw (1.5,-0.5) node {$b$};
			\draw (1.5,-1.5) node {$b$};
			\draw (1.5,-2.5) node {$c$};
			\draw (2.5,-0.5) node {$c$};
			\draw (2.5,-1.5) node {$c$};
			\draw (2.5,-2.5) node {$c$};
			
			\draw (3.5,-0.5) node {$c$};
			\draw (3.5,-1.5) node {$a$};
			\draw (3.5,-2.5) node {$a$};
			\draw (4.5,-0.5) node {$c$};
			\draw (4.5,-1.5) node {$b$};
			\draw (4.5,-2.5) node {$b$};
			\draw (5.5,-0.5) node {$c$};
			\draw (5.5,-1.5) node {$c$};
			\draw (5.5,-2.5) node {$c$};
			
			\draw (6.5,-0.5) node {$a$};
			\draw (6.5,-1.5) node {$c$};
			\draw (6.5,-2.5) node {$a$};
			\draw (7.5,-0.5) node {$b$};
			\draw (7.5,-1.5) node {$c$};
			\draw (7.5,-2.5) node {$b$};
			\draw (8.5,-0.5) node {$c$};
			\draw (8.5,-1.5) node {$c$};
			\draw (8.5,-2.5) node {$c$};
			\end{scope}

			\begin{scope}[xshift=-2cm,yshift=-3cm]
			\draw (0,0) rectangle (3,-3);
			\draw (3,0) rectangle (6,-3);
			\draw (6,0) rectangle (9,-3);
			
			\draw[fill=yellow, rounded corners, inner sep=0.3mm] (0,-2) rectangle (2,-3);
			\draw[fill=yellow, rounded corners, inner sep=0.3mm] (3,0) rectangle (5,-1);
			\draw[fill=yellow, rounded corners, inner sep=0.3mm] (6,-1) rectangle (8,-2);
			
			\fill[pattern=north west lines, pattern color=black!30] (2,0) rectangle (3,-3);
			\fill[pattern=north west lines, pattern color=black!30] (5,0) rectangle (6,-3);
			\fill[pattern=north west lines, pattern color=black!30] (8,0) rectangle (9,-3);
			
			\draw (0.5,-0.5) node {$a$};
			\draw (0.5,-1.5) node {$a$};
			\draw (0.5,-2.5) node {$c$};
			\draw (1.5,-0.5) node {$b$};
			\draw (1.5,-1.5) node {$b$};
			\draw (1.5,-2.5) node {$c$};
			\draw (2.5,-0.5) node {$c$};
			\draw (2.5,-1.5) node {$c$};
			\draw (2.5,-2.5) node {$c$};
			
			\draw (3.5,-0.5) node {$c$};
			\draw (3.5,-1.5) node {$a$};
			\draw (3.5,-2.5) node {$a$};
			\draw (4.5,-0.5) node {$c$};
			\draw (4.5,-1.5) node {$b$};
			\draw (4.5,-2.5) node {$b$};
			\draw (5.5,-0.5) node {$c$};
			\draw (5.5,-1.5) node {$c$};
			\draw (5.5,-2.5) node {$c$};
			
			\draw (6.5,-0.5) node {$a$};
			\draw (6.5,-1.5) node {$c$};
			\draw (6.5,-2.5) node {$a$};
			\draw (7.5,-0.5) node {$b$};
			\draw (7.5,-1.5) node {$c$};
			\draw (7.5,-2.5) node {$b$};
			\draw (8.5,-0.5) node {$c$};
			\draw (8.5,-1.5) node {$c$};
			\draw (8.5,-2.5) node {$c$};
			\end{scope}

			\begin{scope}[xshift=-3cm,yshift=-6cm]
			\draw (0,0) rectangle (3,-3);
			\draw (3,0) rectangle (6,-3);
			\draw (6,0) rectangle (9,-3);
			
			\draw[fill=yellow, rounded corners, inner sep=0.3mm] (0,-2) rectangle (2,-3);
			\draw[fill=yellow, rounded corners, inner sep=0.3mm] (3,0) rectangle (5,-1);
			\draw[fill=yellow, rounded corners, inner sep=0.3mm] (6,-1) rectangle (8,-2);
			
			\fill[pattern=north west lines, pattern color=black!30] (2,0) rectangle (3,-3);
			\fill[pattern=north west lines, pattern color=black!30] (5,0) rectangle (6,-3);
			\fill[pattern=north west lines, pattern color=black!30] (8,0) rectangle (9,-3);
			
			\draw (0.5,-0.5) node {$a$};
			\draw (0.5,-1.5) node {$a$};
			\draw (0.5,-2.5) node {$c$};
			\draw (1.5,-0.5) node {$b$};
			\draw (1.5,-1.5) node {$b$};
			\draw (1.5,-2.5) node {$c$};
			\draw (2.5,-0.5) node {$c$};
			\draw (2.5,-1.5) node {$c$};
			\draw (2.5,-2.5) node {$c$};
			
			\draw (3.5,-0.5) node {$c$};
			\draw (3.5,-1.5) node {$a$};
			\draw (3.5,-2.5) node {$a$};
			\draw (4.5,-0.5) node {$c$};
			\draw (4.5,-1.5) node {$b$};
			\draw (4.5,-2.5) node {$b$};
			\draw (5.5,-0.5) node {$c$};
			\draw (5.5,-1.5) node {$c$};
			\draw (5.5,-2.5) node {$c$};
			
			\draw (6.5,-0.5) node {$a$};
			\draw (6.5,-1.5) node {$c$};
			\draw (6.5,-2.5) node {$a$};
			\draw (7.5,-0.5) node {$b$};
			\draw (7.5,-1.5) node {$c$};
			\draw (7.5,-2.5) node {$b$};
			\draw (8.5,-0.5) node {$c$};
			\draw (8.5,-1.5) node {$c$};
			\draw (8.5,-2.5) node {$c$};

			\draw (9,0) rectangle (12,-3);
			
			\draw[fill=yellow, rounded corners, inner sep=0.3mm,thick] (9,-1) rectangle (11,-2);
			
			\fill[pattern=north west lines, pattern color=black!30] (11,0) rectangle (12,-3);
			
			\draw (9.5,-0.5) node {$a$};
			\draw (9.5,-1.5) node {$c$};
			\draw (9.5,-2.5) node {$a$};
			\draw (10.5,-0.5) node {$b$};
			\draw (10.5,-1.5) node {$c$};
			\draw (10.5,-2.5) node {$b$};
			\draw (11.5,-0.5) node {$c$};
			\draw (11.5,-1.5) node {$c$};
			\draw (11.5,-2.5) node {$c$};
			\end{scope}
			
			\draw[thick, rounded corners, inner sep=0.3mm] (3,0) rectangle (4,-9);
			
			\end{tikzpicture}
		}
	\end{subfigure}
	
	\caption{A Rauzy graph and several associated code meso-slices for $|\tau| = 3$. Here are horizontally successively encoded $\tau_3, \tau_1, \tau_2$ and $\tau_2$, the number being indicated by the location of the line of $c$'s. One can check that $\tau_1$ can be located left to $\tau_2$ in the encoding, by using vertical constraints only, as depicted by the bold rectangle with rounded corners.}
	\label{codingblock}
\end{figure}

In the end, we proved that if we were able to find two cycles $C^1$ and $C^2$ complying with condition $C$, they would be enough to build the construction we desire: the root of a full shift on $N$ elements.
Indeed, take $Z$ the clopen made of all the configurations with, at position $(0,0)$, the bottom of an $(i,j)$-ordered macro-slice with $c^1_{i-1} = c^2_{j-1}$ but $c^1_{i} \neq c^2_{j}$. Suppose that macro-slice is $(k,l)$-coding (with an irrelevant $l$). Map it, along with the $M-1$ that follow horizontally (hence, map a $M \times KMN$ rectangle) on $y_{(0,0)} = \tau_k$. By mapping the whole configuration similarly, $Z$ is homeomorphic to $Y$ with the required properties: $X_{H,V_Y}$ is therefore a $(M,KMN)$th root of $\tau^{\Z^2}$.

Then, to encode only the configurations that are valid in $W$ a Wang shift, we forbid the following additional vertical patterns:
\begin{itemize}
	\item the code meso-slices that would contain both $\tau_k$ as its main-coded tile and $\tau_l$ as its side-coded tile, if $\tau_k$ cannot be horizontally followed by $\tau_l$ in $W$;
	\item and the vertical succession of two ordered macro-slices that would contain code meso-slices with two main-coded tiles that cannot be vertically successive in $W$.
\end{itemize}
With this, we proved \cref{propcondC}: the set of vertical conditions obtained, that define a one-dimensional SFT $V_W$, is such that $X_{H,V_W}$ is a $(M,KMN)$th root of $W$.

\subsection{Summary of the general construction for one strongly connected component}

We suppose that $H \subset \A^\Z$ is a one-dimensional nearest-neighbor SFT such that its Rauzy graph does not verify condition $D$ and is made of only one SCC.

Note that $\mathcal{G}(H)$, since it does not verify condition $D$, contains at least one loopless vertex, and one unidirectional edge.

\begin{table}[b]
	\caption{Table of the main cases, each of them illustrated with an example (the $C^1$ on which we perform the generic construction is the main cycle indicated; the $C^2$ is in red).}
	\begin{center}
		\resizebox{\columnwidth}{!}{
			\begin{tabular}{|c|c|c||c|c||c|c|c|}
				\hline
				\multicolumn{3}{|c||}{Loops} & \multicolumn{5}{c|}{No loop}\\
				\hline
				&  &  & \multicolumn{2}{c||}{Bidirectional edges} & \multicolumn{3}{c|}{No bidirectional edge}\\
				\cline{4-8}
				
				\begin{tikzpicture}[scale=0.4]
				\def \radius {1.5cm};
				
				\node[draw, circle] (a) at (0:\radius) {};
				\node[draw, circle] (b) at (72:\radius) {};
				\node[draw, circle] (c) at (144:\radius) {};
				\node[draw, circle,scale=0.7] (d) at (216:\radius) {$v$};
				\node[draw, circle,scale=0.7] (e) at (288:\radius) {$w$};
				
				\draw (0,0) node {$C^1$};
				
				\draw[->, >=latex] (a) to[bend right=15] (b);
				\draw[->, >=latex] (b) to[bend right=15] (c);
				\draw[->, >=latex] (c) to[bend right=15] (d);
				\draw[->, >=latex] (d) to[bend right=15] (e);
				\draw[->, >=latex] (e) to[bend right=15] (a);
				
				\draw[->, >=latex] (b) to[bend right=15] (a);
				\draw[->, >=latex] (c) to[bend right=15] (b);
				
				\draw[->, >=latex,color=red] (e) to[out=308,in=268,looseness=6] (e);
				\end{tikzpicture}
				
				&
				
				\begin{tikzpicture}[scale=0.4]
				\def \radius {1.5cm};
				
				\node[draw, circle,scale=0.7] (a) at (0:\radius) {$v$};
				\node[draw, circle,scale=0.7] (b) at (72:\radius) {$w$};
				\node[draw, circle] (c) at (144:\radius) {};
				\node[draw, circle] (d) at (216:\radius) {};
				\node[draw, circle,scale=0.7] (e) at (288:\radius) {$u$};
				
				\draw (0,0) node {$C^1$};
				
				\draw[->, >=latex,color=red] (a) to[bend right=15] (b);
				\draw[->, >=latex] (b) to[bend right=15] (c);
				\draw[->, >=latex] (c) to[bend right=15] (d);
				\draw[->, >=latex] (d) to[bend right=15] (e);
				\draw[->, >=latex] (e) to[bend right=15] (a);
				
				\draw[->, >=latex,color=red] (b) to[bend right=15] (a);
				\draw[->, >=latex] (c) to[bend right=15] (b);
				
				\draw[->, >=latex] (b) to[out=92,in=52,looseness=6] (b);
				\end{tikzpicture}
				
				&
				
				\begin{tikzpicture}[scale=0.4]
				\def \radius {1.5cm};
				
				\node[draw, circle,scale=0.7] (a) at (0:\radius) {$v$};
				\node[draw, circle,scale=0.7] (b) at (72:\radius) {$a$};
				\node[draw, circle] (c) at (144:\radius) {};
				\node[draw, circle] (d) at (216:\radius) {};
				\node[draw, circle,scale=0.7] (e) at (288:\radius) {$u$};
				
				\draw (0,0) node {$C^1$};
				
				\draw[->, >=latex] (a) to[bend right=15] (b);
				\draw[->, >=latex] (b) to[bend right=15] (c);
				\draw[->, >=latex] (c) to[bend right=15] (d);
				\draw[->, >=latex] (d) to[bend right=15] (e);
				\draw[->, >=latex] (e) to[bend right=15] (a);
				
				\draw[->, >=latex] (b) to[bend right=15] (a);
				\draw[->, >=latex] (c) to[bend right=15] (b);
				
				\draw[->, >=latex,color=red] (a) to[out=20,in=-20,looseness=6] (a);
				\draw[->, >=latex] (c) to[out=164,in=124,looseness=6] (c);
				\draw[->, >=latex] (d) to[out=236,in=196,looseness=6] (d);
				\draw[->, >=latex] (e) to[out=308,in=268,looseness=6] (e);
				\end{tikzpicture}
				
				&
				
				\begin{tikzpicture}[scale=0.4]
				\def \radius {1.5cm};
				
				\node[draw, circle,scale=0.7] (a) at (0:\radius) {$v$};
				\node[draw, circle,scale=0.7] (b) at (72:\radius) {$w$};
				\node[draw, circle] (c) at (144:\radius) {};
				\node[draw, circle] (d) at (216:\radius) {};
				\node[draw, circle,scale=0.7] (e) at (288:\radius) {$u$};
				
				\draw (0,0) node {$C^1$};
				
				\draw[->, >=latex,color=red] (a) to[bend right=15] (b);
				\draw[->, >=latex] (b) to[bend right=15] (c);
				\draw[->, >=latex] (c) to[bend right=15] (d);
				\draw[->, >=latex] (d) to[bend right=15] (e);
				\draw[->, >=latex] (e) to[bend right=15] (a);
				
				\draw[->, >=latex,color=red] (b) to[bend right=15] (a);
				\draw[->, >=latex] (c) to[bend right=15] (b);
				\end{tikzpicture}
				
				&
				
				\begin{tikzpicture}[scale=0.4,rotate=-90]
				\def \radius {1.5cm};
				
				\node[draw, circle,scale=0.7] (a) at (0:\radius) {$v$};
				\node[draw, circle] (b) at (72:\radius) {};
				\node[draw, circle] (c) at (144:\radius) {};
				\node[draw, circle] (d) at (216:\radius) {};
				\node[draw, circle] (e) at (288:\radius) {};
				
				\node[draw, circle,scale=0.7] (f) at (3.5,0) {$w$};
				
				\draw (0,0) node {$C^1$};
				
				\draw[->, >=latex] (a) to[bend right=15] (b);
				\draw[->, >=latex] (b) to[bend right=15] (c);
				\draw[->, >=latex] (c) to[bend right=15] (d);
				\draw[->, >=latex] (d) to[bend right=15] (e);
				\draw[->, >=latex] (e) to[bend right=15] (a);
				
				\draw[->, >=latex,color=red] (a) to[bend right=15] (f);
				\draw[->, >=latex,color=red] (f) to[bend right=15] (a);
				\end{tikzpicture}
				
				&
				
				\begin{tikzpicture}[scale=0.4,rotate=-90]
				\def \radius {1.5cm};
				
				\node[draw, circle] (a) at (0:\radius) {};
				\node[draw, circle] (b) at (72:\radius) {};
				\node[draw, circle] (c) at (144:\radius) {};
				\node[draw, circle] (d) at (216:\radius) {};
				\node[draw, circle] (e) at (288:\radius) {};
				
				\node[draw, circle] (f) at (-2.5,0.8) {};
				\node[draw, circle] (g) at (-2.5,-0.8) {};
				
				\draw (0,0) node {$C^1$};
				
				\draw[->, >=latex,color=red] (a) to[bend right=15] (b);
				\draw[->, >=latex,color=red] (b) to[bend right=15] (c);
				\draw[->, >=latex] (c) to[bend right=15] (d);
				\draw[->, >=latex,color=red] (d) to[bend right=15] (e);
				\draw[->, >=latex,color=red] (e) to[bend right=15] (a);
				
				\draw[->, >=latex,color=red] (c) to[bend right=15] (f);
				\draw[->, >=latex,color=red] (f) to[bend right=15] (g);
				\draw[->, >=latex,color=red] (g) to[bend right=15] (d);
				\end{tikzpicture}
				
				&
				
				\begin{tikzpicture}[scale=0.4,rotate=-90]
				\def \radius {1.5cm};
				
				\node[draw, circle] (a) at (0:\radius) {};
				\node[draw, circle] (b) at (72:\radius) {};
				\node[draw, circle] (c) at (144:\radius) {};
				\node[draw, circle] (d) at (216:\radius) {};
				\node[draw, circle] (e) at (288:\radius) {};
				
				\node[draw, circle] (f) at (-2.5,1) {};
				\node[draw, circle] (g) at (-2.5,-1) {};
				
				\draw (0,0) node {$C^1$};
				
				\draw[->, >=latex,color=red] (a) to[bend right=15] (b);
				\draw[->, >=latex] (b) to[bend right=15] (c);
				\draw[->, >=latex] (c) to[bend right=15] (d);
				\draw[->, >=latex] (d) to[bend right=15] (e);
				\draw[->, >=latex,color=red] (e) to[bend right=15] (a);
				
				\draw[->, >=latex,color=red] (b) to[bend right=15] (f);
				\draw[->, >=latex,color=red] (f) to[bend right=15] (g);
				\draw[->, >=latex,color=red] (g) to[bend right=15] (e);
				
				\draw[->, >=latex] (c) to[bend right=15] (g);
				\end{tikzpicture}
				
				&
				
				\begin{tikzpicture}[scale=0.4]
				\def \radius {1.5cm};
				
				\node[draw, circle] (a) at (0:\radius) {};
				\node[draw, circle] (b) at (72:\radius) {};
				\node[draw, circle] (c) at (144:\radius) {};
				\node[draw, circle] (d) at (216:\radius) {};
				\node[draw, circle] (e) at (288:\radius) {};
				
				\node[draw, circle] (f) at (2.5,-0.8) {};
				\node[draw, circle] (g) at (2.5,0.8) {};
				
				\draw (0,0) node {$C^1$};
				
				\draw[->, >=latex] (a) to[bend right=15] (b);
				\draw[->, >=latex] (b) to[bend right=15] (c);
				\draw[->, >=latex] (c) to[bend right=15] (d);
				\draw[->, >=latex] (d) to[bend right=15] (e);
				\draw[->, >=latex] (e) to[bend right=15] (a);
				
				\draw[->, >=latex,color=red] (a) to[bend right=15] (f);
				\draw[->, >=latex,color=red] (f) to[bend right=15] (g);
				\draw[->, >=latex,color=red] (g) to[bend right=15] (a);
				\end{tikzpicture}
				
				\\
				
				\hline
				
				Case 1.1 & Case 1.2 & Case 1.3 & Case 2.1 & Case 2.2 & Case 3.1 & Case 3.2 & Case 3.3\\
				
				\hline
			\end{tabular}
		}
	\end{center}
	\label{tablecases}
\end{table}

The idea of the proof of \cref{th:root} is to classify the possible graphs into various cases. In each case, one has a standard procedure to find convenient $C^1$ and $C^2$ inside any graph to perform the generic construction from \cref{subsec:core}. Of course, for some specific cases, we will not meet condition $C$ even if $H$ does not verify condition $D$. However, we will punctually adapt the generic construction to these specificities.

The division into cases is presented in a disjunctive fashion, see \cref{tablecases}:

Is there a loop on a vertex?
\begin{itemize}
	\item If YES: Is there a unidirectional edge $(v,w) \in \vec{E}$ so that $v$ is loopless and $w$ has a loop (or the reverse, which is similar)?
	\begin{itemize}
		\item If YES: This is Case 1.1. We can find $C^1$ and $C^2$ that check condition $C$ with the exception of the possible presence of both an attractive and a repulsive vertices. However, \cref{propapproxaligned} is still verified, because by choosing the smallest possible cycle containing such $v$ and $w$, $v$ has in-degree $1$, a property that allows for an easy synchronization.
		
		\item If NO: Do unidirectional edges have loopless vertices?
		\begin{itemize}
			\item If YES: This is Case 1.2. We can find $C^1$ and $C^2$ that check condition $C$, possibly by reducing to a situation encountered in case 2.2.
			\item If NO: This is Case 1.3. This one generates some exceptional graphs with 4 or 5 vertices that do not check condition $C$ and must be treated separately. However, technical considerations prove that our generic construction still works.
		\end{itemize}
	\end{itemize}
	\item If NO: Is there a bidirectional edge?
	\begin{itemize}
		\item If YES: Is there a cycle of size at least 3 that contains a bidirectional edge?
		\begin{itemize}
			\item If YES: This is Case 2.1. We can find $C^1$ and $C^2$ that verify condition $C$ rather easily.
			\item If NO: This is Case 2.2, in which checking condition $C$ is also easy.
		\end{itemize}
		\item If NO: Is there a minimal cycle with a path between two \emph{different} elements of it, say $c^1_0$ and $c^1_k$, that does not belong to the cycle?
		\begin{itemize}
			\item If YES: Can we find such a path of length different from $k$?
			\begin{itemize}
				\item If YES: This is Case 3.1, a rather tedious case, but we can find cycles $C^1$ and $C^2$ that verify condition $C$ nonetheless.
				\item If NO: This is Case 3.2, which relies heavily on the fact that $\mathcal{G}(H)$ is not of state-split cycle type to find cycles that verify condition $C$.
			\end{itemize}
			\item If NO: This is Case 3.3, an easy case to find cycles that verify condition $C$.
		\end{itemize}
	\end{itemize}
\end{itemize}

In the subsections that follow, we define two cycles $C^1$ and $C^2$ trying to fit condition C as much as possible, and proving when it fails that the propositions from \cref{subsec:core} still hold nonetheless. We name $C^1$ vertices $c^1_i$ and $C^2$ vertices $c^2_j$, with $i \in \{0,\dots,|C^1|-1\}$ and $j \in \{0,\dots,|C^2|-1\}$.

\subsection{Case 1}
\label{subsec:case1}

We suppose that $\mathcal{G}(H)$ contains a loop.

\textbf{Case 1.1:} we can find a unidirectional edge so that the first vertex is loopless and the second has a loop (or the opposite, for which the construction is similar and omitted).

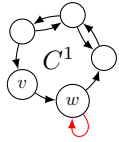
\begin{center}
	\begin{tikzpicture}[scale=0.4]
		\def \radius {1.5cm};
		
		\node[draw, circle] (a) at (0:\radius) {};
		\node[draw, circle] (b) at (72:\radius) {};
		\node[draw, circle] (c) at (144:\radius) {};
		\node[draw, circle,scale=0.7] (d) at (216:\radius) {$v$};
		\node[draw, circle,scale=0.7] (e) at (288:\radius) {$w$};
		
		\draw (0,0) node {$C^1$};
		
		\draw[->, >=latex] (a) to[bend right=15] (b);
		\draw[->, >=latex] (b) to[bend right=15] (c);
		\draw[->, >=latex] (c) to[bend right=15] (d);
		\draw[->, >=latex] (d) to[bend right=15] (e);
		\draw[->, >=latex] (e) to[bend right=15] (a);
		
		\draw[->, >=latex] (b) to[bend right=15] (a);
		\draw[->, >=latex] (c) to[bend right=15] (b);
		
		\draw[->, >=latex,color=red] (e) to[out=308,in=268,looseness=6] (e);
	\end{tikzpicture}
\end{center}

Take the shortest possible cycle containing such an edge, which exists since $\mathcal{G}(H)$ is strongly connected. Call $v$ the loopless vertex and $w$ the vertex with a loop. Naming that cycle $C^1$ with $c^1_0 = w$, and setting $C^2 = \{w\}$, we have to check that they fulfill the conditions of \cref{subsec:core}.

$(v,w)$ is unidirectional and $v$ is loopless. Note that no edge can go from $w$ to any vertex that is not $w$ or $c^1_1$, else we could find a shorter cycle with the same characteristics. Similarly, $v$ has in-degree $1$. Hence:
\begin{itemize}
	\item $|C^1| \geq 3$ since $(w,v) \notin \vec{E}$;
	
	\item $C^1$ and $C^2$ contain a good pair, that is $(c^1_1,w)$.
	
	\item $C^2$ has no uniform shortcut because it is made of one single vertex. $C^1$ has no uniform shortcut because of what precedes about $w$ and because $v$ is loopless;
	
	\item If there was a cross-bridge between $C^1$ and $C^2$, it would mean there are two edges $(c^1_i,w)$ and $(w,c^1_{i+1}) \in \vec{E}$ with $c^1_i \neq w \neq c^1_{i+1}$, which is also impossible because of what precedes about $w$;
	
	\item Here, there can actually be attractive and repulsive vertices for $C^1$, which endangers \cref{propapproxaligned}. However, the only vertex that has an edge going to $v$ is the previous one in the cycle $C^1$, call it $u := c^1_{-2}$. As such, in any column, the $v$ meso-slice must be next to the $u$ meso-slice of the previous column since no other block of $u$ of size $M N$ can be found in said previous column. Hence two consecutive columns are always aligned and we can make the generic construction work with no restriction on attractive and repulsive vertices: \cref{propapproxaligned} still holds, albeit for reasons different from the ones in \cref{subsec:core}.
\end{itemize}

\textbf{Case 1.2:} all unidirectional edges have loopless vertices.

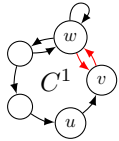
\begin{center}
	\begin{tikzpicture}[scale=0.4]
		\def \radius {1.5cm};
		
		\node[draw, circle,scale=0.7] (a) at (0:\radius) {$v$};
		\node[draw, circle,scale=0.7] (b) at (72:\radius) {$w$};
		\node[draw, circle] (c) at (144:\radius) {};
		\node[draw, circle] (d) at (216:\radius) {};
		\node[draw, circle,scale=0.7] (e) at (288:\radius) {$u$};
		
		\draw (0,0) node {$C^1$};
		
		\draw[->, >=latex,color=red] (a) to[bend right=15] (b);
		\draw[->, >=latex] (b) to[bend right=15] (c);
		\draw[->, >=latex] (c) to[bend right=15] (d);
		\draw[->, >=latex] (d) to[bend right=15] (e);
		\draw[->, >=latex] (e) to[bend right=15] (a);
		
		\draw[->, >=latex,color=red] (b) to[bend right=15] (a);
		\draw[->, >=latex] (c) to[bend right=15] (b);
		
		\draw[->, >=latex] (b) to[out=92,in=52,looseness=6] (b);
	\end{tikzpicture}
\end{center}

Since $\mathcal{G}(H)$ is in Case 1 and Subcase 1.2, it contains loops and its unidirectional edges have loopless vertices. Hence it contains bidirectional edges. Moreover, it cannot contain only bidirectional edges, else it would verify condition $D$. Hence $\mathcal{G}(H)$ contains both unidirectional and bidirectional edges, and by strong connectivity, we can find a (possibly self-intersecting) cycle containing both. Therefore we can find a unidirectional edge followed by a bidirectional edge.

We name $u, v, w$ three successive vertices in the graph so that $(u,v), (v,w)$ and $(w,v) \in \vec{E}$, and $(v,u) \notin \vec{E}$ (so $u$ and $v$ have no loop). Two situations are possible: either there is a path from $w$ to $u$ that does not go through $v$, and we obtain a cycle containing both a unidirectional and a bidirectional edge; or there is not. In that second case, we consider a path from $v$ to $u$:
\begin{itemize}
	\item either it contains no bidirectional edge, and the resulting graph is made of one cycle with unidirectional edges and the bidirectional edge $(v,w)$: this is treated just as Case 2.2 -- with the extra of $w$ having a loop, but this does not change the reasoning.
	\item or the path from $v$ to $u$ contains a bidirectional edge; then concatenated to $(u,v)$ it forms a cycle with both a unidirectional and a bidirectional edge.
\end{itemize}
Iteratively reducing the cycle obtained to the shortest one possible, we either end up in a situation that can be reduced to Case 2.2, or to a cycle similar to the figure above: with a unidirectional edge, followed by a bidirectional edge, and containing no shorter cycle that would fit.

Naming that cycle $C^1$ with $c^1_0 = v$, and defining $C^2 = \{v, w\}$, we have to check that they fulfill the conditions of the generic construction. Note that $w$ cannot lead to any vertex except $c^1_2$, $v$, and possibly $w$; else we could find a cycle shorter than $C^1$ that has the same properties. Note, also, that $v \neq c^1_2$, else we would reduce to Case 2.2. Hence:
\begin{itemize}
	\item $|C^1| \geq 3$;
	
	\item $C^1$ and $C^2$ have a good pair: $(c^1_2,v)$ is one.
	
	\item $C^2$ has no uniform shortcut because it is made of only two vertices and $v$ has no loop. $C^1$ has no uniform shortcut of length $0$ because $u$ and $v$ are loopless, and no uniform shortcut of length $-1$ because $(v,u) \notin \vec{E}$. With what precedes about $w$, there is no uniform shortcut at all;
	
	\item If there was a cross-bridge, we could have two cases:
	\begin{itemize}
		\item First is $(c^1_i,v)$ and $(w,c^1_{i+1}) \in \vec{E}$ with $c^1_i \neq w$ and $c^1_{i+1} \neq v$. Since $v$ has no loop, we deduce $c^1_i \neq v$, hence $c^1_{i+1} \neq w$. Also consider that $c^1_{i+1} = c^1_2$ would imply $c^1_i = w$, which is impossible. Therefore, with what we said on $w$, that kind of cross-bridge cannot happen.
		
		\item Second is $(c^1_i,w)$ and $(v,c^1_{i+1}) \in \vec{E}$ with $c^1_i \neq v$ and $c^1_{i+1} \neq w$. Since $v$ has no loop and $(v,u) \notin \vec{E}$, we also deduce $v \neq c^1_{i+1}$ and $u \neq c^1_{i+1}$. Then we can define a shorter cycle that is ${C^1}^\prime = ( v, c^1_{i+1}, c^1_{i+2}, ..., u )$. Since $(v,u) \notin \vec{E}$, ${C^1}^\prime$ has length at least $3$, contains at least one unidirectional edge, and is strictly shorter than $C^1$. Since $C^1$ is a minimal cycle having these properties and containing a bidirectional edge, ${C^1}^\prime$ must contain no bidirectional edge. Then ${C^1}^\prime$ is a cycle of unidirectional edges such that a bidirectional edge has one vertex in common with it (that is, $\{v, w\}$). This is an iterative reduction that should have already been performed to build $C^1$ and $C^2$, therefore it cannot happen here.
	\end{itemize}
	
	\item If there is an attractive vertex for $C^1$ located in $C^2$, then it is in particular in $C^1$ since $C^2 \subset C^1$. Since they have no loop, $u$ and $v$ can't be attractive or repulsive. If any other vertex than $w$ or the following vertex, call it $x$, was attractive, then it would allow for a direct edge from $w$ to that vertex, and so a fitting cycle strictly shorter than $C^1$ would exist, which is impossible. But if $w$ (resp. $x$) was attractive, in particular $(u,w) \in \vec{E}$ (resp. $(u,x) \in \vec{E}$). Since $w$ (resp. $x$) would have a loop because it also attracts itself, in this Case 1.2 we couldn't have a unidirectional edge $(u,w)$ (resp. $(u,x)$): necessarily $(w,u) \in \vec{E}$ (resp. $(x,u) \in \vec{E}$). The only possibility not to  cause a contradiction with the minimality of $C^1$ is that $C^1$ is already of length $3$ (resp. $4$).
	
	The length $3$ case is treated in \cref{subsubsec:length3}. The length $4$ case with $(u, v, w, x)$ and $x$ attractive is actually impossible, because $(u, v, x)$ reduces to a length $3$ case with the correct properties, but $C^1$ was supposed to be minimal.
\end{itemize}

\textbf{Case 1.3:} all unidirectional edges have vertices with loops.

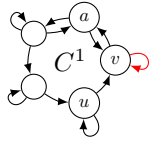
\begin{center}
	\begin{tikzpicture}[scale=0.4]
		\def \radius {1.5cm};
		
		\node[draw, circle,scale=0.7] (a) at (0:\radius) {$v$};
		\node[draw, circle,scale=0.7] (b) at (72:\radius) {$a$};
		\node[draw, circle] (c) at (144:\radius) {};
		\node[draw, circle] (d) at (216:\radius) {};
		\node[draw, circle,scale=0.7] (e) at (288:\radius) {$u$};
		
		\draw (0,0) node {$C^1$};
		
		\draw[->, >=latex] (a) to[bend right=15] (b);
		\draw[->, >=latex] (b) to[bend right=15] (c);
		\draw[->, >=latex] (c) to[bend right=15] (d);
		\draw[->, >=latex] (d) to[bend right=15] (e);
		\draw[->, >=latex] (e) to[bend right=15] (a);
		
		\draw[->, >=latex] (b) to[bend right=15] (a);
		\draw[->, >=latex] (c) to[bend right=15] (b);
		
		\draw[->, >=latex,color=red] (a) to[out=20,in=-20,looseness=6] (a);
		\draw[->, >=latex] (c) to[out=164,in=124,looseness=6] (c);
		\draw[->, >=latex] (d) to[out=236,in=196,looseness=6] (d);
		\draw[->, >=latex] (e) to[out=308,in=268,looseness=6] (e);
	\end{tikzpicture}
\end{center}

Since $\mathcal{G}(H)$ does not verify condition $D$, we can find a cycle with at least one loopless vertex and one unidirectional edge, that, in this specific case, may go through the same vertices twice since it is defined, strictly speaking, as the concatenation of one path from the unidirectional edge to the loopless vertex, and one path back, the two of them being able to intersect.

Define $C^1$ as the smallest cycle built that way. Call $u$ and $v$ the two successive vertices of the unidirectional edge, that is $(u,v) \in \vec{E}$ but $(v,u) \notin \vec{E}$. Note that $u$ and $v$ have a loop since we are in Case 1.3. Moreover, call $a$ the loopless vertex that was also used to build $C^1$. Finally, set $a = c^1_0, u = c^1_{i_0}, v = c^1_{i_0+1}$.

Setting $C^2 = \{v\}$, we have to check that they fulfill the conditions of the generic construction:
\begin{itemize}
	\item $|C^1| \geq 3$;
	
	\item $(c^1_{i_0+2},v)$ is a good pair since $C^1$ passes through $v$ only once in the cycle by construction.
	
	\item $C^2$ is made of only one vertex, hence it has no uniform shortcut. $C^1$ has no uniform shortcut of length $0$ since $a$ has no loop. It has no uniform shortcut of length $-1$ because $(v,u) \notin \vec{E}$. We study a hypothetical shortcut that would allow $(a,c^1_j) \in \vec{E}$: notice that $(c^1_j,a) \in \vec{E}$ since $a$ has no loop. $j \in \{2,...,|C^1|-2\}$ hence one can use either $(a,c^1_j)$ or $(c^1_j,a)$ to build a cycle shorter than $C^1$ containing $a$, $u$ and $v$, which would therefore be a cycle that would have all the required properties. This is impossible by minimality of $C^1$.
	
	\item There can be cross-bridges between $C^1$ and $C^2$. However, a subtle line of reasoning explained below shows that here, keeping the indices used previously, we only have to avoid the cross-bridges [$(c^1_i,v)$ and $(v,c^1_{i+1}) \in \vec{E}$ with $c^1_i \neq v \neq c^1_{i+1}$] with $i = i_0-1$ and $i = i_0+2$ for our construction to work -- that is, for the information to be correctly transmitted.
	
	The case $i = i_0-1$ is impossible because $(v,c^1_{i_0}) = (v,u) \notin \vec{E}$. If $(v,c^1_{i_0+3}) \in \vec{E}$ then we can use it to find a cycle shorter than $C^1$ that contains everything we want -- unless $c^1_{i_0+2} = a$.
	
	If $c^1_{i_0+2} = a$ then we redefine $C^2 = \{u\}$ with which all that we have proved can be adapted: we only have to avoid the cross-bridges [$(c^1_i,v)$ and $(v,c^1_{i+1}) \in \vec{E}$ with $c^1_i \neq v \neq c^1_{i+1}$] with $i = i_0-2$ and $i = i_0+1$. Once again, this is impossible except if we also have $c^1_{i_0-1} = a$. Then $C^1$ is made of only three elements and this case is solved in \cref{subsubsec:length3}.
	
	\item As we will see, there is only one possibility for $C^1$ to have both an attractive and a repulsive vertex. Since $C^2 \subset C^1$, it is enough to consider attractive and repulsive vertices for $C^1$ that are located in $C^1$. Let $t$ be an attractor located in $C^1$ for all elements of $C^1$. Notably, $(a,t) \in \vec{E}$. Since $a$ has no loop, this Case 1.3 causes $(t,a) \in \vec{E}$. Similarly, for $p$ a repulsive vertex, we have not only $(p,a) \in \vec{E}$, but also $(a,p) \in \vec{E}$. Hence the shortest cycle that meets all our requirements is $(a, p, u, v, t)$, so the only possibility for $C^1$ to have both that does not contradict its minimality is to be this precise cycle (with some of the vertices being possibly equal). This is a case we treat in \cref{subsubsec:case1.3}.
\end{itemize}

\textbf{Cross-bridges remark for Case 1.3:} We prove why we only need to avoid two cross-bridges to be sure that the information is entirely transmitted. The basics of Case 1.3 are: $(u,v) \in \vec{E}$, $(v,u) \notin \vec{E}$, $a$ loopless, $C^1$ is a possibly self-intersecting cycle made of the concatenation of a path from $a$ to $u$ and one from $v$ to $a$, all unidirectional edges have a loop, $C^2 = \{v\}$.

Any diagonal region that contains the coding of a tile, delimited by buffers and possibly border slices above or below, does encode exactly one tile, see \cref{fig:complexcode}. Indeed, each of its vertical slices contains at least one of the elements among $\{c^1_{i_0}, c^1_{i_0+2}\}$ (since the buffer is given by $c^1_{i_0+1}$), and these are always part of a micro-slice that encodes something. By construction, the whole slice encodes the same thing, vertically.

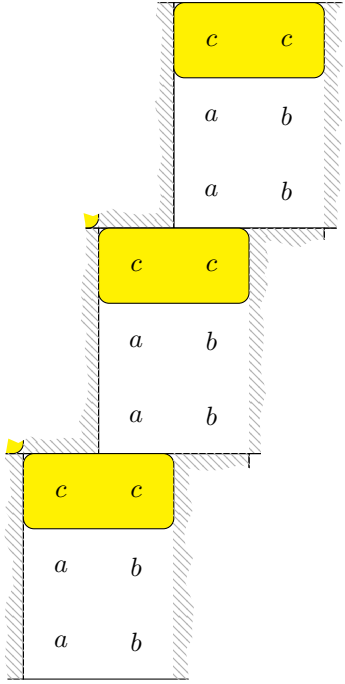
\begin{figure}[tb]
	\centering
	\scalebox{1}{
		\begin{tikzpicture}
			\clip[decorate,decoration={random steps,segment length=5pt,amplitude=2pt}] (-0.2, -9.2) -- (-0.2,-5.8) -- (0.8,-5.8) -- (0.8,-2.8) -- (1.8,-2.8) -- (1.8,0.2) -- (4.2,0.2) -- (4.2,-3.2) -- (3.2,-3.2) -- (3.2,-6.2) -- (2.2,-6.2) -- (2.2,-9.2) -- (-0.2,-9.2);
			
			\begin{scope}[xshift=-1cm]
				\draw (0,0) rectangle (3,-3);
				\draw (3,0) rectangle (6,-3);
				\draw (6,0) rectangle (9,-3);
				
				\draw[fill=yellow, rounded corners, inner sep=0.3mm] (0,-2) rectangle (2,-3);
				\draw[fill=yellow, rounded corners, inner sep=0.3mm] (3,0) rectangle (5,-1);
				\draw[fill=yellow, rounded corners, inner sep=0.3mm] (6,-1) rectangle (8,-2);
				
				\fill[pattern=north west lines, pattern color=black!30] (2,0) rectangle (3,-3);
				\fill[pattern=north west lines, pattern color=black!30] (5,0) rectangle (6,-3);
				\fill[pattern=north west lines, pattern color=black!30] (8,0) rectangle (9,-3);
				
				\draw (0.5,-0.5) node {$a$};
				\draw (0.5,-1.5) node {$a$};
				\draw (0.5,-2.5) node {$c$};
				\draw (1.5,-0.5) node {$b$};
				\draw (1.5,-1.5) node {$b$};
				\draw (1.5,-2.5) node {$c$};
				\draw (2.5,-0.5) node {$c$};
				\draw (2.5,-1.5) node {$c$};
				\draw (2.5,-2.5) node {$c$};
				
				\draw (3.5,-0.5) node {$c$};
				\draw (3.5,-1.5) node {$a$};
				\draw (3.5,-2.5) node {$a$};
				\draw (4.5,-0.5) node {$c$};
				\draw (4.5,-1.5) node {$b$};
				\draw (4.5,-2.5) node {$b$};
				\draw (5.5,-0.5) node {$c$};
				\draw (5.5,-1.5) node {$c$};
				\draw (5.5,-2.5) node {$c$};
				
				\draw (6.5,-0.5) node {$a$};
				\draw (6.5,-1.5) node {$c$};
				\draw (6.5,-2.5) node {$a$};
				\draw (7.5,-0.5) node {$b$};
				\draw (7.5,-1.5) node {$c$};
				\draw (7.5,-2.5) node {$b$};
				\draw (8.5,-0.5) node {$c$};
				\draw (8.5,-1.5) node {$c$};
				\draw (8.5,-2.5) node {$c$};
			\end{scope}

			\begin{scope}[xshift=-2cm,yshift=-3cm]
				\draw (0,0) rectangle (3,-3);
				\draw (3,0) rectangle (6,-3);
				\draw (6,0) rectangle (9,-3);
				
				\draw[fill=yellow, rounded corners, inner sep=0.3mm] (0,-2) rectangle (2,-3);
				\draw[fill=yellow, rounded corners, inner sep=0.3mm] (3,0) rectangle (5,-1);
				\draw[fill=yellow, rounded corners, inner sep=0.3mm] (6,-1) rectangle (8,-2);
				
				\fill[pattern=north west lines, pattern color=black!30] (2,0) rectangle (3,-3);
				\fill[pattern=north west lines, pattern color=black!30] (5,0) rectangle (6,-3);
				\fill[pattern=north west lines, pattern color=black!30] (8,0) rectangle (9,-3);
				
				\draw (0.5,-0.5) node {$a$};
				\draw (0.5,-1.5) node {$a$};
				\draw (0.5,-2.5) node {$c$};
				\draw (1.5,-0.5) node {$b$};
				\draw (1.5,-1.5) node {$b$};
				\draw (1.5,-2.5) node {$c$};
				\draw (2.5,-0.5) node {$c$};
				\draw (2.5,-1.5) node {$c$};
				\draw (2.5,-2.5) node {$c$};
				
				\draw (3.5,-0.5) node {$c$};
				\draw (3.5,-1.5) node {$a$};
				\draw (3.5,-2.5) node {$a$};
				\draw (4.5,-0.5) node {$c$};
				\draw (4.5,-1.5) node {$b$};
				\draw (4.5,-2.5) node {$b$};
				\draw (5.5,-0.5) node {$c$};
				\draw (5.5,-1.5) node {$c$};
				\draw (5.5,-2.5) node {$c$};
				
				\draw (6.5,-0.5) node {$a$};
				\draw (6.5,-1.5) node {$c$};
				\draw (6.5,-2.5) node {$a$};
				\draw (7.5,-0.5) node {$b$};
				\draw (7.5,-1.5) node {$c$};
				\draw (7.5,-2.5) node {$b$};
				\draw (8.5,-0.5) node {$c$};
				\draw (8.5,-1.5) node {$c$};
				\draw (8.5,-2.5) node {$c$};
			\end{scope}

			\begin{scope}[xshift=-3cm,yshift=-6cm]
				\draw (0,0) rectangle (3,-3);
				\draw (3,0) rectangle (6,-3);
				\draw (6,0) rectangle (9,-3);
				
				\draw[fill=yellow, rounded corners, inner sep=0.3mm] (0,-2) rectangle (2,-3);
				\draw[fill=yellow, rounded corners, inner sep=0.3mm] (3,0) rectangle (5,-1);
				\draw[fill=yellow, rounded corners, inner sep=0.3mm] (6,-1) rectangle (8,-2);
				
				\fill[pattern=north west lines, pattern color=black!30] (2,0) rectangle (3,-3);
				\fill[pattern=north west lines, pattern color=black!30] (5,0) rectangle (6,-3);
				\fill[pattern=north west lines, pattern color=black!30] (8,0) rectangle (9,-3);
				
				\draw (0.5,-0.5) node {$a$};
				\draw (0.5,-1.5) node {$a$};
				\draw (0.5,-2.5) node {$c$};
				\draw (1.5,-0.5) node {$b$};
				\draw (1.5,-1.5) node {$b$};
				\draw (1.5,-2.5) node {$c$};
				\draw (2.5,-0.5) node {$c$};
				\draw (2.5,-1.5) node {$c$};
				\draw (2.5,-2.5) node {$c$};
				
				\draw (3.5,-0.5) node {$c$};
				\draw (3.5,-1.5) node {$a$};
				\draw (3.5,-2.5) node {$a$};
				\draw (4.5,-0.5) node {$c$};
				\draw (4.5,-1.5) node {$b$};
				\draw (4.5,-2.5) node {$b$};
				\draw (5.5,-0.5) node {$c$};
				\draw (5.5,-1.5) node {$c$};
				\draw (5.5,-2.5) node {$c$};
				
				\draw (6.5,-0.5) node {$a$};
				\draw (6.5,-1.5) node {$c$};
				\draw (6.5,-2.5) node {$a$};
				\draw (7.5,-0.5) node {$b$};
				\draw (7.5,-1.5) node {$c$};
				\draw (7.5,-2.5) node {$b$};
				\draw (8.5,-0.5) node {$c$};
				\draw (8.5,-1.5) node {$c$};
				\draw (8.5,-2.5) node {$c$};

				\draw (9,0) rectangle (12,-3);
				
				\draw[fill=yellow, rounded corners, inner sep=0.3mm] (9,-1) rectangle (11,-2);
				
				\fill[pattern=north west lines, pattern color=black!30] (11,0) rectangle (12,-3);
				
				\draw (9.5,-0.5) node {$a$};
				\draw (9.5,-1.5) node {$c$};
				\draw (9.5,-2.5) node {$a$};
				\draw (10.5,-0.5) node {$b$};
				\draw (10.5,-1.5) node {$c$};
				\draw (10.5,-2.5) node {$b$};
				\draw (11.5,-0.5) node {$c$};
				\draw (11.5,-1.5) node {$c$};
				\draw (11.5,-2.5) node {$c$};
			\end{scope}
			
		\end{tikzpicture}
	}
	\caption{A coding region reused from \cref{codingblock}, delimited by buffers, with an example of transmission that is ensured between the different slices of it.}
	\label{fig:complexcode}
\end{figure}

Moreover, two adjacent slices also contain the same encoding, using the fact that there is no cross-bridge for $i=i_0-1$ or $i=i_0+2$. Indeed, suppose we have a $c^1_{i_0}$ coding micro-slice in the rightmost slice. To its left, in the second rightmost slice, is a $c^1_{i_0-1}$ coding micro-slice that encodes the same thing,  since there is no cross-bridge for $i=i_0-1$. Since we force two vertically adjacent coding micro-slices to encode the same tile if none of them is a buffer, the coding micro-slice using $c^1_{i_0}$, that is below the one using $c^1_{i_0-1}$, encodes the same tile as the latter. But to the left of the $c^1_{i_0}$ coding micro-tile is a $c^1_{i_0-1}$ coding micro-tile that encodes the same tile. Below this one is, once again, a $c^1_{i_0}$ coding micro-tile that encodes the same tile by construction... The same reasoning works when starting from a $c^1_{i_0+2}$ coding micro-slice in the leftmost slice: it encodes the same thing as the $c^1_{i_0+3}$ coding micro-slice to its right since there is no cross-bridge for $i=i_0+2$, etc.

\subsection{Case 2}
\label{subsec:case2}

Here, we assume that $\mathcal{G}(H)$ contains no loop but at least one bidirectional edge.

\textbf{Case 2.1:} we can find one cycle of length at least $3$ with no repeated vertex, with at least one bidirectional edge.

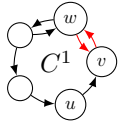
\begin{center}
	\begin{tikzpicture}[scale=0.4]
		\def \radius {1.5cm};
		
		\node[draw, circle,scale=0.7] (a) at (0:\radius) {$v$};
		\node[draw, circle,scale=0.7] (b) at (72:\radius) {$w$};
		\node[draw, circle] (c) at (144:\radius) {};
		\node[draw, circle] (d) at (216:\radius) {};
		\node[draw, circle,scale=0.7] (e) at (288:\radius) {$u$};
		
		\draw (0,0) node {$C^1$};
		
		\draw[->, >=latex,color=red] (a) to[bend right=15] (b);
		\draw[->, >=latex] (b) to[bend right=15] (c);
		\draw[->, >=latex] (c) to[bend right=15] (d);
		\draw[->, >=latex] (d) to[bend right=15] (e);
		\draw[->, >=latex] (e) to[bend right=15] (a);
		
		\draw[->, >=latex,color=red] (b) to[bend right=15] (a);
		\draw[->, >=latex] (c) to[bend right=15] (b);
	\end{tikzpicture}
\end{center}

Focus on a cycle that contains both a bidirectional and a unidirectional edge -- we can find one, else $\mathcal{G}(H)$ would be of symmetric type. Just as in Case 1.2, reduce it iteratively so that it ends up either in a graph -- as small as possible -- similar to the one of Case 2.2 and is therefore treated similarly; or as the smallest cycle with both a bidirectional and a unidirectional edge that does not contain a graph as in Case 2.2.

Name $C^1$ this cycle; name $u$, $v$ and $w$ some successive vertices in $C^1$ such that $(u,v), (v,w)$ and $(w,v) \in \vec{E}$ but $(v,u) \notin \vec{E}$; and set $C^2 = \{v,w\}$. Also, define $c^1_0 = v = c^2_0$.

As in Case 1.2, notice that all edges from $w$ must lead either to $c^1_2$ or to $v = c^1_0$, else we could find a shorter cycle $C^1$. There remains to check that $C^1$ and $C^2$ have the properties we want:
\begin{itemize}
	\item $|C^1| \geq 3$;
	
	\item $C^1$ and $C^2$ have a good pair, that is $(c^1_2,v)$;
	
	\item $C^2$ has no uniform shortcut since it is of length $2$ with no loop. If $C^1$ had a uniform shortcut, it could not be of size $0$ (because all vertices are loopless) or of size $-1$ (because $(v,u) \notin \vec{E}$). Any other size of shortcut is impossible due to the aforementioned property of $w$.
	
	\item If there was a cross-bridge, we reach a contradiction in the exact same fashion as what is done in the case of a cross-bridge in Case 1.2.
	
	\item There cannot be any attractive or repulsive vertex for $C^1$ located in $C^1$ since no vertex has a loop in the present case. None can be located in $C^2$ either since $C^2 \subset C^1$.
\end{itemize}

\textbf{Case 2.2:} any cycle of length at least $3$ with no repeated vertex contains no bidirectional edge.

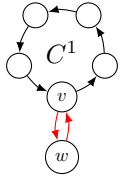
\begin{center}
	\begin{tikzpicture}[scale=0.4,rotate=-90]
		\def \radius {1.5cm};
		
		\node[draw, circle,scale=0.7] (a) at (0:\radius) {$v$};
		\node[draw, circle] (b) at (72:\radius) {};
		\node[draw, circle] (c) at (144:\radius) {};
		\node[draw, circle] (d) at (216:\radius) {};
		\node[draw, circle] (e) at (288:\radius) {};
		
		\node[draw, circle,scale=0.7] (f) at (3.5,0) {$w$};
		
		\draw (0,0) node {$C^1$};
		
		\draw[->, >=latex] (a) to[bend right=15] (b);
		\draw[->, >=latex] (b) to[bend right=15] (c);
		\draw[->, >=latex] (c) to[bend right=15] (d);
		\draw[->, >=latex] (d) to[bend right=15] (e);
		\draw[->, >=latex] (e) to[bend right=15] (a);
		
		\draw[->, >=latex,color=red] (a) to[bend right=15] (f);
		\draw[->, >=latex,color=red] (f) to[bend right=15] (a);
	\end{tikzpicture}
\end{center}

Since there are bidirectional edges in $\mathcal{G}(H)$ (hypothesis of Case 2), which is strongly connected, we can find at least one cycle of length $\geq 3$ with no repeated vertex that has one vertex in common with a bidirectional edge. Choose a minimal cycle among these ones, call it $C^1$; name $v$ the vertex it has in common with the bidirectional edge, and $w$ the other vertex. We define $C^2 = \{v, w\}$. Call $c^1_0 = v = c^2_0$.
\begin{itemize}
	\item $|C^1| \geq 3$;
	
	\item $(c^1_1,w)$ is a good pair for $C^1$ and $C^2$;
	
	\item $C^2$ has no uniform shortcut since it is of length $2$ with no loop. If $C^1$ had uniform shortcuts, they could not be of length $0$ because none of its vertices has a loop; they could not be of length $-1$ because none of its edges is bidirectional; and they could not be of any other length else the shortcut starting from $v$ would allow us to define a strictly shorter cycle with the same property, contradicting the minimality of $C^1$.
	
	\item If there was a cross-bridge, we could have two cases:
	\begin{itemize}
		\item First is $(c^1_i,v)$ and $(w,c^1_{i+1}) \in \vec{E}$, with $c^1_i \neq w$ and $v \neq c^1_{i+1}$. Then we could use the edge $(w,c^1_{i+1})$ for the following cycle: $\{ w, c^1_{i+1}, c^1_{i+2}, ..., v \}$. It would be of length at least $3$ since $c^1_{i+1} \neq v$, no vertex would repeat, and it would contain one bidirectional edge. This is impossible by assumption of Case 2.2.
		
		\item Second is $(c^1_i,w)$ and $(v,c^1_{i+1}) \in \vec{E}$, with $c^1_i \neq v$ and $w \neq c^1_{i+1}$. Then we could use the edge $(v,c^1_{i+1})$ to define a cycle strictly shorter than $C^1$, with the same properties, of length at least $3$ ($c^1_{i+1}$ cannot precede $v$, else we would have a bidirectional edge). This is impossible.
	\end{itemize}
	
	\item There is no attractive or repulsive vertex for $C^1$ located in $C^1$ since none has a loop. But it seems that there can be an attractive and/or repulsive vertex for $C^1$ located in $C^2$, that is, $w$. Nevertheless, if $w$ was both attractive and repulsive, using part of $C^1$ we could build a cycle of length at least $3$ with no repeated vertex including $v$ and $w$, hence a bidirectional edge. This is forbidden in this Case 2.2.
\end{itemize}

\subsection{Case 3}
\label{subsec:case3}

In this subsection, we assume $\mathcal{G}(H)$ contains no loop and no bidirectional edge.

\textbf{Case 3.1:} considering the smallest cycle $C$ in $\mathcal{G}(H)$ one can find, there exists a path $\gamma$ between two different vertices of $C$ that does not intercept $C$ elsewhere, and $\gamma$ is of a length different from the length between these vertices inside $C$.

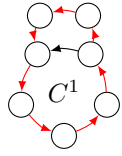
\begin{center}
	\begin{tikzpicture}[scale=0.4,rotate=-90]
		\def \radius {1.5cm};
		
		\node[draw, circle] (a) at (0:\radius) {};
		\node[draw, circle] (b) at (72:\radius) {};
		\node[draw, circle] (c) at (144:\radius) {};
		\node[draw, circle] (d) at (216:\radius) {};
		\node[draw, circle] (e) at (288:\radius) {};
		
		\node[draw, circle] (f) at (-2.5,0.8) {};
		\node[draw, circle] (g) at (-2.5,-0.8) {};
		
		\draw (0,0) node {$C^1$};
		
		\draw[->, >=latex,color=red] (a) to[bend right=15] (b);
		\draw[->, >=latex,color=red] (b) to[bend right=15] (c);
		\draw[->, >=latex] (c) to[bend right=15] (d);
		\draw[->, >=latex,color=red] (d) to[bend right=15] (e);
		\draw[->, >=latex,color=red] (e) to[bend right=15] (a);
		
		\draw[->, >=latex,color=red] (c) to[bend right=15] (f);
		\draw[->, >=latex,color=red] (f) to[bend right=15] (g);
		\draw[->, >=latex,color=red] (g) to[bend right=15] (d);
	\end{tikzpicture}
\end{center}

Define $C^1$ a cycle with said property for a path $\gamma$ between two of its vertices so that $C^1$ is a cycle of minimal length. If there are several cycles of minimal length, choose one so that we can find a path $\gamma$ as short as possible. Now, we name the vertices: $\gamma$ is a path between $c^1_0$ and $c^1_k$, that is not of length $k$. Define $C^2$ as the concatenation of $\gamma$ and $( c^1_k, c^1_{k+1}, ..., c^1_0 )$, with $c^2_0 := c^1_0$ and $c^2_l = c^1_k$ with $l = |\gamma| > k$.

$C^1$ and $C^2$ verify all the conditions we need:
\begin{itemize}
	\item $|C^1| \geq 3$ since there is no bidirectional edge;
	
	\item $(c^1_1, c^2_1)$ is a good pair;
	
	\item In both $C^1$ and $C^2$, we have no uniform shortcut of length $0$ or $-1$ since there are no loops and no bidirectional edges. $C^1$ cannot have any other length of uniform shortcut, or even any edge between two of its vertices, else we could find a strictly smaller cycle with a path of a different length between two of its vertices.
	
	Suppose $C^2$ has a uniform shortcut of length $j$. The point if it happens is to build a new cycle, $C^3$, so that $C^1$ and $C^3$ work for the generic construction in \cref{subsec:core}.
	If we have an uniform shortcut, then $c^2_0 = c^1_0$ cannot lead to an element that $C^2$ shares with $C^1$ by minimality of the latter, hence the edge of its shortcut must lead to some $c^2_j$, with $1< j < l$ the uniform length of the shortcuts. Necessarily, $(c^1_0 = c^2_0, c^2_j, c^2_{j+1}, ..., c^2_l = c^1_k)$ being a path between two elements of $C^1$ that is strictly shorter than $\gamma$, we have $l-j+1 = k$ in order not to reach a contradiction. Hence $j = l-k+1$. Then $C^\prime := (c^2_0, c^2_{l-k+1}, c^2_{l-k+2}, ..., c^2_l = c^1_k, c^1_{k+1}, ..., c^1_{-1})$ is a cycle of length $|C^1|$.
	
	First, we study the case $k \neq 1$. Then $(c^2_0, c^2_1, c^2_2, ..., c^2_{l-k+1})$ is a path of length $l-k+1 < l$, linking two elements of $C^\prime$ (these elements are $c^2_0$ and $c^2_{l-k+1}$, which are consecutive in $C^\prime$). Since $C^\prime$ is of length $|C^1|$ and we found a path of length smaller than $l$ joining two of its vertices, this fact contradicts the minimality of $C^1$, so we cannot actually have $k \neq 1$.
	
	Necessarily $k=1$. Then $j = l$, and the edge between $c^2_0$ and $c^2_l = c^1_k = c^1_1$ is already part of $C^1$ -- it is $(c^1_0, c^1_1)$. We have $l>k$ so $l \neq 1$; if $l=2$ then $j=2$ so $c^2_2=c^1_1$ would have an edge going to $c^2_4 = c^1_3$, an element of $C^1$ necessarily (possibly $c^1_0$). This is impossible by minimality of $C^1$. We deduce that $l>2$.
	
	We set $C^3 := (c^2_0, c^2_1, c^2_2, c^1_{3}, c^1_{4}, ..., c^1_{|C^1|-1})$ using the edge $(c^2_2,c^2_{l+2})$ since $j = l$, which is the edge $(c^2_2,c^1_{3})$ since $c^2_l = c^1_1$.
	There is a specific case if $c^1_3 = c^1_0 = c^2_0$, where both $C^1$ and $C^3$ end up being triangles, but the reasoning below still holds.
	
	Instead of using $C^1$ and $C^2$, we check that choosing $C^1$ and $C^3$ for our generic construction works well:
	\begin{itemize}
		\item $|C^1| \geq 3$, this does not change;
		
		\item $(c^1_1, c^2_1)$ is still a good pair for $C^1$ and $C^3$;
		
		\item $C^1$ being still defined the same way, it does not contain any uniform shortcut or even any edge between two of its vertices. If $C^3$ contains uniform shortcuts, the one starting at $c^1_0 = c^2_0$ must lead to $c^2_2$ since it must be of size different from $1$ and it must not lead to an element in common with $C^1$, because the latter is minimal. But if $(c^2_0,c^2_2) \in \vec{E}$, then we could find a path strictly shorter than $\gamma$ between $c^1_0$ and $c^1_1$, that would not be of length $1$ (because $c^2_2 \neq c^2_l$). This contradicts the minimality of $C^1$.
		
		\item Since no edge between two non-consecutive elements of $C^1$ is possible, the unique cross-bridge between $C^1$ and $C^3$ would be some $(c^1_i,c^2_2)$ and $(c^2_1,c^1_{i+1}) \in \vec{E}$. But it would also be a cross-bridge between $C^1$ and $C^2$, and this is in all cases impossible, see below.
		
		\item There is no attractive or repulsive vertex for $C^1$ in $C^1$ since no element of $C^1$ has a loop. Suppose there are both an attractive and a repulsive vertex for $C^1$ located in $C^3 \setminus C^1$, call them $t^3$ and $p^3$. They must be distinct (because the graph has no bidirectional edge) and not be in $C^1$; hence $C^3$ contains at least $2$ exclusive vertices.
		
		Then the idea is to use $C^1$ as $C^3$ in the generic construction and vice versa: $|C^3| \geq 3$, and all of our other properties hold when swapping $C^3$ and $C^1$, except the attractive and repulsive conditions. Hence the only facts that we have to verify to exchange their roles is that there is no attractive or repulsive vertex for $C^3$ located either in $C^3$ or in $C^1$. There is none in $C^3$ because no vertex of $C^3$ has a loop. If there was both an attractive and a repulsive vertices for $C^3$ in $C^1$, call them $t^1$ and $p^1$, then notably $t^1$ would lead to $t^3$ and vice versa... But no bidirectional edge exists here. So, up to exchanging what is $C^1$ and what is $C^3$, we cannot have both an attractive and a repulsive vertex for $C^1$. This reasoning will be applied again and be called the \emph{trick of exchanging the roles}.
	\end{itemize}
	Therefore in the worst case, if there are uniform shortcuts in $C^2$, we can build the generic construction from \cref{subsec:core} with $C^1$ and $C^3$.
	
	\item Suppose we have a cross-bridge, that is, $(c^1_i,c^2_{j+1})$ and $(c^2_j,c^1_{i+1}) \in \vec{E}$. Since $C^1$ is minimal, $c^2_j$ and $c^2_{j+1}$ are not elements of $C^1$, so $j+1<l$. Then $(c^1_0 = c^2_0, c^2_1, ..., c^2_j, c^1_{i+1})$ is a path between two elements of $C^1$ that is necessarily strictly shorter than $\gamma$ since $j+1 < l$ (with $c^2_l = c^1_k$). This is possible only if the obtained path is of length $i+1$, the same length as the one between $c^1_0$ and $c^1_{i+1}$ in $C^1$. It would mean that $j = i$, but then $(c^1_i, c^2_{j+1}, c^2_{j+2}, ..., c^2_l = c^1_k)$ would also be a path $\gamma^\prime$ shorter than $\gamma$, and the only possibility is then that $k-i+1 = l-j+1$ (the distance between $c^1_i$ and $c^1_k$ is the length of $\gamma^\prime$), of which we deduce $k=l$, which is impossible.
	
	\item Once again, for attractive and repulsive vertices we use the trick of exchanging the roles of $C^1$ and $C^2$.
\end{itemize}

\textbf{Case 3.2:} considering the smallest cycle $C$ in $\mathcal{G}(H)$ with no repeated vertex, any path $\gamma$ we can find between two different vertices of $C$ that does not intercept $C$ elsewhere is of the same length as the length between these vertices inside $C$; and we can find at least one such path $\gamma$.

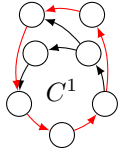
\begin{center}
	\begin{tikzpicture}[scale=0.4,rotate=-90]
		\def \radius {1.5cm};
		
		\node[draw, circle] (a) at (0:\radius) {};
		\node[draw, circle] (b) at (72:\radius) {};
		\node[draw, circle] (c) at (144:\radius) {};
		\node[draw, circle] (d) at (216:\radius) {};
		\node[draw, circle] (e) at (288:\radius) {};
		
		\node[draw, circle] (f) at (-2.5,1) {};
		\node[draw, circle] (g) at (-2.5,-1) {};
		
		\draw (0,0) node {$C^1$};
		
		\draw[->, >=latex,color=red] (a) to[bend right=15] (b);
		\draw[->, >=latex] (b) to[bend right=15] (c);
		\draw[->, >=latex] (c) to[bend right=15] (d);
		\draw[->, >=latex] (d) to[bend right=15] (e);
		\draw[->, >=latex,color=red] (e) to[bend right=15] (a);
		
		\draw[->, >=latex,color=red] (b) to[bend right=15] (f);
		\draw[->, >=latex,color=red] (f) to[bend right=15] (g);
		\draw[->, >=latex,color=red] (g) to[bend right=15] (e);
		
		\draw[->, >=latex] (c) to[bend right=15] (g);
	\end{tikzpicture}
\end{center}

Define $C^1 := C$. We use the following algorithm: we start with $V_0 := \{c^1_0\}$ and $(V_i)_{i \in [1, |C^1|]}$ empty. Then we recursively append to $V_{i+1}$ all vertices $w$ in $\mathcal{G}(H)$ so that there is a $v \in V_i$ with $(v,w) \in \vec{E}$, with a modulo $|C^1|$ on the index so that $V_{|C^1|} = V_0$.
The algorithm halts when it tries to append vertices to a $V_i$ that are all already in it, which happens because $\mathcal{G}(H)$ is made of a finite number of vertices. The fact that no path exterior to $C^1$ is of a different length than the corresponding path $C^1$, plus the absence of any loop or bidirectional edge, makes all the $V_i$ disjoint. Finally, the strong connectivity we assumed ensures that $H = \bigsqcup_{i=0}^{|C^1|-1} V_i$.

We use the fact that $H$ does not verify condition $D$, specifically is not of state-split cycle type. Since by construction, for any $v \in V_i$, we have $(v,w) \in \vec{E} \Rightarrow w \in V_{i+1}$, the only possibility is that $\exists v \in V_{i_0}, \exists w^\prime \in V_{i_0+1}, (v,w^\prime) \notin \vec{E}$. However, we also have some $v^\prime \in V_{i_0+1}, (v,v^\prime) \in \vec{E}$ and $w \in V_{i_0}, (w, w^\prime) \in \vec{E}$. Obviously, the four vertices are different. Now, take a path $\gamma_1$ from $v^\prime$ to $v$ that is as short as possible. Take a different path $\gamma_2$ from $w^\prime$ to $w$ that still has as many possible vertices in common with $\gamma_1$. We redefine $C^1$ as the concatenation of $\gamma_1$ and $(v,v^\prime)$ and $C^2$ as the concatenation of $\gamma_2$ and $(w,w^\prime)$.

It is rather easy to see that the properties we need for our generic construction are verified:

\begin{itemize}
	\item $|C^1| \geq 3$, since there is no bidirectional edge;
	
	\item The unique common part between $C^1$ and $C^2$ is the biggest sequence of vertices $\gamma_1$ and $\gamma_2$ have in common, so starting from the first pair on which they disagree we obtain a good pair;
	
	\item There is no uniform shortcut between $C^1$ and $C^2$ if $\gamma_1$ was chosen minimal and $\gamma_2$ as close to $\gamma_1$ as possible;
	
	\item There is no cross-bridge for the same reason;
	
	\item For attractive and repulsive vertices, we use the trick of exchanging the roles of $C^1$ and $C^2$ described in Case 3.1.
\end{itemize}

\textbf{Case 3.3:} considering the smallest cycle $C$ in $\mathcal{G}(H)$ with no repeated vertex, we can find no path between two different vertices of this cycle.

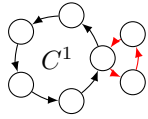
\begin{center}
	\begin{tikzpicture}[scale=0.4]
		\def \radius {1.5cm};
		
		\node[draw, circle] (a) at (0:\radius) {};
		\node[draw, circle] (b) at (72:\radius) {};
		\node[draw, circle] (c) at (144:\radius) {};
		\node[draw, circle] (d) at (216:\radius) {};
		\node[draw, circle] (e) at (288:\radius) {};
		
		\node[draw, circle] (f) at (2.5,-0.8) {};
		\node[draw, circle] (g) at (2.5,0.8) {};
		
		\draw (0,0) node {$C^1$};
		
		\draw[->, >=latex] (a) to[bend right=15] (b);
		\draw[->, >=latex] (b) to[bend right=15] (c);
		\draw[->, >=latex] (c) to[bend right=15] (d);
		\draw[->, >=latex] (d) to[bend right=15] (e);
		\draw[->, >=latex] (e) to[bend right=15] (a);
		
		\draw[->, >=latex,color=red] (a) to[bend right=15] (f);
		\draw[->, >=latex,color=red] (f) to[bend right=15] (g);
		\draw[->, >=latex,color=red] (g) to[bend right=15] (a);
	\end{tikzpicture}
\end{center}

Since $\mathcal{G}(H)$ does not verify condition $D$, it cannot be a plain cycle, hence a path exists from one vertex of $C$ to itself that does not intersect $C$ elsewhere. Define $C^1:=C$ and this vertex as $c^1_0$. Then, considering the smallest path $\gamma$ from $c^1_0$ to itself outside of $C^1$, we define $C^2 := \gamma$, with $c^2_0 := c^1_0$.

It remains to check that these $C^1$ and $C^2$ verify the properties we need.

\begin{itemize}
	\item $|C^1| \geq 3$;
	
	\item $C^1$ and $C^2$ have exactly one vertex in common, $c^1_0 = c^2_0$, and so $(c^1_1,c^2_1)$ is a good pair;
	
	\item There is no uniform shortcut of length $0$ or $-1$ neither in $C^1$ nor in $C^2$, since there are no loop and no bidirectional edge. There is no uniform shortcut of any other length; else consider the edge starting at $c^1_0$, be it in $C^1$ or in $C^2$: it would allow us to build a shorter $C^1$ or a shorter $C^2$, contradicting the fact that the two of them have been chosen to be minimal.
	
	\item There is no cross-bridge between $C^1$ and $C^2$ because it would allow us to build path outside of $C^1$ between two distinct elements of $C^1$, which is impossible in this case;
	
	\item There is no attractive or repulsive vertex for $C^1$ in $C^1$, because no element of $C^1$ has a loop. There is no attractive or repulsive vertex for $C^1$ in $C^2$ else we could build a path outside of $C^1$ between two distinct elements of $C^1$, which is impossible in this case.
\end{itemize}

\subsection{Additional cases}
\label{subsec:additional}

\subsubsection{Length 3 Cases:}
\label{subsubsec:length3}

For most three-vertex graphs, we can apply the generic construction from \cref{subsec:core} without any problem. However, some of them require to be slightly more cautious, because some properties are missing. These are, up to a change of labels:

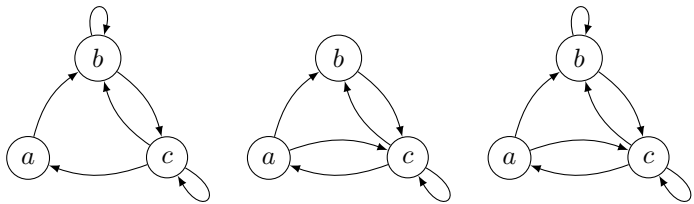
\begin{center}
	\begin{tikzpicture}[scale=0.8]
		
		\begin{scope}
			\node[draw, circle] (b) {$b$};
			\node[draw, circle] (a) [below left = 0.9cm and 0.5cm of b] {$a$};
			\node[draw, circle] (c) [below right = 0.9cm and 0.5cm of b] {$c$};
			
			\draw[->, >=latex] (a) to[bend left=20] (b);
			\draw[->, >=latex] (b) to[bend left=20] (c);
			\draw[->, >=latex] (c) to[bend left=20] (a);
			
			\draw[->, >=latex] (c) to[bend left=20] (b);
			
			\draw[->, >=latex] (b) to[loop above] (b);
			\draw[->, >=latex] (c) to [out=330,in=300,looseness=12] (c);
		\end{scope}
		
		\begin{scope}[xshift=4cm]
			\node[draw, circle] (b) {$b$};
			\node[draw, circle] (a) [below left = 0.9cm and 0.5cm of b] {$a$};
			\node[draw, circle] (c) [below right = 0.9cm and 0.5cm of b] {$c$};
			
			\draw[->, >=latex] (a) to[bend left=20] (b);
			\draw[->, >=latex] (b) to[bend left=20] (c);
			\draw[->, >=latex] (c) to[bend left=20] (a);
			
			\draw[->, >=latex] (c) to[bend left=20] (b);
			\draw[->, >=latex] (a) to[bend left=20] (c);
			
			\draw[->, >=latex] (c) to [out=330,in=300,looseness=12] (c);
		\end{scope}
		
		\begin{scope}[xshift=8cm]
			\node[draw, circle] (b) {$b$};
			\node[draw, circle] (a) [below left = 0.9cm and 0.5cm of b] {$a$};
			\node[draw, circle] (c) [below right = 0.9cm and 0.5cm of b] {$c$};
			
			\draw[->, >=latex] (a) to[bend left=20] (b);
			\draw[->, >=latex] (b) to[bend left=20] (c);
			\draw[->, >=latex] (c) to[bend left=20] (a);
			
			\draw[->, >=latex] (c) to[bend left=20] (b);
			\draw[->, >=latex] (a) to[bend left=20] (c);
			
			\draw[->, >=latex] (b) to[loop above] (b);
			\draw[->, >=latex] (c) to [out=330,in=300,looseness=12] (c);
		\end{scope}
		
	\end{tikzpicture}
\end{center}

Here, there are both an attractive and a repulsive vertex. Still, similarly to case 1.1, $a$ has in-degree $1$ since only $c$ leads to $a$, and so \cref{propapproxaligned} holds, since $a$ forces the alignment of columns.

Besides, we have to be careful about the cross-bridge property. In the first example, we choose $C^2=(b)$ (there is no cross-bridge then because $(b,a) \notin \vec{E}$). In the second and in the third, we choose $C^2=(a,c)$ ($a$ has no loop hence there is no problem of cross-bridge or of uniform shortcut in $C^2$). All the other properties from Condition $C$ are verified.

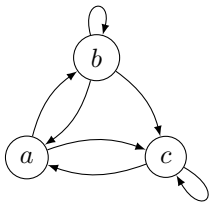
\begin{center}
	\begin{tikzpicture}
		\node[draw, circle] (b) {$b$};
		\node[draw, circle] (a) [below left = 0.9cm and 0.5cm of b] {$a$};
		\node[draw, circle] (c) [below right = 0.9cm and 0.5cm of b] {$c$};
		
		\draw[->, >=latex] (a) to[bend left=20] (b);
		\draw[->, >=latex] (b) to[bend left=20] (c);
		\draw[->, >=latex] (c) to[bend left=20] (a);
		
		\draw[->, >=latex] (b) to[bend left=20] (a);
		\draw[->, >=latex] (a) to[bend left=20] (c);
		
		\draw[->, >=latex] (b) to[loop above] (b);
		\draw[->, >=latex] (c) to [out=330,in=300,looseness=12] (c);
	\end{tikzpicture}
\end{center}

Finally, here we have three problems:

\begin{itemize}
	\item We could have cross-bridges, so to avoid them we choose $C^2 = (a, b)$ (we have no problem of uniform shortcut in $C^2$ since $a$ has no loop);
	\item We have attractive and repulsive vertices;
	\item And here we cannot rely on an element of the alphabet that must be followed or preceded by a specific other one to solve that problem.
\end{itemize}

The reasoning is slightly more subtle then: if we try to perform the generic construction, take two successive columns $K_1$ and $K_2$. In any macro-slice of $K_1$, there is some $a$ meso-slice (made of $N M$ symbols $a$) that is vertically preceded by a $c$ meso-slice. The $a$ meso-slice must not be next to any symbol $a$ in column $K_2$. But then if there is any $c$ in the part of $K_2$ horizontally adjacent to this $a$ meso-slice, the aforementioned $c$ meso-slice of $K_1$ is horizontally followed by at least one $b$ (be it from a regular meso-slice, a border, or a code); but this cannot be. Hence an $a$ meso-slice in $K_1$ can only be horizontally followed by symbols $b$. So two columns are always aligned even if we have attractive and repulsive vertices.

The rest of the generic construction works normally.

\subsubsection{Specificity of Case 1.3:}
\label{subsubsec:case1.3}

We focus on specific subcase where $C^2 = (c)$ and $C^1 = (a,p,b,c,t)$ where all vertices must not necessarily be different, with $t$ attractive, $p$ repulsive, $(c,b) \notin \vec{E}$, loops on $b$ and $c$, and $a$ loopless. Additionally, the initial and terminal vertices of any unidirectional edge must have a loop. Five cases can happen; here we treat only the fourth one, all the others are done similarly. Note that the fifth case is treated among the Length 3 Cases.
\begin{itemize}
	\item All elements are distinct;
	\item $p=t$ and all others distinct;
	\item $c=t$ and all others distinct;
	\item $b=p$ and all others distinct;
	\item $b=p$ and $c=t$: this is one of the three-vertex graphs we have seen before.
\end{itemize}

Note that in all those cases, stemming from the analyze performed in case 1.3 in \cref{subsec:case1}, our generic construction seems to work except for the presence of both attractive and repulsive vertices, that endangers \cref{propapproxaligned}. The only fact that we have to check is that we can circumvent this obstacle in a way similar to what is done in \cref{subsubsec:length3}.

If $b=p$, we obtain the following graph:

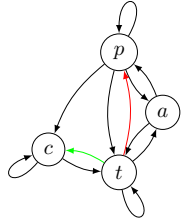
\begin{center}
	\scalebox{0.8}{
		\begin{tikzpicture}[scale=0.7]
			
			\def \radius {1.5cm};
			
			\node[draw, circle] (a) at (0:\radius) {$a$};
			\node[draw, circle] (p) at (72:\radius) {$p$};
			\node[draw, circle] (c) at (216:\radius) {$c$};
			\node[draw, circle] (t) at (288:\radius) {$t$};
			
			\draw[->, >=latex] (a) to[bend right=15] (p);
			\draw[->, >=latex] (c) to[bend right=15] (t);
			\draw[->, >=latex] (t) to[bend right=15] (a);
			
			\draw[->, >=latex] (p) to[bend right=15] (a);
			\draw[->, >=latex] (p) to[bend right=15] (c);
			\draw[->, >=latex] (p) to[bend right=15] (t);
			\draw[->, >=latex] (p) to[out=60,in=90,looseness=12] (p);
			
			\draw[->, >=latex] (a) to[bend right=15] (t);
			\draw[->, >=latex] (t) to[out=280,in=310,looseness=12] (t);
			
			\draw[->, >=latex] (c) to[out=230,in=200,looseness=12] (c);
			
			\draw[->, >=latex,color=red] (t) to[bend right=15] (p);
			\draw[->, >=latex,color=green] (t) to[bend right=15] (c);
			
		\end{tikzpicture}
	}
\end{center}

We added the red edge so that we cannot reduce the graph to a strictly smaller cycle (with three vertices) on which we already proved the generic construction worked. If there was an edge between $a$ and $c$, it would be bidirectional (since $a$ has no loop). But since the edges between $c$ and $t$ or $c$ and $p$ cannot be both bidirectional, we could reduce the present cycle to a strictly smaller one containing a unidirectional edge, a loop and a loopless vertex. So there is no edge between $a$ and $c$. Since $p=b$ there is none from $p$ to $c$. The only optional edge available is $(t,c)$ (in green).

We use the same method as before: take two horizontally successive columns $K_1$ and $K_2$.  In any macro-slice of $K_1$, there is a $c$ meso-slice slice that is above a $t$ meso-slice, itself above a $a$ meso-slice.

The $a$ meso-slice in $K_1$ cannot be horizontally followed by a border or a code meso-slice in $K_2$ because they contain $c$. The $c$ meso-slice in $K_1$ must not be followed horizontally by any symbol $p$ or $a$ in column $K_2$ -- so most of it (at least $N M / 2$) is in contact neither with a border nor with a code meso-slice, but with a meso-slice made of only one symbol. If that symbol is $c$, then the aforementioned $a$ meso-slice is in contact with a meso-slice made of a symbol $a$. This is impossible.
Hence the $c$ meso-slice of $K_1$ is mostly followed by a $t$ meso-slice of $K_2$, and from this we recover \cref{propapproxaligned}.

The three other cases are treated similarly, exploring with what each $C^1$ meso-slice can be in contact to ensure that \cref{propapproxaligned} is valid even without all of condition $C$. Checking the rest of the properties follows case 1.3.

\subsection{Proof of Theorem 4.1 for several strongly connected components}

The idea if $H$ has several SCCs is to build one, by products of SCCs, that is none of the three types that constitute condition $D$. We can then apply what we did in the previous subsections.

The direct product $S_1 \times S_2$ of two SCCs $S_1$ and $S_2$ is made of pairs $(s_1,s_2)$, where an edge exists between two pairs if and only if edges exist in both $S_1$ and $S_2$ between the corresponding vertices. It can be used in our construction by forcing pairs of elements $(s_1,s_2) \in S_1 \times S_2$ to be vertically one on top of the other.

Since $H$ does not verify condition $D$, it has a non-reflexive SCC $S_1$, a non-symmetric SCC $S_2$ and a non-state-split SCC $S_3$ (two of them being possibly the same). But then:

\begin{itemize}
	\item Since $S_1$ is non-reflexive, no SCC of $S_1 \times S_2 \times S_3$ is reflexive. Indeed, since $S_1$ is strongly connected, all vertices of $S_1$ are represented in any SCC $C$ of that graph product, meaning that for any $s_1 \in S_1$ there is at least one vertex of the form $(s_1,*,*)$ in $C$. But if $C$ had loop on all its vertices, then in particular $S_1$ would be reflexive.
	
	\item Similarly, since $S_2$ is non-symmetric, no SCC of $S_1 \times S_2 \times S_3$ is symmetric.
	
	\item Finally, since $S_3$ is non-state-split, no SCC of $S_1 \times S_2 \times S_3$ is a state-split cycle. Indeed, suppose $S$ is such a state-split SCC of the direct product. It can be written as a collection of classes $(V_i)_{i \in I}$ of elements from $S_1 \times S_2 \times S_3$ that we can project onto $S_3$, getting new classes $(W_i)_{i \in I}$, with some elements of $S_3$ that possibly appear in several of these.
	Let $c$ be any vertex in $S_3$ that appears at least twice \emph{with the least difference of indices between two classes where it appears}; say $c \in W_i$ and $c \in W_{i+k}$. Since $S$ is state-split, all elements in $W_{i+1}$ are exactly the elements of $S_3$ to which $c$ leads. But it is the same for $W_{i+k+1}$. Hence $W_{i+1} = W_{i+k+1}$. From this we deduce that $W_i=W_{i+k}$ for any $i$, using the fact that indices are modulo $|I|$. Since $k$ is the smallest possible distance between classes having a common element, classes from $(W_i)_{i \in \{0,\dots,k-1\}}$ are all disjoint; and they obviously contain all vertices from $S_3$. Now simply consider these classes $W_0$ to $W_{k-1}$: you get the proof that $S_3$ is state-split.
\end{itemize}


\section{Some consequences of classical problems on two-dimensional subshifts  under interplay between horizontal and vertical conditions}
\label{sec:consequences}

\subsection{Periodicity}

For $x \in X$ a two-dimensional subshift, we say that $x$ is \emph{periodic} (of period $\vec{v}$) if there exists $\vec{v} \in \Z^2 \setminus \{(0,0)\}$ such that $\forall (i, j) \in \Z^2, x_{(i,j)} = x_{(i,j)+\vec{v}}$. In a more general setting of SFTs on groups in general, this is called weak periodicity.

\begin{corollary}
	\label{coroaperiodic}
	Let $H$ be a one-dimensional nearest-neighbor SFT.
	
	\begin{center}
		$X_{H,V}$ is empty or contains a periodic configuration for all one-dimensional SFTs $V$
		
		$\Leftrightarrow$ $\mathcal{G}(H)$ verifies condition $D$.
	\end{center}
\end{corollary}

\begin{proof}
	If $\mathcal{G}(H)$ verifies condition $D$, then, as is detailed in the proof of \cref{th:DP}, whatever may be the chosen $V$, we can find a patch $P$ that respects the local rules of $X_{H,V}$ and tiles the plane periodically. Hence $X_{H,V}$ admits a periodic configuration.
	
	If $\mathcal{G}(H)$ does not verify condition $D$, then, using \cref{th:root}, we know that for any two-dimensional SFT $Y$ with no periodic configuration, there exists some one-dimensional SFT $V_Y$ such that $X_{H,V_Y}$ is a $(m,n)$th root of $Y$.
	
	We consider $Y$ a two-dimensional SFT with no periodic configuration (see \cite{dprobinson} for instance). Then, naming $V_Y$ the corresponding one-dimensional SFT from \cref{th:root}, we know that there exists $\psi\colon Y \hookrightarrow X_{H,V}$ continuous with $\sqcup_{0 \leq i < m, 0 \leq j < n} \sigma^{(i,j)}(\psi(Y)) = X_{H,V}$ for some integers $m$ and $n$. Note that we use $\psi$ the inverse map of $\phi$ in the definition of a $(m,n)$th root, the reasoning being easier with it.
	
	If $X_{H,V}$ contained a $\sigma$-periodic configuration, then $\sqcup_{0 \leq i < m, 0 \leq j < n} \sigma^{(i,j)}(\psi(Y))$ would, and so $\psi(Y)$ would too (since configurations in $\sqcup_{0 \leq i < m, 0 \leq j < n} \sigma^{(i,j)}(\psi(Y))$ are merely translates of the ones in $\psi(Y)$).
	
	Call $\psi(y)$ such a periodic configuration, with $y \in Y$. There exists some $\vec{v} = (a,b) \in \Z^2$ such that $\sigma^{\vec{v}}(\psi(y)) = \psi(y)$. But consequently, $\sigma^{mn\vec{v}}(\psi(y)) = \psi(y)$. Then $\sigma^{(anm,bmn)}(\psi(y)) = \psi(y)$. Using \cref{th:root} again, we obtain that $\psi(\sigma^{(an,bm)}(y)) = \psi(y)$. $\psi$ being bijective, we finally get:
	\[
	\sigma^{(an,bm)}(y) = y.
	\]
	We found a periodic configuration in $Y$. This being impossible, we conclude that $X_{H,V}$ contains no periodic configuration.
\end{proof}

\subsection{The Domino Problem}
\label{DominoProblemCombined}

We consider now the domino problem when horizontal and vertical constraints interplay.  We want to understand when two one-dimensional SFTs are compatible to build a two-dimensional SFT, and by extension where the frontier between decidability (one-dimensional) and undecidability (two-dimensional) yields. This question is notably reflected by the following adapted version of the Domino Problem:

\begin{definition}
	Let $H \subset \A^\Z$ be a SFT. The \emph{Domino Problem depending on $H$} is the language
	\[DP_I(H):=\{<V> \mid V \subset \A^\Z \text{ is an SFT and } X_{H,V} \neq \emptyset\}.\]
\end{definition}

\begin{remark}
	Just as for $DP_h(H)$ from \cref{subsec:horizontalconstraints}, this problem is always defined for a given $H$.
\end{remark}

From this definition and \cref{sec:simulation}, we deduce:

\begin{theorem}
	\label{th:DP}
	Let $H$ be a nearest-neighbor one-dimensional SFT.
	
	\begin{center}
		$DP_I(H)$ is decidable $\Leftrightarrow$ $\mathcal{G}(H)$ verifies condition $D$.
	\end{center}
\end{theorem}

\begin{proof}
	Proof of $\Leftarrow$: assume $\mathcal{G}(H)$ verifies condition $D$. Then its SCCs share a common type, be it reflexive, symmetric, or state-split cycle. For each of these three cases, we produce an algorithm that takes as input a one-dimensional SFT $V \subset \A^\Z$, and that returns \texttt{YES} if $X_{H,V}$ is nonempty, and \texttt{NO} otherwise.
	
	Let $M$ be the maximal size of forbidden patterns in $\mathcal{F}_V$ (since $V$ is an SFT, such an integer exists).
	\begin{itemize}			
		\item If $\mathcal{G}(H)$ has state-split cycle type SCCs: let L be the LCM of the number of $V_i$s in each component.
		If there is no rectangle of size $L \times M(|\A|^{LM}+1)$ respecting local rules of $X_{H,V}$ and containing no transient element, then answer \texttt{NO}. Indeed, any configuration in $X_{H,V}$ contains valid rectangles as large as we want that do not contain transient elements.
		If there is such a rectangle $R$, then by the pigeonhole principle it contains at least twice the same rectangle $R^\prime$ of size $L \times M$. To simplify the writing, we assume that the rectangle that repeats is the one of coordinates $[1,L] \times [1,M]$ inside $R$ where $[1,L]$ and $[1,M]$ are intervals of integers, and that it can be found again with coordinates $[1,L] \times [k,k+M-1]$. Else, we simply truncate a part of $R$ so that it becomes true.
		
		Define $P := R|_{[1,L] \times [1,k+M-1]}$.
		Since $V$ has forbidden patterns of size at most $M$, and since $R$ respects our local rules, $P$ can be vertically juxtaposed with itself (overlapping on $R^\prime$).
		
		$P$ can also be horizontally juxtaposed with itself (without overlap). Indeed, one line of $P$ uses only elements of one SCC of $H$ (since elements of two different SCCs cannot be juxtaposed horizontally, and we banned transient elements). Since $L$ is a multiple of the length of all cycle classes, the first element in a given line can follow the last element in the same line. Hence all lines of $P$ can be juxtaposed with themselves.
		
		As a conclusion, $P$ is a valid patch that can tile $\Z^2$ periodically. Therefore, $X_{H,V}$ is nonempty; return \texttt{YES}.
		
		\item If $\mathcal{G}(H)$ has symmetric type SCCs the construction is similar, but this time build a rectangle $R$ of size $2 \times M(|\A|^{2M}+1)$. Either we cannot find one and return \texttt{NO}; or we can find one and from it extract a patch that tiles the plane periodically and return \texttt{YES}.
		
		\item Finally, if $\mathcal{G}(H)$ has reflexive type SCCs, the construction is even simpler than before. Build a rectangle $R$ of size $1 \times M(|\A|^{M}+1)$; the rest of the reasoning is identical.
	\end{itemize}
	Proof of $\Rightarrow$ is due to \cref{th:root}, and is done by contraposition. If $\mathcal{G}(H)$ does not verify condition $D$, then for any Wang shift $W$ we can algorithmically build some one-dimensional SFT $V_W$ such that $X_{H,V_W}$ is a root of $W$, see \cref{th:root}. If we were able to solve $DP_I(H)$, then there would exist a Turing Machine $\mathcal{M}$ able to tell us if $X_{H,V}$ is empty for any one-dimensional SFT $V$. But as a consequence we could build a Turing Machine $\mathcal{N}$ taking as input any Wang shift $W$, and building the corresponding $V_W$ following \cref{sec:simulation}. Then, by running $\mathcal{M}$, $\mathcal{N}$ would be able to tell us if $X_{H,V_W}$ is empty or not. Then it could answer if $W$ is empty or not; but determining the emptiness or nonemptiness of every Wang shift is equivalent to $DP(\Z^2)$ being decidable, which is false. Hence, since $DP(\Z^2)$ is undecidable, $DP_I(H)$ is too.
\end{proof}

\begin{remark}
	A pair of conjugate SFTs $H_1$ and $H_2$ may yield different results, with $DP_{I}(H_1)$ decidable but $DP_{I}(H_2)$ undecidable. Consider for instance the following Rauzy graphs and applications on finite words (extensible to biinfinite words):
	
	\begin{minipage}{0.55\textwidth}
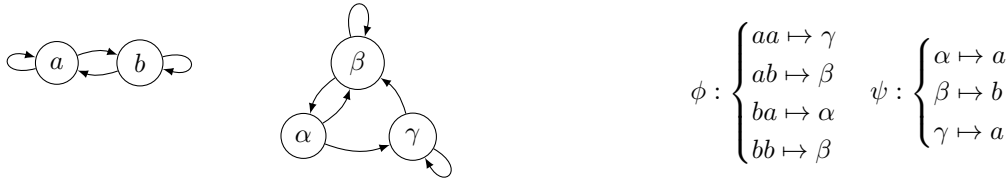
		\begin{tikzpicture}
		\begin{scope}[xshift=-4cm]
		\node[draw, circle] (a) {$a$};
		\node[draw, circle] (b) [right = 0.5cm of a] {$b$};
		
		\draw[->, >=latex] (b) to[bend left=20] (a);
		\draw[->, >=latex] (a) to[bend left=20] (b);
		
		\draw[->, >=latex] (b) to [loop right] (b);
		\draw[->, >=latex] (a) to [loop left] (a);
		\end{scope}
		
		\begin{scope}
		\node[draw, circle] (b) {$\beta$};
		\node[draw, circle] (a) [below left = 0.5cm and 0.25cm of b] {$\alpha$};
		\node[draw, circle] (c) [below right = 0.5cm and 0.25cm of b] {$\gamma$};
		
		\draw[->, >=latex] (a) to[bend right=20] (b);
		
		\draw[->, >=latex] (c) to[bend right=20] (b);
		\draw[->, >=latex] (b) to[bend right=20] (a);
		\draw[->, >=latex] (a) to[bend right=20] (c);
		
		\draw[->, >=latex] (b) to[loop above] (b);
		\draw[->, >=latex] (c) to [out=330,in=300,looseness=8] (c);
		\end{scope}
		\end{tikzpicture}
	\end{minipage}
	\begin{minipage}{0.35\textwidth}
		$\phi:
		\begin{cases}
		aa\mapsto \gamma \\
		ab\mapsto \beta \\
		ba\mapsto \alpha \\
		bb\mapsto \beta \\
		\end{cases}
		\psi:
		\begin{cases}
		\alpha\mapsto a \\
		\beta\mapsto b \\
		\gamma\mapsto a \\
		\end{cases}$
	\end{minipage}
	
	These graphs describe conjugate SFTs through these applications. However, the first graph has decidable $DP_I(H)$ and the second has not, by use of \cref{th:DP}.
\end{remark}

\section{Impact of the interplay between horizontal and vertical conditions on the algorithmic complexity of the entropy}
\label{EntropyCombined}

\subsection{Horizontal constraints without condition D}
\label{nocondD}

From the construction in \cref{sec:simulation} we deduce the following:

\begin{proposition}
	\label{prop:XHV}
	Let $H$ be a one-dimensional nearest-neighbor SFT that does not satisfy condition D. Then there exists a one-dimensional SFT $V$  such that $h(X_{H,V})$ is not computable. \end{proposition}

\begin{proof}
Let $W$ be a Wang shift with a non-computable entropy.  By using \cref{th:root} and \cref{rootentropy}, there exists  a one-dimensional SFT $V$ such that $h(W)= KM^2N h(X_{H,V})$ with $N$ the number of elements in the alphabet of $W$, $K$ and $M$ being defined as in \cref{subsec:core} which depend only of $H$. Therefore, $h(X_{H,V})$ is not computable.
\end{proof}

\begin{remark}
	For a given $H$, the $V$-dependent entropies in general are still not characterized and it seems difficult since in the construction of the root, the entropy decreases if the number of Wang tiles increases. In fact the previous proof allows to expect that in the case where the condition D is not satisfied, there exists a constant $C_H$ such that for every $\Pi_1$-computable number $h$ smaller than $C_H$ there exists a vertical one-dimensional subshift $V$ that allows $h(X_{H,V})=h$. However, we do not obtain this result exactly since if the Kolmogorov complexity of $h$ is important, then the cardinal of the alphabet of the Wang subshift that has $h$ as entropy, named $N$ in the previous proof, is also important.
	
	Nevertheless, in the case where there are cycles in the Rauzy graph which defines $H$ that do not appear in the coding of the Wang subshift used in \cref{th:root}, we can encode $h$ in a given part of $X_{H,V}$, diluted by the necessarily large macro-slices, but then add a noisy zone where these cycles can be used to increase the entropy, as it is done in the proof of	 \cref{th:RealizationEntropy}.
\end{remark}

\subsection{Condition D, computable entropy}
\label{condDcomputable}

A one-dimensional SFT $H$ that verifies condition $D$ can yield computable entropies. Indeed, one has the immediate result that allows only a small range of available entropies:

\begin{proposition}
	The entropies $h(X_{\A^\Z,V})$ accessible for $V \subset \A^\Z$ SFT are all the entropies accessible for one-dimensional SFTs with alphabet $\A$. These are included in the values $\log_2(\lambda) \leq \log_2(|\A|)$, where $\lambda$ is a Perron number. Notably, they all are computable.
\end{proposition}

\begin{proof}
	We use the fact that $N_{X_{\A^\Z,V}}(n,n) = N_V(n)^n$ since any two $n$-long columns can be juxtaposed horizontally here.
	We conclude using \cite{LindMarcus} that states that the available entropies for one-dimensional SFTs are the $\log_2(\lambda)$ where $\lambda$ is a Perron number.
\end{proof}

\begin{remark}
	An open question remains: what exactly are these accessible $\log_2(\lambda)$ obtained for a fixed size of alphabet?
	\cref{th:bonus} -- and \cite{HM} -- gives an answer for dimension $2$, but this exact question is, to our knowledge, not answered in dimension $1$.
\end{remark}

A similar result holds for a larger class of graphs, one of the possibilities for respecting condition D:

\begin{proposition}
	Let $H \subset \A^\Z$ be a nearest-neighbor SFT whose Rauzy graph is so that each of its SCCs is a state-split cycle. Then for all $V \subset \A^\Z$ SFT, $h(X_{H,V})$ is computable.
\end{proposition}

\begin{proof}
	Consider that $\mathcal{G}(H) = \bigsqcup_{i=0}^{p-1} U_i$ is a state-split cycle of length $p > 0$ -- the proof for a graph made of several state-split cycles is similar and briefly mentioned at the end. We prove that for any $V \subseteq \A^\Z$, $h(X_{H,V})$ is computable. Let $V$ be such a SFT.
	
	Let $\phi\colon \A^\Z \rightarrow \{0,\dots,p-1\}^\Z$ be the factor that maps any symbol in the component $U_i$, on $i$. We also call $\phi$ its restricted version to $\A^*$.
	
	For $n \in \N$, $j \in \{0,\dots,p-1\}$ and $u \in \{0,\dots,p-1\}^n$, let $u^j = u + (j,\dots,j)$ with the addition made modulo $p$. We also define $S_u = \{ w \in \mathcal{L}_V(n) \mid \phi(w)=u \}$ and $N_u = |S_u|$.
	
	We have the following, for integers $m$ and $n$:
	
	\[
	N_{X_{H,V}} (pm,n) = \sum_{u \in \{0,\dots,p-1\}^n} \left( \Pi_{j=0}^{p-1} N_{u^j} \right)^m.
	\]
	
	Indeed, a rectangle of size $pm \times n$ in $X_{H,V}$ is given by a vertical word $u$ in $\{0,\dots,p-1\}^n$ that fixes for the whole rectangle where elements of each $U_i$ will be. Then column $j$ can be made of any succession of $n$ symbols that respects $u^j$ -- that are in the correct $U_i$'s.
	
	Therefore, we have:
	\begin{align*}
		h(X_{H,V}) & = \lim\limits_{n \to +\infty} \dfrac{\log_2\left(N_{X_{H,V}}(pn,n)\right)}{pn^2}\\
		& = \lim\limits_{n \to +\infty} \dfrac{\log_2\left(\sum_{u \in \{0,\dots,p-1\}^n} \left( \Pi_{j=0}^{p-1} N_{u^j} \right)^n\right)}{pn^2}\\
		& = \lim\limits_{n \to +\infty} \dfrac{\log_2\left( \left(\max\limits_{v \in \{0,\dots,p-1\}^n} \Pi_{j=0}^{p-1} N_{v^j}\right)^n \sum_{u \in \{0,\dots,p-1\}^n} \dfrac{\left( \Pi_{j=0}^{p-1} N_{u^j} \right)^n}{\left(\max\limits_{v \in \{0,\dots,p-1\}^n} \Pi_{j=0}^{p-1} N_{v^j}\right)^n} \right)}{pn^2}\\
		& = \lim\limits_{n \to +\infty} \dfrac{\log_2\left( \left(\max\limits_{v \in \{0,\dots,p-1\}^n} \Pi_{j=0}^{p-1} N_{v^j}\right)^n \right)}{pn^2} + \lim\limits_{n \to +\infty} \dfrac{\log_2\left( \sum_{u \in \{0,\dots,p-1\}^n} \dfrac{\left( \Pi_{j=0}^{p-1} N_{u^j} \right)^n}{\left(\max\limits_{v \in \{0,\dots,p-1\}^n} \Pi_{j=0}^{p-1} N_{v^j}\right)^n} \right)}{pn^2}\\
		& = \lim\limits_{n \to +\infty} \dfrac{\log_2\left( \max\limits_{v \in \{0,\dots,p-1\}^n} \Pi_{j=0}^{p-1} N_{v^j} \right)}{pn}
	\end{align*}
	since in the penultimate line the second term can be bounded by $0$ from below (at least one $u \in \{0,\dots,p-1\}^n$ in the index of the sum reaches the maximum of the denominator, so the sum is at least $1$), and by $\frac{\log_2(p^n)}{pn^2}$, which tends to $0$.
	
	With the previous computation, it is also clear that for all $n \in \N$,
	
	\[
	\dfrac{\log_2\left(N_{X_{H,V}}(pn,n)\right)}{pn^2} \geq \dfrac{\log_2\left( \max\limits_{v \in \{0,\dots,p-1\}^n} \Pi_{j=0}^{p-1} N_{v^j} \right)}{pn}
	\]
	
	hence
	
	\[
	h(X_{H,V}) \geq \lim\limits_{n \to +\infty} \dfrac{\log_2\left( \max\limits_{v \in \{0,\dots,p-1\}^n} \Pi_{j=0}^{p-1} N_{v^j} \right)}{pn}
	\]
	
	From this we can deduce that the sequence $\dfrac{\log_2\left( \max\limits_{v \in \{0,\dots,p-1\}^n} \Pi_{j=0}^{p-1} N_{v^j} \right)}{pn}$ actually converges to $h(X_{H,V})$ from below. There is a Turing Machine, constructed algorithmically with $V$, that computes any term of it. Indeed, for any $v \in \{0,\dots,p-1\}^n$, $N_v$ can be computed because it depends on a one-dimensional SFT. Then the max on a finite set and all the other operations are also doable. As a consequence, $h(X_{H,V})$ is left-recursively enumerable. Being, by \cite{HM}, right-recursively enumerable, it is computable.
	
	The proof is similar if $\mathcal{G}(H)$ is made of, say, $k>1$ state-split cycles $C_j$ numbered from $1$ to $k$: consider $p$ the LCM of their periods; $\phi$ projects words of $\A^n$ on $(\{1,\dots,k\} \times \{0,\dots,p-1\})^n$ that indicates for each letter to which $C_j$ it belongs, then to which $U^j_i$ of that $C_j$. Once again, one vertical word of length $n$ is enough to reconstruct, for any $m \in \N$, a $pm \times n$ rectangle except for the precise choice of an element in each $U^j_i$. The rest of the computation is similar.
\end{proof}

\subsection{Condition D, uncomputable entropy}
\label{condDuncomputable}


Other horizontal constraints that verify condition $D$, such as some of reflexive type, allow for a greater range of accessible entropies. In what follows, we briefly investigate transformations from graphs of one-dimensional SFTs that originally \emph{do not} verify condition $D$ where the resulting graph does verify condition $D$, yet is robust enough so that the generic construction in \cref{subsec:core} can still be mostly performed: it keeps the capacity to obtain, up to a multiplicative factor, right-recursively enumerable entropies by choosing adequate vertical constraints.

\begin{proposition}
	\label{prop:loops}
	Let $H \subset \A^\Z$ be a nearest-neighbor SFT whose Rauzy graph does not respect condition $D$, and so that it has either no almost-attractive or attractive vertex, or no almost-repulsive or repulsive vertex -- where almost means that the only missing edge is a loop on the designated vertex.
	
	Let $\tilde{H} \subset \A^\Z$ be the nearest-neighbor SFT obtained when adding a loop to every vertex of the Rauzy graph of $H$. Then $\tilde{H}$ verifies condition $D$, but there exists a subshift of finite type $V$, which can be obtained algorithmically with $H$ as input, such that $h(X_{\tilde{H},V})$ is not computable.
\end{proposition}

\begin{proof}[Proof (Sketched)]
	First, one can consider that the Rauzy graph of $H$ has a loop on every vertex except one: such a graph also yields $\tilde{H}$ when a loop is added to its last loopless vertex, resulting in the same possible constructions -- and entropies -- with $\tilde{H}$ as horizontal constraints.
	
	When trying to apply the generic construction from \cref{subsec:core} on such a $H$, one ends up in case 1.1, where the only risk in the construction is the presence of both attractive or repulsive vertices. This is prevented here by the hypotheses of the theorem. Therefore by use of \cref{th:root}, for any two-dimensional SFT $Y \subset \mathcal{B}^{\Z^2}$, there exists a one-dimensional SFT $V_Y\subset\A^\Z$ such that $X_{H,V_Y}$ is a $(m,n)$th root of $Y$ for some $m, n \in \N^2$. Furthermore, $m$, $n$ and $V_Y$ can be computed algorithmically.
	
	However, the construction of $X_{H,V_Y}$ has to be slightly adapted, because if one were to follow the construction of \cref{sec:simulation} as is,  $X_{\tilde{H}, V_Y}$ could glue an $(i,j)$ macro-slice to the following types of macro-slices:
	\begin{enumerate}
		\item an $(i+1,j+1)$ macro-slice;
		\item another $(i,j)$ macro-slice;
		\item an $(i+1,j+1)$ macro-slice but the second slice is exactly one cell down;
		\item another $(i,j)$ macro-slice but the second slice is exactly one cell up;
		\item an $(i+1,j)$ macro-slice;
		\item an $(i,j+1)$ macro-slice;
		\item an $(i+1,j)$ macro-slice but the second slice is exactly one cell down;
		\item an $(i,j+1)$ macro-slice but the second slice is exactly one cell up.
	\end{enumerate}
	Case 1 is the one supposed to happen. Case 2 will be acceptable because it does not bring entropy, as seen below. Cases 5, 6, 7 and 8 collapse on other cases because $H$ has a $C^2$ made of a unique element in its generic construction, causing the $j$ index to be irrelevant.
	Cases 3 and 4 are the ones we have to deal with. In short, this is done with two things:
	\begin{itemize}
		\item two extra symbols at the bottom of any $(i,j)$-coding micro-slice, forced to be $c^1_i$;
		\item and forcing an $(i+1,j+1)$ macro-slice below any $(i,j)$ macro-slice, instead of having each column based on a single pair $(i,j)$.
	\end{itemize}
	These modifications forbid one-cell shifts between two columns, else the $C^1$ elements at the bottom of each coding micro-slice would be in contact with another element of $C^1$ at distance $2$, resulting in a uniform shortcut of length $2$ in the Rauzy graph of $H$, which is forbidden.
	In what follows, we suppose we modify $V_Y$ with all this, so that we obtain a new $X_{H,V_Y}$ and only Cases 1 and 2 of the previous enumeration happen for the corresponding $X_{\tilde{H},V_Y}$.
	
	When performing this adapted construction, only the dividing constants are modified in \cref{prop:XHV} (because the size of the macro-slices is modified). The result itself is unchanged: notably, it still shows that a non-computable entropy $h(Y)$ yields a non-computable entropy $h(X_{H,V_Y})$. The only part left is to prove that $h(X_{\tilde{H},V_Y}) = h(X_{H,V_Y})$. 
	
	Since the only difference between $X_{\tilde{H},V_Y}$ and $X_{H,V_Y}$ is the possibility to repeat a column several times in a row, we have
	\[
	N_{X_{\tilde{H},V_Y}}(n,n) \geq N_{X_{H,V_Y}}(n,n)
	\]
	but also
	\[
	N_{X_{\tilde{H},V_Y}}(n,n) = \sum_{k=1}^{n} \sum_{i_\ell | i_1 + \dots + i_k = n} N_{X_{H,V_Y}}(i_\ell,n) \leq \sum_{k=1}^n \binom{n+k-1}{n} N_{X_{H,V_Y}}(n,n) = \binom{2n}{n+1} N_{X_{H,V_Y}}(n,n).
	\]
	by counting, for $n$ columns, how many types of them there are. $\log_2(\binom{2n}{n+1})$ is a $\mathcal{O}(2n\log_2(2n))$, and therefore applying $\lim_n \dfrac{\log_2(.)}{n^2}$ to these bounds shows that the entropy is the same.
\end{proof}



\end{document}